\documentclass{amsart}

\newtheorem{theorem}{Theorem}[section]

\usepackage{amsmath, amsthm, amssymb, graphicx, enumitem, color}

\newcounter{citedtheorems}

\newtheorem{defn}{Definition}[section]
\newtheorem{theorem-m1}[defn]{Main Theorem}
\newtheorem{theorem-c1}[defn]{Central Theorem}

\newtheorem*{theorem-x}{Theorem}
\newtheorem*{theorem-m}{Main Theorem}
\newtheorem*{theorem-n}{Main Theorem}
\newtheorem*{cor-x}{Corollary}
\newtheorem*{lemma-x}{Lemma}
\newtheorem*{concl-x}{Conclusion}
\newtheorem*{claim-x}{Claim}
\newtheorem*{thm-r}{Theorem \ref{concl:sop2-max}}
\newtheorem*{thm-q}{Theorem \ref{theorem:p-t}}
\newtheorem*{claim-s}{Claim \ref{m1-sat}}

\newtheorem{prob}{Problem}
\newtheorem{red}[defn]{Reduction}
\newtheorem{thm-lit}[citedtheorems]{Theorem}
\newtheorem{defn-lit}[citedtheorems]{Definition}
\newtheorem{fact}[defn]{Fact}
\newtheorem{cor}[defn]{Corollary}

\newtheorem{concl}[defn]{Conclusion}
\newtheorem{conv}[defn]{Convention}
\newtheorem{conv-r}[defn]{Conventions and Remarks}
\newtheorem{claim}[defn]{Claim}

\newtheorem{lemma}[defn]{Lemma}
\newtheorem{obs}[defn]{Observation}
\newtheorem{rmk}[defn]{Remark}

\newtheorem{disc}[defn]{Discussion}
\newtheorem*{scl}{Subclaim}

\newtheorem{conj}[defn]{Conjecture}
\newtheorem{qst}[defn]{Question}

\newcommand{\mb}{\mathbf{b}}
\newcommand{\lost}{\L os' }
\newcommand{\los}{\L os }
\newcommand{\br}{\vspace{2.5mm}}
\newcommand{\bbr}{\vspace{3mm}}

\newcommand{\rstr}{\upharpoonright}
\newcommand{\mci}{\mathcal{I}}

\newcommand{\cis}{\operatorname{cis}}
\newcommand{\hx}{\hspace{3.5mm}}
\newcommand{\hz}{\hspace{1mm}}

\newcommand{\kleq}{\trianglelefteq}

\newcommand{\ml}{\mathcal{L}}
\newcommand{\mch}{\mathcal{H}}
\newcommand{\tlf}{\trianglelefteq}

\newcommand{\step}{\vspace{3mm}\noindent\emph}
\newcommand{\lp}{\emph{(}}
\newcommand{\rp}{\emph{)}}
\newcommand{\mct}{\mathcal{T}}
\newcommand{\ord}{\operatorname{Or}}
\newcommand{\tr}{\operatorname{Tr}}
\newcommand{\cts}{\mc^{\mathrm{ct}}(\cs)}
\newcommand{\fun}{\operatorname{fun}}

\newcommand{\de}{\mathcal{D}}

\newcommand{\eff}{\mathcal{F}}

\newcommand{\fss}{{\mathcal{P}}_{\aleph_0}}
\newcommand{\Los}{\L o\'s }

\newcommand{\uu}{{W}}

\newcommand{\trg}{{T_{rg}}}
\newcommand{\tfeq}{{T_{feq}}}
\newcommand{\rn}{\operatorname{range}}

\newcommand{\jj}{\mathbf{j}}

\newcommand{\mc}{\mathcal{C}}

\newcommand{\pr}{\operatorname{Pr}}

\newcommand{\vp}{\varphi}
\newcommand{\lcf}{\operatorname{lcf}}
\newcommand{\cf}{\operatorname{cf}}

\newcommand{\bn}{\mathbf{n}}

\newcommand{\mcp}{\mathcal{P}}
\newcommand{\ma}{\mathbf{a}}

\newcommand{\dom}{\operatorname{dom}}

\newcommand{\ts}{T_{SOP_2}}

\newcommand{\bij}{\operatorname{Par}} 
\newcommand{\maxdom}{\max (\dom}

\newcommand{\xp}{\mathfrak{p}}
\newcommand{\xt}{\mathfrak{t}}
\newcommand{\xb}{\mathfrak{b}}
\newcommand{\xh}{\mathfrak{h}}
\newcommand{\mg}{\mathbf{G}}
\newcommand{\vv}{\mathbf{V}}
\newcommand{\cn}{\mathcal{N}}

\newcommand{\bq}{\mathbf{Q}}
\newcommand{\bc}{\mathbf{c}}
\newcommand{\bd}{\mathbf{d}}
\newcommand{\lls}{<^*}
\newcommand{\cs}{\mathbf{s}}
\newcommand{\tc}{\mathbf{c}}
\newcommand{\lgn}{\operatorname{lg}}
\newcommand{\xr}{\operatorname{val}}
\newcommand{\nt}{\operatorname{Int}}

\newcount\skewfactor
\def\mathunderaccent#1#2 {\let\theaccent#1\skewfactor#2
\mathpalette\putaccentunder}
\def\putaccentunder#1#2{\oalign{$#1#2$\crcr\hidewidth
\vbox to.2ex{\hbox{$#1\skew\skewfactor\theaccent{}$}\vss}\hidewidth}}
\def\name{\mathunderaccent\tilde-3 }

\theoremstyle{definition}

\theoremstyle{remark}

\numberwithin{equation}{section}

\begin{document}

\title[Cofinality spectrum theorems]{Cofinality spectrum theorems in model theory, set theory and general topology}


\author{M. Malliaris and S. Shelah}\thanks{Malliaris was partially supported by NSF DMS-1001666,  
Shelah's NSF grant DMS-1101597, a Godel research fellowship, and a Sloan fellowship. 
Shelah was partially supported by the Israel Science Foundation grants 710/07 and 1053/11.
This is paper 998 in Shelah's list of publications.}

\address{Department of Mathematics, University of Chicago, 5734 S. University Avenue, Chicago, IL 60637, USA}
\email{mem@math.uchicago.edu}

\author{S. Shelah}
\address{Einstein Institute of Mathematics, Edmond J. Safra Campus, Givat Ram, The Hebrew
University of Jerusalem, Jerusalem, 91904, Israel, and Department of Mathematics,
Hill Center - Busch Campus, Rutgers, The State University of New Jersey, 110
Frelinghuysen Road, Piscataway, NJ 08854-8019 USA}
\curraddr{}
\email{shelah@math.huji.ac.il}

\subjclass[2010]{Primary 03C20, 03C45, 03E17}


\dedicatory{}

\begin{abstract}
We connect and solve two longstanding open problems in quite different areas: the model-theoretic question of 
whether $SOP_2$ is maximal in Keisler's order, 
and the question from general topology/set theory of whether 
$\xp = \xt$, the oldest problem on cardinal invariants of the continuum. 
We do so by showing these problems can be translated into instances of 
a more fundamental problem which we state and solve completely, using model-theoretic methods. 
\end{abstract}

\maketitle


\tableofcontents
\newpage

\section{Introduction}

We connect and solve two fundamental open problems in quite different areas: 
the model-theoretic problem of the maximality of $SOP_2$ in Keisler's order, 
and the problem from general topology/set theory of whether $\xp = \xt$, as well as some natural set-theoretic
questions about cuts in regular ultrapowers of linear order. 

Let us begin with the simpler of the two problems to state: whether $\xp = \xt$.
Cantor proved in 1874 that the continuum is uncountable, i.e. $\aleph_0 < 2^{\aleph_0}$ \cite{cantor}. 
The study of cardinal invariants or characteristics of the continuum illuminates this gap by 
studying connections between cardinals measuring the continuum which arise from different perspectives: 
combinatorics, algebra, topology, and measure theory. 
Although there are many cardinal invariants and many open questions about them (see e.g. 
the surveys of van Douwen \cite{douwen}, Vaughan \cite{vaughan}, and Blass \cite{blass}), 
the problem of whether $\xp = \xt$ is the oldest and so holds a place of honor. (Moreover, usually if such an equality was 
not obviously true it was consistently false, by forcing.) 

Before reviewing the history, we give the easily stated definition:

\begin{defn} \emph{(see e.g. \cite{douwen})} \label{d:intro}
We define several properties which may hold of a family $D \subseteq [\mathbb{N}]^{\aleph_0}$, i.e. a family of 
infinite sets of natural numbers. 
Let $A \subseteq^* B$ mean that $\{ x : x \in A, ~x \notin B \}$ is finite.

\begin{itemize}
\item $D$ has a \emph{pseudo-intersection} if there is an infinite $A \subseteq \mathbb{N}$ such that
for all $B \in D$, $A \subseteq^* B$.
\item $D$ has the s.f.i.p. $($strong finite intersection property$)$
if every nonempty finite subfamily has infinite intersection.
\item $D$ is
called a tower if it is well ordered by ${\supseteq^*}$ and has no infinite pseudo-intersection.
\end{itemize}

Then: 
\begin{align*}
\xp = & \min \{ |\eff| ~ : ~ \eff \subseteq [\mathbb{N}]^{\aleph_0}~ \mbox{has the s.f.i.p. but has no infinite pseudo-intersection} \} \\
\xt = & \min \{ |\mct| ~ : ~ \mct \subseteq [\mathbb{N}]^{\aleph_0} ~\mbox{is a tower} \} 
\end{align*}
\end{defn}

Clearly, both $\xp$ and $\xt$ are at least $\aleph_0$ and no more than $2^{\aleph_0}$. 
It is easy to see that $\xp \leq \xt$, since a tower has the s.f.i.p. 
By a 1934 theorem of Hausdorff $\aleph_1 \leq \xp$ \cite{hausdorff}. In 1948 \cite{roth-b},
Rothberger proved (in our terminology) that $\xp = \aleph_1$ implies $\xp = \xt$, which begs the 
question of whether $\xp = \xt$.

\br
\begin{prob} \label{prob-2}
Is $\xp = \xt$? 
\end{prob}

Problem \ref{prob-2} appears throughout the literature. 
Van Douwen presents six primary cardinal invariants (of a model of set theory), 
$\mathfrak{a}, \mathfrak{b}, \mathfrak{d}, \mathfrak{p}, \mathfrak{s}$, 
and $\mathfrak{t}$; he attributes $\mathfrak{b}, \mathfrak{p}, \mathfrak{t}$ to Rothberger 1939 and 1948,
$\mathfrak{d}$ to Katetov 1960, $\mathfrak{a}$ to Hechler 1972 and Solomon 1977, and $\mathfrak{s}$ to Booth 1974 
(see \cite{douwen} p. 123). Vaughan \cite{vaughan} Problem 1.1 includes 
the only inequalities about van Douwen's six cardinals which remained open in 1990; it is noted there that
``we believe (a)'' (i.e. whether $\xp < \xt$ is consistent with $ZFC$) ``is the most interesting.''  
Following Shelah's solution of Vaughan's 1.1(b) in \cite{Sh700} (showing it was independent),  
Problem \ref{prob-2} is therefore both the oldest {and} the only remaining open inequality about van Douwen's cardinals. 

There has been much work on $\xp$ and $\xt$, for example:  
Bell \cite{bell} proved that $\xp$ is the first cardinal $\mu$ for which $MA_\mu$($\sigma$-centered) fails; 
in topological language, this asserts that no separable compact Hausdorff space can be covered by 
fewer than $\xp$ nowhere dense sets. 
Szyma\'nski proved that $\xp$ is regular. 
Piotrowski and Szyma\'nski \cite{p-s} proved that $\xt \leq \operatorname{add}(\mathcal{M})$, where $\mathcal{M}$ is the 
ideal of meager sets, and $\operatorname{add}(I)$ denotes the smallest number of sets in an ideal $I$ with union not in $I$.
Shelah proved in \cite{Sh:885} that if $\xp < \xt$ then there is a so-called peculiar cut in
${^\omega \omega}$, see Section \ref{s:p-t}, Theorem \ref{t:885} below; we will leverage this result in the present work.

In \S \ref{s:p-t} of this paper, we apply the methods developed in earlier sections to answer Problem \ref{prob-2}:

\begin{theorem-x} \emph{(Theorem \ref{theorem:p-t})}
$\xp = \xt$.
\end{theorem-x}

In the context of the general framework we build (of cofinality spectrum problems), this answer is natural, but 
it is a priori very surprising. 
Given the length of time this problem had remained open, the expectation was an independence result.

\br

We now describe the second problem, which concerns a criterion for maximality in Keisler's 1967 order on theories. In this discussion, all theories are complete and countable.  
If $\de$ is an ultrafilter on $\lambda$, 
let us say that \emph{$\de$ saturates $M$} if the ultrapower $M^\lambda/\de$ is $\lambda^+$-saturated.\footnote{Saturation is a 
fullness condition. Given a model $M$ 
and $A \subseteq \dom(M)$, the $n$-types over $A$ are the maximal consistent sets of formulas in $n$ free variables with parameters from $A$. 
A type $p(\bar{x})$ is realized in $M$ if there is $\bar{a}$ such that $M \models \vp(\bar{a})$ for each $\vp \in p$, otherwise it is omitted.  
A model is called $\kappa$-saturated if all types, or equivalently all $1$-types, over all subsets of $M$ of size $<\kappa$ are realized.}
When $\de$ is a regular ultrafilter on $\lambda$ 
and $M, N$ are both models of the same complete, countable first-order theory $T$, 
then $\de$ saturates $M$ iff $\de$ saturates $N$, so we may simply say that $\de$ saturates $T$. 

Keisler's order is the relation: 
$T_1 \tlf T_2$  ~ iff ~ for any infinite cardinal $\lambda$ and any regular ultrafilter $\de$ on $\lambda$, if $\de$ saturates $T_2$ then 
$\de$ saturates $T_1$.  
Keisler's order is a pre-order on theories, and a partial order on the equivalence classes. It was recognized early on that 
this order gave a powerful way of comparing the complexity of any two theories, even in different languages. 
Determining the structure of this order is a large-scale, and largely open, 
problem in model theory.   
It is not known, for instance, whether the order is linear, or whether it is absolute, although it has long been 
thought to have finitely many classes and be linearly ordered. 
Its structure on the stable theories is known \cite{Sh:a}, and the first dividing line in Keisler's order 
among the unstable theories has just been discovered \cite{MiSh:999}. 

It has been known since 1967 that the order has a maximum class, whose model-theoretic identity has remained elusive. Around 1963, Keisler proved   
the existence of a strong family of ultrafilters, the \emph{good} regular ultrafilters: see \cite{keisler-1}, or Definition \ref{good-filters} below.\footnote{Keisler's proof 
assumed GCH, which was later eliminated by Kunen \cite{kunen}.} In defining his order, Keisler connected these filters to model theory 
by showing that on one hand, good regular ultrafilters saturate any theory, and on the other, there exist theories complex enough to 
be able to code any failure of goodness as an omitted type in the ultrapower.  In other words, there exist $T$ such that $\de$ saturates $T$ iff $\de$ is good. 
This shows the existence of a maximum class in Keisler's order, which can be characterized set-theoretically as the 
set of complete countable first-order theories which are saturated only by those regular $\de$ which are good. 

A first surprise was that complexity in the sense of coding is not an essential model-theoretic feature of theories in the maximum Keisler class: in 1978, Shelah 
\cite{Sh:a} proved that any theory of linear order, or more precisely with the strict order property, belongs to the maximum class. 
Later, this was weakened to the strong order property $SOP_3$, which retains many features of linear order. For a long time, it was not clear whether to expect 
that any theories \emph{without} the obvious features of linear order might belong to the maximum class. Attention focused on  
$SOP_2$, a kind of maximally inconsistent tree: we say that a theory $T$ has $SOP_2$ if for some formula $\vp(\bar{x},\bar{y})$, we may, in some sufficiently 
saturated model of $T$, find a set of parameters for $\vp$ indexed by the nodes of a binary tree such that the set of instances of $\vp$ along any 
path through the tree is consistent, whereas any two instances assigned to incomparable nodes are inconsistent. This occurs easily in models of linear order, 
but also in the generic triangle-free graph. (These properties are discussed further in \S \ref{s:defns}.) The second major problem this paper addresses is:

\begin{prob} \label{prob-e}
Does $SOP_2$ imply maximality in Keisler's order?
\end{prob}

Problem \ref{prob-e} does not settle the identity of the maximum class; why, then, is it so significant? For one, it moves the possible boundary of the maximum class onto what we 
believe to be a dividing line, in the sense of classification theory. Moreover, 
we believe: 

\begin{conj} \label{conj:a}
$SOP_2$ characterizes maximality in Keisler's order.
\end{conj}

Evidence for this conjecture is discussed in \ref{c:evidence} below. 
In \S \ref{s:sop2} below we give a positive answer to Problem \ref{prob-e}: 

\begin{theorem-x} \emph{(Theorem \ref{concl:sop2-max})}
Let $T$ be a theory with $SOP_2$. Then $T$ is maximal in Keisler's order.
\end{theorem-x}

\noindent To prove Theorem \ref{concl:sop2-max} we develop a general theory of 
cofinality spectrum problems, a main contribution of the paper (which we will return to presently). 
Applying a theorem of Malliaris \cite{mm4}-\cite{mm5}, we then show: 

\begin{theorem-x} \emph{(Theorem \ref{t:tfeq})}
There is a minimum class among the non-simple theories in Keisler's order, which
contains the theory $\tfeq$ of a parametrized family of independent equivalence relations.
\end{theorem-x}

\br In order to state the paper's remaining central theorem, 
we now outline some features of our approach which connects and solves both Problems \ref{prob-2} and \ref{prob-e} above.
In this discussion, we use $\de$ to denote a regular ultrafilter on $\lambda$. 
 
By \emph{cut} we will mean an unfilled pre-cut (in set-theoretic literature cuts 
in reduced powers are often referred to as \emph{gaps}).  
Define the cut spectrum of the ultrafilter $\de$ to be:
\[ \mc(\de) = \{ (\kappa_1, \kappa_2) ~:~ \kappa_1, \kappa_2 \mbox{ regular}, ~\kappa_1 + \kappa_2 \leq \lambda, ~ \mbox{} 
(\mathbb{N}, <)^\lambda/\de \mbox{ has a $(\kappa_1, \kappa_2)$-cut } \}. \]
Since any theory of linear order is in the maximal Keisler class, we have that:  

\begin{fact} \label{max-good}
Let $\de$ be a regular ultrafilter on $\lambda$. Then $\de$ is $\lambda^+$-good iff $\mc(\de) = \emptyset$.
\end{fact}

By a tree $(\mct, \tlf)$ we will mean a set $\mct$ of sequences partially ordered by initial segment $\tlf$ such that the 
set of predecessors of any element of $\mct$ is well-ordered. 
Say that $\de$ has $\kappa$-treetops if for any tree $(\mct, \tlf)$ and any infinite regular cardinal $\gamma < \kappa$, 
any strictly increasing $\gamma$-indexed sequence in $(\mct, \tlf)^\lambda/\de$ has an upper bound.  
In \S \ref{s:sop2} below, we prove that if $\de$ saturates some theory with $SOP_2$ then $\de$ has $\lambda^+$-treetops. 
Then a positive answer to Problem \ref{prob-e} would follow from a positive answer to:

\begin{prob} \label{prob-a}
Suppose $\de$ has $\lambda^+$-treetops. Is $\mc(\de) = \emptyset$? 
\end{prob}

This is because a positive answer to Problem \ref{prob-a} would establish that any ultrafilter able to saturate some theory with $SOP_2$ is necessarily good.  

To further compare $\mc(\de)$ with the treetops of $\de$, let us define $\xp_\de$ to be the least $\kappa$ such that there exist $\kappa_1, \kappa_2$ with 
$\kappa_1 + \kappa_2 = \kappa$ and $(\kappa_1, \kappa_2) \in \mc(\de)$. Likewise, define $\xt_\de$ to be the greatest $\kappa$ such that 
$\de$ has $\kappa$-treetops. 
Below $\xp_\de$ there are no cuts, then, and below $\xt_\de$ infinite paths through trees have upper bounds. 
The names $\xp_\de$, $\xt_\de$ are chosen to be suggestive of the cardinal invariants of \S\ref{s:p-t}, and note that Problem \ref{prob-a} 
is equivalent to asking: 

\begin{prob} \label{prob-r}
Can it happen that $\xp_\de < \xt_\de$?
\end{prob}

This analogy between Problems \ref{prob-2} and \ref{prob-e} becomes formal as follows. Preliminary investigation, 
e.g. the claim in Section \ref{s:motiv-up} below, shows that the hypothesis of treetops has nontrivial consequences for 
the possible cuts in ultrapowers of linear orders. Moreover, such proofs may be carried out using only a few essential features of ultrapowers: 
the fact that ultrapower of $(\mathbb{N}, <)$ is pseudofinite (it remains a discrete linear order in which any nonempty, bounded, 
definable subset has a first and last element), and the fact that ultrapowers commute with reducts (meaning here that we may expand a given model of linear order 
to have available certain trees of functions from the order to itself, which the proof may then manipulate). 
\S \ref{s:motiv-up} below gives further details.  

In \S \ref{s:definitions}, we define the true general context: cofinality spectrum problems, a central definition of the paper.   
These are given essentially as the data $\cs$ of an elementary pair of models $M \preceq M_1$, equipped with a suitably closed 
set of formulas $\Delta$ which define distinguished pseudofinite linear orders in $M_1$, and which admit expansion to the elementary pair $M^+ \preceq M^+_1$ where this expanded language has 
available uniform definitions for certain trees of functions from each distinguished linear order to itself (trees under the partial order given by initial segment). 
For a given cofinality spectrum problem $\cs$, let $\ord(\cs)$ denote the set of distinguished orders, and $\tr(\cs)$ the set of associated trees. 
Let $\xt_\cs$ be the minimal regular cardinal $\tau$
such that some tree in $\tr(\cs)$ has a strictly increasing $\tau$-sequence with no upper bound.
When $\lambda < \xt_\cs$, we say $\cs$ has \emph{$\lambda^+$-treetops}.  Let $\cts$ be the set of pairs of
regular cardinals $(\kappa_1, \kappa_2)$ which appear as the cofinalities of a cut in some distinguished linear order of $\cs$. 

In this context, the problem of comparing linear orders and trees has become: 

\begin{qst}[Central question] \label{q:central}
Let $\cs$ be a cofinality spectrum problem. What are the possible values of
\[ \mc(\cs, \xt_\cs) = \{ (\kappa_1, \kappa_2) : (\kappa_1, \kappa_2) \in \cts, ~ \kappa_1 + \kappa_2 < \xt_\cs \} ~ \hspace{2mm}{?}\]
\end{qst}

In particular, is $\mc(\cs, \xt_\cs) = \emptyset$? That is, for a cofinality spectrum problem $\cs$, can it happen that $\xp_\cs < \xt_\cs$? 
As the definition of cofinality spectrum problems includes that of regular ultrapowers, this generalizes Problem \ref{prob-r}.  
Thus, the maximality of $SOP_2$ in Keisler's order would follow from proving that for any cofinality spectrum problem $\cs$, 
$\mc(\cs, \xt_\cs) = \emptyset$. 

In Section \ref{s:p-t}, 
it is proved that \emph{if} $\xp < \xt$ \emph{then} there exists a cofinality spectrum problem in which $\xp_\cs \leq \xp < \xt \leq \xt_\cs$, 
where $\xp$, $\xt$ are the cardinal invariants from Definition \ref{d:intro}.  
(In fact, $\xp_\cs = \xp$ and $\xt_\cs = \xt$, though this is not used.)
This cofinality spectrum problem is built using $M = M^+ = (\mch(\omega_1), \epsilon)$ with $\mg$ a generic ultrafilter for the forcing notion $([\omega]^{\aleph_0}, \supseteq^*)$ 
and $M_1 = M^+_1$ the generic ultrapower $M^\omega/\mg$. 
By definition of $\xp_\cs$, this says that if $\xp < \xt$ there is a cofinality spectrum problem $\cs$ for which 
$\mc(\cs, \xt_\cs) \neq \emptyset$. Thus, in order to prove $\xp = \xt$, it would suffice to prove, in ZFC, that 
$\mc(\cs, \xt_\cs) = \emptyset$ for all cofinality spectrum problems $\cs$.

\br

These reductions underline the importance of Question \ref{q:central}, whose solution is the core of the paper. The main steps are as follows, and the proofs 
are primarily model-theoretic.  
In Section \ref{s:context}, we introduce the right level of abstraction. 
In Section \ref{s:uniqueness}, we prove that for any cardinal $\kappa < \min \{ \xp^+_\cs, \xt_\cs \}$ 
there is at most one $\kappa^\prime$ such that $(\kappa, \kappa^\prime) \in \mc(\cs, \xt_\cs)$. 
In other words, the lower cofinality (or coinitiality) function is well defined. 
In Section \ref{s:sat}, we prove that cofinality spectrum problems have a certain amount of local saturation; here, local types are 
types of elements in one of the distinguished linear orders. 
In Section \ref{s:towards}, we show that any cofinality spectrum problem contains a certain amount of Peano arithmetic, sufficient to carry out G\"odel coding. 
In Section \ref{s:symmetric}, we rule out symmetric cuts in $\mc(\cs, \xt_\cs)$. 
In Section \ref{s:lcf-aleph-0}, a warm-up to Section \ref{s:main-theorems}, 
we show that if the lower cofinality of $\aleph_0$ is not $\aleph_1$ then it is at least $\min \{ \xp^+_\cs, \xt_\cs \}$. 
In the key Section \ref{s:main-theorems}, we rule out \emph{all} asymmetric cuts in $\mc(\cs, \xt_\cs)$. 
In Section \ref{s:main-thm}, we give one of the paper's main theorems: for any cofinality spectrum problem $\cs$, $\mc(\cs, \xt_\cs) = \emptyset$. 

Sections \ref{s:defns}--\ref{s:min-non-simple} develop the consequences for regular ultrapowers and Keisler's order. 
Section \ref{s:p-t} addresses $\xp$ and $\xt$. 

Our work here also gives a new characterization of Keisler's good ultrafilters, Theorem \ref{maximal-x} quoted below. 
Cofinality spectrum problems include other central examples, such as Peano arithmetic, as developed further in our paper \cite{MiSh:F1361}.

\[ * \hspace{5mm} * \hspace{5mm} * \hspace{5mm} * \hspace{5mm} * \hspace{5mm} * \]

\bbr

\subsection{List of main theorems}
The main results of the paper are the following.

\begin{theorem-m} \emph{(Theorem \ref{no-cuts})}
Let ~$\cs$ be a cofinality spectrum problem.  Then $\mc(\cs, \xt_\cs) = \emptyset$.
\end{theorem-m}

\begin{theorem-x} \emph{(Theorem \ref{concl:sop2-max})}
Let $T$ be a theory with $SOP_2$. Then $T$ is maximal in Keisler's order.
\end{theorem-x}

\begin{theorem-x} \emph{(Theorem \ref{maximal-x})}
For a regular ultrafilter $\de$ on $\lambda$, the following are equivalent: 
(a) $\de$ has $\lambda^+$-treetops, (b) $\kappa \leq \lambda$ implies $(\kappa, \kappa) \notin \mc(\de)$, 
(c) $\mc(\de) = \emptyset$, (d) $\de$ is $\lambda^+$-good. 
\end{theorem-x}

\begin{theorem-x} \emph{(Theorem \ref{t:tfeq})}
There is a minimum class among the non-simple theories in Keisler's order, which
contains the theory $\tfeq$ of a parametrized family of independent equivalence relations.
\end{theorem-x}

\begin{theorem-x} \emph{(Theorem \ref{theorem:p-t})}
$\xp = \xt$.
\end{theorem-x}

Our methods give various results about the structure of regular ultrapowers.   
For instance, combining our work here with work of Malliaris \cite{mm-thesis}, \cite{mm4}, we prove that any 
ultrafilter which saturates some non-low or non-simple theory must be flexible, Conclusion \ref{concl-flex} below. 

The sections on regular ultrafilters and on cardinal invariants of the continuum may be read independently of each other, while Sections 
\ref{s:context}--\ref{s:main-thm} are prerequisites for both.

\br
In a work in preparation we intend to deal also with the following questions: Do cofinality spectrum problems 
require high theories $($like bounded Peano arithmetic, Section $\ref{s:towards}$ below$)$? Last elements? Successor? Do the theorems on the cofinality spectrum $[$i.e.
$\ref{m2}$ $($thus: $\ref{lcf-cor}$, $\ref{n-is-enough}$, $\ref{cor:sym}$, $\ref{extend}$,  
$\ref{m5}$--$\ref{m5-subclaim}$, $\ref{last-cut}$ and also $\ref{c:ceiling}$, $\ref{o:incr}$, $\ref{m2a}$, $\ref{bijection}$$)$
really need the same assumptions?  What happens with so-called $\xp_\lambda, \xt_\lambda$? 
This work in preparation will also show that the lower cofinality is well defined in a more general setting, see $\ref{disc:lcf}$ below.
We also intend to deal with characterizing the maximum class in the related order $\tlf_*$, leveraging the work of 
\cite{DzSh}, \cite{ShUs} mentioned above and the results here. 

\br

\noindent\textbf{Acknowledgments.} 
We would like to warmly thank the anonymous referee for thoughtful, extensive reports  
which significantly improved the presentation of this paper.

\section{Cofinality spectrum problems} \label{s:definitions} \label{s:context}

In this section we give some central definitions of the paper. 

\begin{conv}[Conventions on notation] \emph{} \label{c:first}
\begin{enumerate}
\item $\ell(\bar{y})$ is the length of a tuple of variables $\bar{y}$. 
This length exists in the ambient model of set theory. It is always a standard ordinal, and usually finite.  
\item {Definable means with parameters, unless otherwise stated.} 
\item We say that a linear order is pseudofinite if it is discrete and every nonempty, bounded, definable subset has a first and last element. 
\item When $T$ is a theory or $M$ a model, we follow model-theoretic convention and write $\tau(T)$, $\tau(M)$, $\tau_T$, $\tau_M$, etc. to indicate the ambient vocabulary 
or signature. 
\item Given an ultrapower $M^I/\de$, we fix in advance some lifting $M^I/\de \rightarrow M^I$. Then for any element $a$ of the ultrapower and $t \in I$,
we write $a[t]$ to denote the projection to the $t$-th coordinate of this fixed representation. 
\end{enumerate}
\end{conv}

\subsection{A motivating example} \label{s:motiv-up} \emph{ }
Let $\de$ be a regular ultrafilter on $\lambda$. Let $\mc(\de)$ and $\lambda^+$-treetops be as defined in the Introduction immediately before and after Fact \ref{max-good}. 
We begin with a proof to illustrate the effect of treetops on $\mc(\de)$, which will serve to motivate the definition of cofinality spectrum problems. 
Namely, with Problems \ref{prob-a}-\ref{prob-r} from the Introduction in mind, we rule out symmetric cuts. 

\begin{lemma} \label{c:motiv} 
Suppose $\de$ is a regular ultrafilter on $\lambda$ with $\lambda^+$-treetops, and $\kappa < \lambda^+$ is regular.  
Then $\mc(\de)$ has no $(\kappa, \kappa)$-cuts, i.e. $(\kappa, \kappa) \notin \mc(\de)$.
\end{lemma}

\begin{proof}
Let $M = (\mathbb{N}, <)$ and $M_1 = M^\lambda/\de$.  
Note that by the first fundamental theorem of ultraproducts, \lost theorem, 
any nonempty bounded definable subset of the domain of $M_1$ has a first and last, 
i.e. least and greatest, element.  

Assume for a contradiction that in $M_1$ there is $(\bar{a}, \bar{b}) = (\langle a_\alpha : \alpha < \kappa \rangle, \langle b_\alpha : \alpha < \kappa \rangle)$ 
such that for all $\beta < \alpha < \kappa$, $M_1 \models a_\beta < a_\alpha < b_\alpha < b_\beta$ but there is no $c \in M_1$ such that 
$a_\alpha < c < b_\alpha$ for all $\alpha < \kappa$. 

Let $(\mct, \tlf)$ be the tree whose elements are finite sequences of pairs of natural numbers, partially ordered by initial segment. 
Expand $M$ to a model $M^+$ in which the tree $(\mct, \tlf)$ is definable, for instance by letting $M^+ = (\mch(\omega_1), \epsilon)$ and 
identifying $\mathbb{N} = dom(M)$ with $\omega$ in $M^+$. In this expansion, what is important is that:

\begin{enumerate}[label=\emph{(\alph*)}]
\item $\mathbb{N}$ is a definable set, $\mct$ is a definable set and $\tlf$ is a definable relation. 
\item Elements of $\mct$ are functions from an initial segment of $\mathbb{N}$ into $\mathbb{N} \times \mathbb{N}$.
\item The following are definable, uniformly, for each $a \in \mct$: 
\begin{enumerate}
\item[(i)] the length function $\lgn(a)$  
\item[(ii)] the function giving $\maxdom(a))$, i.e. $\lgn(a)-1$,
\item[(iii)] for each $n \leq \maxdom(a))$, the evaluation function $a(n)$, 
\item[(iv)] for each $n \leq \maxdom(a))$, the projection functions $a(n,0)$ and $a(n,1)$, where these denote the 
two coordinates of $a(n)$. 
\end{enumerate}
\end{enumerate}

The second fundamental theorem of ultraproducts, Theorem \ref{commute-with-reducts} p. \pageref{commute-with-reducts} below, states that ultraproducts commute with reducts. 
Having chosen $M^+$, then, let $M^+_1$ be the induced expansion of $M_1$. Then e.g. elements of $\mct^{M^+_1}$ are functions from an initial segment of 
the nonstandard integers into pairs of nonstandard integers, and (a)-(c) remain definable in $M^+_1$.  
The length function $\lgn$ will then often take a nonstandard length. (Throughout the paper,  
we will use $\lgn$ to denote a possibly nonstandard length function definable in the structure under consideration. 
In contrast, we use $\ell(\bar{y})$ following Convention \ref{c:first}.)

Let $\vp(x)$ be a formula expressing: $x \in \mct$ and $n < m \leq \maxdom(x))$ implies  
$x(n,0) < x(m,0) < x(m,1) < x(n,1)$. In $M^+_1$, $\vp$ defines an infinite subtree 
of $\mct^{M^+_1}$, which we denote in the rest of the proof\footnote{Why not simply begin with the definition for $\mct_*$? The present approach generalizes more easily to later sections, 
where we will distinguish certain families of basic trees and then work, as necessary, within their definable sub-trees.} as $\mct_*$. 
Note that if $M^+_1 \models \vp(c)$ and $n \leq \maxdom(c))$, then $M^+_1 \models \vp(c \rstr_n)$. 

By induction on $\alpha < \kappa$ we now choose elements $c_\alpha$ of $\mct_*$ and $n_\alpha$ of $\mathbb{N}^{M^+_1}$  
such that:
\begin{enumerate}
\item for all $\beta < \alpha < \kappa$, $M^+_1 \models c_\beta \tlf c_\alpha$ 
\item for all $\alpha < \kappa$,  $n_\alpha = \maxdom(c_\alpha))$ 
\item for all $\alpha < \kappa$, $c_\alpha(n_\alpha, 0) = a_\alpha$, and $c_\alpha(n_\alpha, 1) = b_\alpha$. 
\end{enumerate}

For the base case, let $c_0 = \langle (a_0, b_0) \rangle$. 

When $\alpha = \beta + 1$, having defined $c_\beta$, let $c_\alpha = {c_\beta}^{\smallfrown}\langle (a_\alpha, b_\alpha) \rangle$. 
Then $c_\alpha \in \mct_*$ by definition of the sequences $\bar{a}, \bar{b}$. Let $n_\alpha = n_\beta + 1$. 

When $\alpha$ is a limit ordinal, by the hypothesis of treetops there is $c_* \in \mct_*$ such that $\beta < \alpha$ implies $c_\beta \tlf c_*$. 
Let $n_* = \maxdom(c_*))$. By definition of $\mct_*$ and of $\tlf$, it will be the case that 
$\beta < \alpha$ implies $c_\beta(n_\beta, 0) = c_*(n_\beta, 0) < c_*(n_*,0) < c_*(n_*, 1) < c_*(n_\beta, 1) = c_\beta(n_\beta, 1)$, but  
it may also be the case that $a_\alpha < c_*(n_*,0) < c_*(n_*, 1) < b_\alpha$. However, the set 
\[ \{ n \leq n_* : c_*(n,0) < a_\alpha \land b_\alpha < c_*(n,1) \} \]
is a nonempty bounded definable subset of nonstandard integers, so has a maximal element, call it $m_*$.  
Note that necessarily $c_\beta \tlf c_* \rstr_{m_*}$ for each $\beta < \alpha$. 
Now $c_\alpha = {{c_* \rstr_{m_*}}^\smallfrown \langle (a_\alpha, b_\alpha) \rangle}$ realizes $\vp$ and, letting 
also $n_\alpha = m_*$, this completes the inductive construction. 

Now the path $\bar{c} = \langle c_\alpha : \alpha < \kappa \rangle$ in $\mct_*$ corresponds, by construction, to 
a cofinal sequence in our original cut. By hypothesis of treetops, there is an upper bound $c_{\star} \in \mct_*$ so $\alpha < \kappa$ implies $c_\alpha \tlf c_{\star}$. 
Let $n_{\star} = \maxdom(c_{\star}))$. Then for each $\alpha < \kappa$, by definition of $\vp$, 
\[  a_\alpha = c_\alpha(n_\alpha, 0) = c_{\star}(n_\alpha, 0) < c_{\star}(n_{\star}, 0) < c_{\star}(n_{\star}, 1) < c_{\star}(n_\alpha, 1) = c_\alpha(n_\alpha, 1) = b_\alpha \]
realizing the cut $(\bar{a}, \bar{b})$. This contradiction shows $(\bar{a}, \bar{b})$ cannot exist. 
\end{proof}

\br

Lemma \ref{c:motiv} hints at the power of the treetops hypothesis with respect to saturation of the underlying order. It suggests 
that \emph{in regular ultrapowers}, one might productively analyze existence or nonexistence of other families of cuts in $\mc(\de)$ 
by considering progressively more sophisticated trees.
However, on closer inspection, the use of regular ultrapowers does not appear essential. 
Rather, the key features of that setting were:

\begin{enumerate}
\item \lost theorem, used to show $M \preceq M_1$, and that the set $\mathbb{N}^{M_1}$ behaved in a pseudofinite way: it remained a definable, discrete linear 
order in which any bounded, nonempty definable subset had a first and last element. 
\item Ultrapowers commute with reducts, invoked to study a given order in $M, M_1$ by uniformly (simultaneously) expanding 
both models to have available a suitable tree of sequences from the order to itself. Note that even after the expansion, 
any bounded, nonempty definable subset of $\mathbb{N}^{M^+_1}$ has a first and last element. 
\end{enumerate}

The setting of cofinality spectrum problems, a fundamental definition of this paper, can be seen as a suitable 
abstraction of this setting (i.e. studying cuts in linear orders under the hypothesis of existence of paths through 
related families of trees) which retains the key features just described. 
Although Definitions \ref{d:estt} and \ref{d:csp} have wider application, regular ultrapowers 
will be an important context for the second half of the paper.

\subsection{Main definitions}
We now give several main definitions of the paper. 

First, given a family of formulas defining linear orders, 
we specify which trees we would like to have available.\footnote{Although in Section \ref{s:motiv-up} we used that ultrapowers commute with reducts, 
giving us access to any reasonable tree,  
it is not necessary to have all or even many trees available.  
Here we ask for uniform definitions for the necessary families of trees.}
The conditions are of two kinds: on one hand, we require that $\Delta$ have a certain form and satisfy some closure
conditions, and on the other hand we require that certain sets derived from $\Delta$ are definable. 

\begin{defn} \label{d:estt}  \emph{(ESTT, enough set theory for trees)}
Let $M_1$ be a model and $\Delta$ a nonempty set of formulas in the language of $M_1$.
We say that $(M_1, \Delta)$ has \emph{enough set theory for trees} when the following conditions are true.

\begin{enumerate}
\item $\Delta$ consists of first-order formulas $\vp(\bar{x},\bar{y};\bar{z})$, with 
$\ell(\bar{x}) = \ell(\bar{y})$. 

\item For each $\vp \in \Delta$ and each parameter $\overline{c} \in {^{\ell(\overline{z})}M_1}$, 
$\vp(\bar{x},\bar{y}, \bar{c})$ defines a discrete linear order on $\{ \bar{a} : M_1 \models \vp(\bar{a}, \bar{a}, \bar{c}) \}$ with a first and last element. 

\item The family of all linear orders defined in this way will be denoted $\ord(\Delta, M_1)$. Specifically, each $\ma \in \ord(\Delta, M_1)$ 
is a tuple $(X_\ma, \leq_\ma, \vp_\ma, \overline{c}_\ma, d_\ma)$, where:
\begin{enumerate}
\item $X_\ma$ denotes the underlying set $\{ \bar{a} : M_1 \models \vp_\ma(\bar{a}, \bar{a}, \bar{c}_\ma) \}$
\item $\bar{x} \leq_\ma \bar{y}$ abbreviates the formula $\vp_\ma(\bar{x},\bar{y},\bar{c}_\ma)$ 
\item $d_\ma \in X_\ma$ is a bound for the length of elements in the associated tree; it is often, but not always, $\max X_\ma$. 
If $d_\ma$ is finite, we call $\ma$ trivial. 
\end{enumerate}

\item For each $\ma \in \ord(\ma)$, $(X_\ma, \leq_\ma)$ is pseudofinite, meaning that any bounded, nonempty, $M_1$-definable subset has a $\leq_\ma$-greatest 
and $\leq_\ma$-least element. 

\item For each pair $\ma$ and $\mb$ in $\ord(\Delta, M_1)$, there is $\bc \in \ord(\Delta, M_1)$ such that: 
\begin{enumerate} 
\item there exists an $M_1$-definable bijection $\pr: X_\ma \times X_\mb \rightarrow X_\bc$ such that the coordinate projections are $M_1$-definable. 
\item if $d_\ma$ is not finite in $X_\ma$ and $d_\mb$ is not finite in $X_\mb$, then also $d_\bc$ is not finite in $X_\bc$. 
\end{enumerate} 

\item For some nontrivial $\ma \in \ord(\Delta, M_1)$, there is $\bc \in \ord(\Delta, M_1)$ such that
$X_\bc = \pr(X_\ma \times X_\ma)$ and the ordering $\leq_\bc$ satisfies: 
\[ M_1 \models (\forall x \in X_\ma)(\exists y \in X_\bc)(\forall x_1, x_2 \in X_\ma)(\max \{ x_1, x_2 \}  \leq_\ma x \iff \pr(x_1, x_2) \leq_\bc y) \]

\item To the family of distinguished orders, we associate a family of trees, as follows. 
For each formula $\vp(\bar{x},\bar{y}; \bar{z})$ in $\Delta$ there are formulas $\psi_0, \psi_1, \psi_2$ of the language of $M_1$ such that 
for any $\ma \in \ord(\cs)$ with $\vp_\ma = \vp$:

\begin{enumerate}
\item $\psi_0(\bar{x}, \overline{c}_\ma)$ defines a set, denoted $\mct_\ma$, of functions from $X_\ma$ to $X_\ma$. 
\item $\psi_1(\bar{x}, \bar{y}, \overline{c})$ defines a function $\lgn_\ma : \mct_\ma \rightarrow X_\ma$ satisfying:
	\begin{enumerate}
	\item for all $b \in \mct_\ma$, $\lgn_\ma(b) \leq_\ma d_\ma$. 
	\item for all $b \in \mct_\ma$, $\lgn_\ma(b) = \maxdom(b))$.  
\end{enumerate}
\item $\psi_2(\bar{x},\bar{y},\overline{c})$ defines a function from 
$\{ (b,a) : b \in \mct_\ma, a \in X_\ma, a <_\ma \lgn_\ma(b) \}$ 
into $X_\ma$ whose value is called $\xr_\ma(b,c)$, and abbreviated $b(a)$. 
\begin{enumerate}
	\item  if $c \in \mct_\ma$ and $\lgn_\ma(c) < d_\ma$ and $a \in X_\ma$, then $c^\smallfrown \langle a \rangle$ exists, 
	i.e.  there is $c^\prime \in \mct_\ma$ such that $\lgn_\ma(c^\prime) = \lgn_\ma(c)+1$, $c^\prime(\lgn_\ma(c)) = a$, and   
	\[ (\forall a <_\ma \lgn_\ma(c) ) (c(a) = c^\prime(a)) \]
\item $\psi_0(\bar{x}, \overline{c})$ implies that if $b_1 \neq b_2 \in \mct_\ma$, $\lgn_\ma(b_1) = \lgn_\ma(b_2)$ then for some 
$n <_\ma \lgn_\ma(b_1)$, $b_1(n) \neq b_2(n)$.
\end{enumerate}

\item $\psi_3(\bar{x},\bar{y},\overline{c})$ defines the partial order $\tlf_{\ma}$ on $\mct_{\ma}$ given by initial segment, that is, 
such that that $b_1 \tlf_{\ma} b_2$ implies:
\begin{enumerate}
	\item for all $b, c \in \mct_\ma$, $b \tlf c$ implies $\lgn_\ma(b) \leq_\ma \lgn_\ma(c)$.

	\item $\lgn_\ma(b_1) \leq_\ma \lgn_\ma(b_2)$
	\item $(\forall a <_\ma \lgn_\ma(b_1) ) \left( b_2(a) = b_1(a)\right)$
\end{enumerate}
\end{enumerate}
\end{enumerate}
The family of all $\mct_\ma$ defined this way will be denoted $\tr(\Delta, M_1)$.  We refer to elements of this family as trees. 
\end{defn}

In $\ref{d:estt}(2)$, we ask for both first and last elements. The first element is important: we repeatedly use that these orders 
are pseudofinite. The last element is technical; we could alter this but would need to make other changes to ensure that the 
derived trees are not too large.   
Also, the condition in (6) should be understood as saying that while we need Cartesian products to exist, it is largely 
unimportant what exactly the order on these products is (as long as it is discrete, pseudofinite and so forth). 
It's sufficient that one such order behave well, like the usual G\"odel pairing function. We could have alternately asked that (6) hold for all 
products.

\begin{conv} \label{conv:second}
Many proofs will consider elements of a given tree $\mct_\ma$ 
or a given linear order $X_\ma$.  Although these orders or trees may be defined by some $\vp(\bar{x}, \dots)$ with 
$\ell(\bar{x}) > 1$, we will often omit overlines on their elements; this should not cause confusion as 
the provenance of such elements is clear. 
\end{conv}

\begin{defn} \label{cst}  \label{d:csp}
Say that $(M, M_1, M^+, M^+_1, T, \Delta)$ is a \emph{cofinality spectrum problem} when:
\begin{enumerate}
\item  $M \preceq M_1$.
\item $T \supseteq Th(M)$ is a theory in a possibly larger vocabulary.
\item $\Delta$ is a set of formulas in the language of $M$, i.e., we are interested in studying the orders of $\ml(M) = \ml(M_1)$ in the
presence of the additional structure of $\ml(M^+) = \ml(M^+_1)$. 
\item $M^+$, $M^+_1$ expand $M, M_1$ respectively so that $M^+ \preceq M^+_1 \models T$ 
and $(M^+_1, \Delta)$ has enough set theory for trees. 
\item We may refer to the components of $\cs$ as $M^\cs$, $\Delta^\cs$, etc. for definiteness. When $T = Th(M)$, $M=M^+$,
$M_1 = M^+_1$, or $\Delta$ is the set of all formulas $\vp(x,y,\overline{z})$ in the language of $T$ which satisfy $\ref{d:estt}(2)-(4)$, 
these may be omitted. 
\end{enumerate}
\end{defn}

In what follows, we work almost exclusively with $M^+, M^+_1$ but our results are for $M, M_1$. 
We will make the following conventions in the remainder of the paper:

\begin{conv}[Conventions on cofinality spectrum problems]  \label{c:defn} \emph{ } 
\begin{enumerate}
\item Extending Definition $\ref{d:estt}(3)(c)$
we say that $\cs$ is \emph{trivial} if every $\ma \in \ord(\cs)$ is trivial $($note that this need not mean all the orders are finite, just that the choice 
of tree-height bound $d_\ma$ is in each instance finite$)$. 
We will assume that all cofinality spectrums we work with are 
nontrivial, though sometimes this is repeated for emphasis. 
\item Likewise, by ``$\ma \in \ord(\cs)$'' we will mean ``$\ma \in \ord(\cs)$ and $\ma$ is nontrivial,'' unless otherwise indicated. 
\item By  definable we shall mean definable in the larger \emph{expanded} model,
i.e. in $M^+_1$, possibly with parameters, unless otherwise stated. 
\item We will present Cartesian products without explicitly mentioning the pairing functions, writing e.g. 
``let $\ma \in \ord(\cs)$ and let $\mb$ be such that $X_\mb = X_\ma \times X_\ma$.''
\item If $\bc, \ma_0, \dots, \ma_k \in \ord(\cs)$ are such that 
$X_\bc = X_{\ma_0} \times \cdots \times X_{\ma_k}$, and $c \in \mct_\bc$ and $n \in X_\bc$, $n < \lgn(c)$ $($thus, $c(n)$ is well defined$)$ then
we will write $c(n,i)$ to mean the $i$th coordinate of $c(n)$.  
\end{enumerate}
\end{conv}

\begin{defn} \label{cs-ord}
When $\cs_1, \cs_2$ are cofinality spectrum problems, write $\cs_1 \leq \cs_2$ to mean:
\begin{itemize}
\item $M^{\cs_1} = M^{\cs_2}$, $M^{\cs_1}_1 = M^{\cs_2}_1$
\item $\tau(M^{+, \cs_1}) \subseteq \tau(M^{+, \cs_2})$, i.e. the vocabulary may be larger, and likewise
$T^{+, \cs_1} \subseteq T^{+, \cs_2}$.
\item $( M^{+, \cs_2} \rstr_{\tau(M^{+, \cs_1})} ) \cong M^{+, \cs_1}$
\item $( M^{+, \cs_2}_1 \rstr_{\tau(M^{+, \cs_1})} ) \cong M^{+, \cs_1}_1$
\item $\Delta^{\cs_1} \subseteq \Delta^{\cs_2}$.  
\end{itemize} 
\end{defn}

We will study properties of orders and trees arising in cofinality spectrum problems, as we now describe.  
We follow the model-theoretic terminology, writing cuts instead of gaps. For clarity: 

\begin{defn} \label{mc} \emph{(Cuts, pre-cuts and representations of cuts)}
Let $\cs$ be a cofinality spectrum problem, $\ma \in \ord(\cs)$.  When $\kappa_1, \kappa_2$ are regular, define:

\begin{enumerate}
\item  A $(\kappa_1, \kappa_2)$-cut in $X_\ma$, i.e. in $(X_\ma, \leq_\ma)$,
is given by a pair of sets $(C_1, C_2)$ such that
\begin{itemize}
\item[(a)] $C_1 \cap C_2 = \emptyset$
\item[(b)] $C_1$ is downward closed, $C_2$ is upward closed
\item[(c)] $(\forall x \in C_1)(\forall y \in C_2) (x <_\ma y)$
\item[(d)] $C_1 \cup C_2 = X_\ma$
\end{itemize}

\item If $(C_1, C_2)$ satisfies conditions $(a)$, $(b)$, $(c)$ for being a cut it is called a \emph{pre-cut}, meaning that possibly
$(\exists c)(C_1 < c < C_2)$.

\item By a $(\kappa_1, \kappa_2)$-\emph{representation} of a pre-cut $(C_1, C_2)$ in $X_\ma$ we mean a pair of sequences
\\$(\langle a_i : i < \kappa_1 \rangle$, $\langle b_j : j < \kappa_2 \rangle)$ of elements of $X_\ma$ such that
\begin{enumerate}
\item $\kappa_1, \kappa_2$ are regular cardinals
\item $\langle a_i : i < \kappa_1 \rangle$ is strictly $<_\ma$-increasing and cofinal in $C_1$
\item $\langle b_j : j < \kappa_2 \rangle$ is strictly $<_\ma$-decreasing and coinitial in $C_2$ 
\end{enumerate}
\end{enumerate}
\end{defn}

When the $($pre-$)$cut  $(C_1, C_2)$ has a $(\kappa_1, \kappa_2)$-representation we say that it is a 
$(\kappa_1, \kappa_2)$-$($pre-$)$cut. When there is no danger of confusion, we may informally identify cuts or pre-cuts with one of their representations.

\begin{defn} \label{cst:card}
For a cofinality spectrum problem $\cs$ we define the following: 
\begin{enumerate}
\item $\ord(\cs) = \ord(\Delta^\cs, M^\cs_1)$ 

\br
\item $\cts = $
$\{ (\kappa_1, \kappa_2)   : ~$for some $\ma \in \ord(\cs, M_1)$, $(X_\ma, \leq_\ma)$ 
 has a $(\kappa_1, \kappa_2)$-cut$\}$

\br
\item \label{d:tr-cs} $\tr(\cs) = \{ \mct_\ma : \ma \in \ord(\cs) \} = \tr(\Delta^\cs, M^\cs_{1})$ 

\br

\item $\mc^{\mathrm{ttp}}(\cs) =$ $\{ \kappa ~:~ \kappa \geq \aleph_0, ~ \ma \in \ord(\cs)$, and there is in the tree $\mct_\ma$ 
 a strictly increasing sequence of cofinality $\kappa$ with no upper bound $\}$

\br
\item \label{card:t}  Let $\xt_\cs$ be $\min \mc^{\mathrm{ttp}}(\cs)$ and let $\xp_\cs$ be
$\min \{ \kappa : (\kappa_1, \kappa_2) \in \cts ~\mbox{and}~ \kappa = \kappa_1 + \kappa_2 \}$.
\end{enumerate}

\br
\emph{Our main focus in this paper will be $\mc(\cs, \xt_\cs)$ where this means:}

\br
\begin{enumerate}[resume]
\item For $\lambda$ an infinite cardinal, write
\[ \mc(\cs, \lambda) = \{ (\kappa_1, \kappa_2) ~:~ \kappa_1 + \kappa_2 < \lambda, ~ (\kappa_1, \kappa_2) \in \cts \}. \] 
\end{enumerate}
\end{defn}

\br

Note that by definition of $\cts$ and $\mc^{\mathrm{ttp}}(\cs)$, both $\xt_\cs$ and $\xp_\cs$ are regular. 
The following is a central definition of the paper:

\begin{defn} \emph{(Treetops)} \label{treetops}
Let $\cs$ be a cofinality spectrum problem and $\xt_\cs$ be given by $\ref{cst:card}(\ref{card:t})$. When
$\lambda \leq \xt_\cs$ we say that \emph{$\cs$ has $\lambda$-treetops}.
Our main focus will be the case $\lambda = \mu^+$ for some $\mu < \xt_\cs$. 
\end{defn}

The name \emph{treetops} reflects the definition of $\xt_\cs$: when 
$\kappa = \cf(\kappa) < \xt_\cs$, $\ma \in \ord(\cs)$ thus $\mct_\ma \in \tr(\cs)$, $\ref{cst:card}(5)$, 
any strictly increasing $\kappa$-sequence of elements of $\mct_\ma$ has an upper bound in $\mct_\ma$.

\begin{defn} \label{bc}
Let $\cs$ be a cofinality spectrum problem, $\ma \in \ord(\cs)$. 

\begin{enumerate}
\item[(0)] Write $0_\ma$ for the $\leq_\ma$-least element of $X_\ma$.
\item For any natural number $k$ and any $a \in X_\ma$, let $S^k_\ma(a)$ denote the $k$th successor of $a$ in the discrete
linear order $\leq_\ma$, if defined, and likewise let $S^{-k}_\ma(a)$ denote the $k$th predecessor of $a$, if defined. 
We will generally write $S^k(a)$, $S^{-k}(a)$ when $\ma$ is clear from context.
\item Say that $c \in \mct_\ma$ is \emph{below the ceiling} if $S^k(\lgn(c)) <_\ma d_\ma$ for all $k<\omega$, i.e. if these 
successors exist and the statements are true. 
\end{enumerate}
\end{defn}

We retain the notation $S^k(\cdots)$ rather than abbreviating to $+ k$ in Definition \ref{bc} because addition will be introduced formally later on. 

\br

To conclude this section, let us emphasize that Definition \ref{d:estt} is a choice of abstraction, suitable for our purposes but also relatively strong. 
In future work, it will be useful to consider various substitutions and weakenings. Towards this, we include some alternatives 
in \ref{d:alt}-\ref{d:alt2} below. 

\begin{disc} \label{d:alt} \emph{(Alternate version A: 1-to-1)} In Definition $\ref{d:estt}$ one could  
drop the bound $d_\ma$ and retain an implicit
bound on the length of sequences in $\mct_\ma$ by requiring no repetition in the range. Formally, one would modify 
$\ref{d:estt}$ as follows: drop $d_\ma$ from the definition of $\ord(\Delta, M_1)$, add condition $(1)$ below and replace 
$(7)(b)(iv)$ by $(2)$ below. 
\begin{enumerate}
\item $a_1 <_\ma a_2 <_\ma \lgn_\ma(b)$ implies $(b(a_1)) \neq (b(a_2))$
\item if $b \in \mct_\ma$, $a \in X_\ma \setminus \{ \xr_\ma(b(a^\prime)) : a^\prime < \lgn_\ma(b) \}$ 
and there is still room to concatenate, i.e.
\[ | X_\ma \setminus \left( \{ a \} \cup \{ \xr_\ma(b(a^\prime)) : a^\prime < \lgn_\ma(b) \} \right) | \geq 2 \]
\emph{then} $b^\smallfrown \langle a \rangle$ exists, i.e. there is $b^\prime$ such that $b \tlf_\ma b^\prime$, 
$\lgn_\ma(b^\prime) = \lgn_\ma(b)+1$ and $\xr(b^\prime(\lgn(b)) = a$. 
\end{enumerate}
\end{disc}

\begin{disc} \label{d:alt2} \emph{(Alternate version B: Allowing other orders)}
In Definition $\ref{d:estt}$ one could rename the current set $\ord(\cs)$ as $\operatorname{Psf-ord}(\cs)$, the
set of \emph{pseudofinite orders}, and make
the following changes. 
First, allow other kinds of linear orders in $\ord(\cs)$ $($e.g. dense linear orders, or 
any definable linear order$)$. Second, change the requirements on trees so that any tree is a set of sequences from
some $X_\ma$ into some $X_\mb$ where $\ma \in \operatorname{Psf-ord}(\cs)$ and $\mb \in \ord(\cs)$.
\end{disc}
\br

\subsection{Key examples} \label{s:examples}

A first motivating example is that of Section \ref{s:motiv-up}.  
Consider $M \models (\mathbb{N}, <)$. Then there are a set of $\ml$-formulas $\Delta \supseteq \{ x \leq y \}$, an expanded language $\ml^+$, 
an $\ml^+$-theory $T \supseteq Th(M)$ and such that $(M, M_1, M^+, M^+_1, T, \Delta)$ is a cofinality spectrum problem. 
For instance, let $T = (\mathcal{H}(\omega_1), \epsilon)$ and identify $\mathbb{N}$ with $\omega$. 
The reader interested primarily in our model-theoretic conclusions for regular ultrapowers may wish to look first 
at Section $\ref{s:sop2}$, where the main definitions are specialized to the case of regular ultrapowers. 

A second motivating example involves pairs of models which admit expansions. 
Suppose that $M \preceq M_1$, and so that this example is nontrivial, that $M$ includes an infinite definable linear order. 
$M, M_1$ can be respectively expanded to models $M^+, M_1^+$ in a larger signature $\tau_{M^+} \supseteq \tau_M$ such that:
\begin{enumerate}
\item $M^+ \preceq M_1^+$
\item $M = (M^+\rstr P^{M^+}) \rstr \tau_M$
\item $M_1 = (M_1^+\rstr P^{M_1^+}) \rstr \tau_M$
\item $M^+ = (\mch(\chi), \in, P^{M^+}, R^{M^+})_{R \in \tau(M)}$ for some $\chi$, where $M \in \mch(\chi)$
\end{enumerate}
As ultrapowers commute with reducts, this example generalizes the first one.

A third motivating example will be defined in Section \ref{s:p-t}, but is fundamental to one of our applications so we mention it here. 
Let $M = (\mch(\aleph_1), \in)$, $\bq = ([\mathbb{N}]^{\aleph_0}, \supseteq^*)$, $\mg$ a generic subset of $\bq$ and let $\cn = M^\omega/\mg$ be the
generic ultrapower. Then $M \preceq N$. It will be shown in Section \ref{s:p-t} that there is a natural cofinality spectrum problem $\cs$ with 
$M = M^+_\cs$, $N = M^+_{1,\cs}$.

A fourth motivating example involves arithmetic with bounded induction. 
Let $M$ be a model of Peano arithmetic with bounded induction, meaning that we have induction only for formulas all of whose
quantifiers are bounded. Choose $M_1$ so that $M \preceq M_1$. Let $T = Th(M)$ and let $\Delta$ be a set which includes
$\{(\exists w)( x + w = y) \}$ and its finite products. 
So the elements of $\ord(\Delta, M_1)$ are $($finite Cartesian products of$)$ 
initial segments of the domain of $M^\cs_1$, and the fact that this may be regarded as a cofinality
spectrum problem follows from the fact that G\"odel coding can be carried out in this context, meaning that we may speak about
sequences -- which are not too long --
of elements of $X_\ma$ and thus about trees. The naturalness of this example appears in Section $\ref{s:towards}$ below.

\br

\subsection{Basic tools}

We now develop some consequences of Definition $\ref{d:csp}$. Recall that definable means with parameters in $M^+_1$, 
and that \emph{below the ceiling} was defined in Definition \ref{bc} above.

\begin{claim} \emph{(Treetops for definable sub-trees)} \label{t-def}
Let $\cs$ be a cofinality spectrum problem, $M^+_1 = M^{+,\cs}_1$. 
Let $\ma \in \ord(\cs)$, so $\mct_\ma \in \tr(\cs)$. Let $\vp$ be a formula, possibly with parameters, in $M^+_1$,
and let $(\mct, \tlf_\ma)$ be the subtree of $(\mct_\ma, \tlf_\ma)$ defined by $\vp$ in $M^+_1$. 
Let $\langle c_\alpha : \alpha < \kappa \rangle$ be a $\tlf_\ma$-increasing sequence of elements of $\mct$, 
$\kappa = \cf(\kappa) < \xt_\cs$. Then there is $c_* \in \mct$ such that for all $\alpha < \kappa$, 
$c_\alpha \tlf_\ma c_*$.
\end{claim}

\begin{proof}
By definition of $\xt_\cs$, there is an element $c \in \mct_\ma$, not necessarily in $\mct$, such that for all $\alpha < \kappa$,
$c_\alpha \tlf_\ma c$. 
The set $\{ \lgn(c^\prime) : c^\prime  \tlf_\ma c$ and $c^\prime \in \mct \}$ 
is a nonempty definable subset of $X_\ma$, hence contains a last member $a_*$.
Let $c_*$ be such that $c_* \tlf c$ and $\lgn(c_*) = a_*$, i.e. $c_* = c\rstr_{a_*}$. 
\end{proof}

\begin{lemma} \label{c:ceiling} \emph{(Treetops below the ceiling)}
Let $\cs$ be a cofinality spectrum problem, $\ma \in \ord(\cs)$, $\kappa < \min \{ \xp_\cs, \xt_\cs \}$. Let $\mct \subseteq \mct_\ma$ be a definable
subtree and $\overline{c} = \langle c_\alpha : \alpha < \kappa \rangle$ 
a strictly $\tlf_\ma$-increasing sequence of elements of $\mct$. Then there exists 
$c_{**} \in \mct$ such that $\alpha < \kappa$ implies $c_\alpha \tlf c_{**}$ and $c_{**}$ is below the ceiling.
\end{lemma}

\begin{proof}  
Let $c_* \in \mct$ be such that $\alpha < \beta$ implies $c_\alpha \tlf c_*$, as given by Claim \ref{t-def}. 
As we assumed the sequence $\overline{c}$ is strictly increasing, for each $\alpha < \kappa$ the element $c_\alpha$ must be below the ceiling. 
If $c_*$ is also below the ceiling, we finish. Otherwise,  
\[ ( \{ \lgn(c_\beta) : \beta < \alpha \},  \{ S^{-k}(\lgn(c_*)) : k < \omega \} ) \]
describes a pre-cut in $X_\ma$. It cannot be a cut, as then $(\aleph_0, \kappa) \in \cts$, contradicting the definition of 
$\xp_\cs$. Choose $a \in X_\ma$ realizing this pre-cut, and let $c_{**} = c_*\rstr a$.
\end{proof}

Normally, $\xp_\cs \leq \xt_\cs$ though we shall not need this for the main theorems of the paper, instead keeping track of each separately.  

\section{The function $\lcf$ is well defined strictly below $\min \{ \xp^+_\cs, \xt_\cs \}$}  \label{s:uniqueness}

In the next six sections, 
we give a series of constructions which show how to translate certain conditions on
realization of pre-cuts in linear order into conditions on existence of paths through trees, leading to the central theorem of Section \ref{s:main-thm}. 

We generally write e.g. $M^+_1$ rather than $M^{+, \cs}_1$, but this should not cause confusion; 
the components of a cofinality spectrum problem are always understood to depend on a background $\cs$ fixed at the beginning of a proof.

\begin{lemma} \label{o:incr} \label{m2a}
Let $\cs$ be a cofinality spectrum problem. If $\ma \in \ord(\cs)$ is nontrivial, then for each infinite regular $\kappa \leq \xp_\cs$: 
\begin{enumerate}
\item there is a strictly decreasing $\kappa$-indexed sequence 
$\overline{a} = \langle a_\alpha : \alpha < \kappa \rangle$ of elements of $X_\ma$ such that
\[ ( \{ S^{k}(0_\ma) : k < \omega \}, \{ a_\alpha : \alpha < \kappa \} ) \]
represents a pre-cut $($and possibly a cut$)$ in $X_\ma$.
\item there is a strictly increasing $\kappa$-indexed sequence 
$\overline{a} = \langle a_\alpha : \alpha < \kappa \rangle$ of elements of $X_\ma$ such that
\[ ( \{ a_\alpha : \alpha < \kappa \}, \{ S^{-k}(d_\ma) : k < \omega \} ) \]
represents a pre-cut $($and possibly a cut$)$ in $X_\ma$.
\item thus, there is at least one infinite regular $\theta$ such that $(\kappa, \theta) \in \cts$, 
witnessed by a $(\kappa, \theta)$-cut in $X_\ma$.  
\item thus, there is at least one infinite regular $\theta^\prime$ such that $(\theta^\prime, \kappa) \in \cts$, 
witnessed by a $(\theta^\prime, \kappa)$-cut in $X_\ma$. 
\end{enumerate}
\end{lemma}

\begin{proof} 
(1) By induction on $\alpha < \kappa$ we choose elements $a_\alpha \in X_\ma$ such that:
\begin{itemize}
\item[(a)] for each $\alpha < \kappa$ and each $k<\omega$, $S^{k}(0_\ma) <_\ma a_\alpha$ 
\item[(b)] $\beta < \alpha$ implies $a_\alpha <_\ma a_\beta$. 
\end{itemize}
For $\alpha = 0$, let $a_0$ be the last element of $X_\ma$. As $\ma$ is nontrivial, condition (a) is satisfied. 
For $\alpha = \beta + 1$, let $a_\alpha = S^{-1}(a_\beta)$, recalling that any nonempty definable subset of $X_\ma$ has a greatest element so
the predecessor of any element not equal to $0_\ma$ is well defined.
As (a) holds for $\beta$ by inductive hypothesis, it will remain true for $\beta + 1$. 
For limit $\alpha$, by inductive hypothesis, 
\[ ( \{ S^{k}(0_\ma) : k < \omega \}, \{ a_\beta : \beta < \alpha \} ) \] 
is a pre-cut. However, it cannot be a cut, as then we would have $(\aleph_0, \cf(\alpha)) \in \cts$ 
with $|\alpha| + \aleph_0 < |\kappa| \leq \xp_\cs$, contradicting the definition of $\xp_\cs$. Let $a_\alpha$ be any element
realizing this pre-cut. This completes the construction of the sequence and thus the proof. 

(2) The argument is parallel to (1), going up instead of down. 

(3) Let $\overline{a}$ be a $\kappa$-indexed strictly increasing $\leq_\ma$-monotonic sequence of elements of $X_\ma$ given by (2). By construction, 
$B = \{ b \in X_\ma : \alpha < \kappa$ implies $a_\alpha <_\ma b \} \neq \emptyset$. 
Let $\theta$ be the cofinality of $B$ considered with the reverse order.
It cannot be the case that for some $b \in B$, 
\[ (\{ a_\alpha : \alpha < \kappa \}, \{ b \} ) \]
represents a cut, since $\overline{a}$ is strictly increasing, thus $\alpha < \kappa$ implies $a_\alpha <_\ma S^{-1}(b) <_\ma b$. 
So $\theta$ is an infinite (regular) cardinal, and $(\kappa, \theta) \in \cts$.

(4) The argument is parallel using (1). 
\end{proof}

Of course, in Lemma \ref{o:incr}, the cardinals $\theta$, $\theta^\prime$ may be quite large. In what follows, we will be interested in whether they must always 
be at least the size of $\xt_\cs$.

\begin{theorem} \label{m2} \emph{(Uniqueness)}
Let $\cs$ be a cofinality spectrum problem.
Then for each regular $\kappa \leq \xp_\cs$, $\kappa < \xt_\cs$:
\begin{enumerate}
\item there is one and only one $\lambda$ such that $(\kappa, \lambda) \in \cts$.
\item $(\kappa, \lambda) \in \cts$ iff $(\lambda, \kappa) \in \cts$.
\end{enumerate}
\end{theorem}

\begin{proof} 
Fix $\kappa$ satisfying the hypotheses of the theorem. We will prove that whenever we are given:
\begin{itemize}
\item $\ma, \mb \in \ord(\cs)$
\item $( \langle a^1_\alpha : \alpha < \kappa \rangle, \langle b^1_\epsilon : \epsilon < \theta_1 \rangle)$ representing a $(\kappa, \theta_1)$-cut in $(X_\ma, <_\ma)$
\item $( \langle b^2_\epsilon : \epsilon < \theta_2 \rangle, \langle a^2_\alpha : \alpha < \kappa \rangle)$ representing a $(\theta_2, \kappa)$-cut in $(X_\mb, <_\mb)$
\end{itemize}
then $\theta_1 = \theta_2$. The proof will proceed essentially by threading together the $\kappa$-sides of the cuts.  

To see why this will suffice for the theorem, note first that Lemma \ref{m2a} guarantees existence of \emph{some} $\theta_1$, $\theta_2$ such that 
$(\kappa, \theta_1) \in \cts$ and $(\theta_2, \kappa) \in \cts$. The statement in the previous paragraph will establish that 
$\theta_1 = \theta_2$. It will then follow by transitivity of equality that if $(\kappa, \theta), (\kappa, \theta^\prime) \in \cts$ 
then $\theta = \theta^\prime$, and likewise if $(\theta, \kappa), (\theta^\prime, \kappa) \in \cts$ then $\theta = \theta^\prime$.

Let $\bc \in \ord(\cs)$ be such that $X_\bc = X_\ma \times X_\mb$. Thus, for any $x \in \mct_\bc$ and any $n \leq \maxdom(x))$, 
$x(n)$ is a pair  $(x(n,1), x(n,2))$, with $x(n,1) \in X_\ma$, and $x(n,2) \in X_\mb$. 
Consider the definable subtree $\mct_0 \subseteq \mct_\bc$ consisting of all $x \in \mct_\bc$ such that
$x$ is strictly increasing in the first coordinate and strictly decreasing in the second, i.e.
\[ n^\prime <_\bc n <_\bc \lgn(x) \mbox{ implies } ( x(n^\prime, 1) <_\ma x(n, 1) ) \land ( x(n, 2) <_\mb  x(n^\prime, 2) ).\]
Keeping in mind the representations of cuts fixed at the beginning of the proof, 
we now choose $c_\alpha$, $n_\alpha$ by induction on $\alpha < \kappa$,
such that:
\begin{enumerate}
\item $c_\alpha \in \mct_0$ and $n_\alpha \in X_\bc$
\item $\beta < \alpha$ implies $M^+_1 \models c_\beta \tlf_\bc c_\alpha$
\item $c_\alpha$ is below the ceiling, Definition \ref{bc}
\item $n_\alpha = \lg(c_\alpha) - 1$, so $\maxdom(c_\alpha))$ is well defined (and $c_\alpha$ is not the empty sequence)
\item $c_\alpha(n_\alpha, \ell) = a^\ell_\alpha$ for $\ell = 1, 2$
\end{enumerate}

For $\alpha = 0$: let $c_0 = \langle a^1_0, a^2_0 \rangle$ and let $n_0 = 0_\ma$.

For $\alpha = \beta + 1$: since $c_\beta$ is below the ceiling, concatenation is possible. So by conditions (4)-(5) of the inductive hypothesis, 
we may concatenate $\langle a^1_{\beta+1}, a^2_{\beta+1}\rangle$ 
onto the existing sequence, and let $n_\alpha = n_\beta + 1$.

For $\alpha < \kappa$ limit: As $\cf(\alpha) < \min \{ \xp_\cs, \xt_\cs \}$, 
apply Lemma \ref{c:ceiling} to choose $c \in \mct_0$ such that
$\beta < \alpha$ implies $M^+_1 \models c_\beta \tlf c$, and $c$ is below the ceiling. Then the set
\[\{ n : n <_\bc \lgn(c), M^+_1 \models  ( c(n, 1) <_\ma a^1_\alpha ) \land ( a^2_\alpha <_\mb c(n, 2) ) \} \]
is definable, bounded and nonempty in $X_\bc$, so has a maximal element $n_*$.  By (b) above and the choice of $c$ as an upper bound, 
necessarily for all $\beta < \alpha$
\[ M^+_1 \models ( a^1_\beta <_\ma c(n_*, 1) ) \land ( c(n_*, 2) <_\mb a^2_\beta ) \]
As $c$ and thus all of its initial segments are below the ceiling, concatenation is possible.  
Define $c_\alpha$ to be $(c|_{n_*+1})^{\smallfrown}(a^1_\alpha, a^2_\alpha)$. 
Let $n_\alpha = n_*+1$.  By construction, $c_\alpha$ will remain
strictly monotonic in all coordinates and will $\tlf_\bc$-extend the existing sequence, as desired.
This completes the inductive choice of the sequence.

As $\kappa < \xt_\cs$, by Claim \ref{t-def} there is $c \in \mct_0$ so that $\alpha < \kappa$ implies $c_\alpha \tlf_\ma c$. Let $n_{**} = \lgn(c)-1$, so $n_{**} \in X_\bc$. 
For $\ell = 0,1,2$ and each $\epsilon < \theta_\ell$, recalling the $b^i_\epsilon$ from our original choice of cuts, define:
\[ n_{\epsilon, 1} = \max \{ n \leq_\bc n_{**} : c(n,1) <_\ma b^1_\epsilon \} \]
\[ n_{\epsilon, 2} = \max \{ n \leq_\bc n_{**} :  b^2_\epsilon <_\mb c(n,2) \} \]
\br
\noindent 
Clearly, $\alpha < \kappa$ implies $n_\alpha <_\bc n_{\epsilon, \ell}$ for $\ell = 1, 2$. 
By the choice of sequences witnessing the original cuts, for $\ell = 1, 2$ we have that 
\[ (\langle n_\alpha : \alpha < \kappa \rangle, \langle n_{\epsilon, \ell} : \epsilon < \theta_\ell \rangle) \]
represents a cut in $X_\bc$.  (Clearly it is a pre-cut; if it were filled, say, by $i_*$, then the elements 
$c(i_*, 1)$ and $c(i_*, 2)$ would realize our two original cuts, a contradiction.)  
As we assumed $\theta_1, \theta_2$ were each regular, clearly $\theta_1 = \theta_2$. 

This completes the proof.
\end{proof}

In light of Theorem \ref{s:uniqueness}, the following will be well defined. 

\begin{defn}[The lower cofinality $\lcf(\kappa, \cs)$]  \label{lcf-cs}
Let $\cs$ be a cofinality spectrum problem.
For regular $\kappa \leq \xp_\cs$, $\kappa < \xt_\cs$, we define $\lcf(\kappa, \cs)$ 
to be the unique $\theta$ such that $(\kappa, \theta) \in \cts$.
\end{defn}

\begin{rmk}  
Lemma \ref{c:ceiling} could be proved under the hypothesis   
that $\lcf(\aleph_0, \cs)$ is large, as could Lemma $\ref{o:incr}$.
\end{rmk}

The function $\kappa \mapsto \lcf(\kappa, \cs)$ remains interesting even after the main theorems of this paper. 

\begin{cor} \label{lcf-cor}
Let $\cs$ be a cofinality spectrum problem and $\gamma$ a regular cardinal, $\gamma \leq \xp_\cs$, $\gamma < \xt_\cs$.
Then the following are equivalent:
\begin{enumerate}
\item $\lcf(\gamma, \cs) = \gamma^\prime$
\item $(\gamma, \gamma^\prime) \in \cts$
\item $(\gamma^\prime, \gamma) \in \cts$
\end{enumerate}
\end{cor}

\begin{cor} \label{n-is-enough}
Let $\cs$ be a cofinality spectrum problem and let $\kappa, \theta$ be
regular cardinals with $\kappa \leq \xp_\cs$, $\kappa < \xt_\cs$. 
In order to show that $(\kappa, \theta) \notin \cts$, 
it is sufficient to show that for \emph{some} nontrivial $\ma \in \ord(\cs)$,
$X_\ma$ has no $(\kappa, \theta)$-cut.
\end{cor}

Corollary \ref{n-is-enough} shows that each of the orders is in some sense a reflection of the structure of the whole cofinality spectrum problem, 
at least for the cardinals we consider. As a result, when proving results about $\cts$ we are free to work in the nontrivial $\ma$ which appears most suited to the given proof.   
Moreover:

\begin{concl} \label{cor:sym}
Let $\cs$ be a cofinality spectrum problem. 
In light of $\ref{m2}$ we may, without loss of generality, study $\mc(\cs, \xt_\cs)$ by looking at 
\[ \{ (\kappa_1, \kappa_2) : (\kappa_1, \kappa_2) \in \mc(\cs, \xt_\cs), \kappa_1 \leq \kappa_2 \} \]
\end{concl}

\noindent The proof of Theorem \ref{m2} has the following very useful corollary. 

\begin{cor}[Definable monotonic maps exist] \label{bijection} 
Let $\cs$ be a cofinality spectrum problem, $\ma, \mb \in \ord(\cs)$, and $\kappa = \cf(\kappa) \leq \xp_\cs$, 
$\kappa < \xt_\cs$.
Let $\overline{a} = \langle a_\alpha : \alpha < \kappa \rangle$ be a strictly $<_\ma$-monotonic sequence of elements of $X_\ma$, and let
$\overline{b} = \langle b_\alpha : \alpha < \kappa \rangle$ be a strictly $<_\mb$-monotonic sequence of elements of $X_\mb$.
Then in $M^+_1$ there is a definable monotonic partial injection $f$ from $X_\ma$ to $X_\mb$ whose 
domain includes $\{ a_\alpha : \alpha < \kappa \}$ and such that $f(a_\alpha) = b_\alpha$ for all $\alpha < \kappa$.
\end{cor}

\begin{proof}
Without loss of generality, $\overline{a}$ is increasing and $\overline{b}$ is decreasing (otherwise carry out this proof twice, and compose the functions). 

Carry out the proof of Theorem \ref{m2},
substituting $\overline{a}$ here for $\langle a^1_\alpha : \alpha < \kappa \rangle$ there, and $\overline{b}$ here for
$\langle a^2_\alpha : \alpha < \kappa \rangle$ there. 
Let $\overline{c} = \langle c_\alpha : \alpha < \kappa \rangle$ be the corresponding path through the definable tree 
$\mct_0$ constructed in the proof of Theorem \ref{m2}, and let
$c \in \mct_0$ be an upper bound for this path in the tree, as there. Then let $f: \overline{a} \rightarrow \overline{b}$ be the function whose graph is
$\{ (c(n,0), c(n,\ell)) : n <_\bc \lgn(c)  \}$. Clearly $f$ is definable, and the hypothesis of monotonicity in each coordinate from the proof of Theorem \ref{m2} guarantees that $f$ is an injection, which takes $a_\alpha$ to $b_\alpha$ for all $\alpha < \kappa$. 
\end{proof}

A note to model theorists:
Recall that by a characterization of Morley and Vaught the saturated models are exactly those which are both homogeneous and universal. 
Regular ultrapowers are universal for models of cardinality no larger than the size of the index set, and thus
failures of saturation come from failures of homogeneity.
Saturation of regular ultrapowers can therefore be gauged by the absence or existence
of internal structure-preserving maps between small subsets of the ultrapower.
A recurrent theme of this paper, illustrated by Corollary \ref{bijection}, is how trees assist in the building of such maps. 

\begin{cor} \label{extend}
Let $\cs_1, \cs_2$ be cofinality spectrum problems and suppose that $M^{\cs_1} = M^{\cs_2}$, 
$M^{{+,\cs}_1}_1 = M^{{+, \cs}_2}_1$. 
If $\ord(\cs_1) \cap \ord(\cs_2)$ contains a nontrivial $\ma$ $($\emph{thus} 
$\Delta^{\cs_1} \cap \Delta^{\cs_2} \neq \emptyset$$)$ then for all regular $\kappa$
with $\kappa \leq \xp_\cs$, $\kappa < \xt_\cs$, 
\[ \lcf(\kappa, \cs_1) = \lcf(\kappa, \cs_2) \]
Moreover, the same conclusion holds if $\cs \leq \cs^\prime$.
\end{cor}

Thus, without loss of generality, when computing $\cts$ for such $\kappa$ 
we may work in a larger language $($provided $M^\cs, M^\cs_1$ admit the corresponding expansion and remain an elementary pair in
the larger language$)$ and/or consider a larger set of formulas $\Delta$, provided that it meets the closure conditions of
Definition \ref{d:estt}.

\begin{cor} \label{extend2}
Given a cofinality spectrum problem $\cs$, we may assume $\ord(\cs)$ is closed under definable subsets of $X_\ma$, i.e. 
whenever $\ma \in \ord(\cs)$, $\psi(x)$ a formula in the language of $M_1$  such that
$\psi(x) \vdash$ $x \in X_\ma$, there is $\mb \in \ord(\cs)$ with $\leq_\mb = \leq_\ma$ and
\[ X_\mb = \{ a \in X_\ma : M_1 \models \psi(a) \}. \]
For definiteness, we specify that $d_\mb = \min \{ d_\ma, ~\max \{ x: \psi(x) \} \}$. 
\end{cor}

\begin{proof}
Let $\Delta^\prime \supseteq \Delta^\cs$ be the set of $\vp^\prime(x,y,\overline{z}^\prime)$ such that 
for some $\vp(x,y,\overline{z}_1) \in \Delta$ and
$\psi(x, \overline{z}_2) \in \tau(M^\cs)$ we have that $\overline{z}^\prime = {\overline{z}_1}^\smallfrown\overline{z}_2$ 
and $\vp^\prime = \vp \land \psi$. Now apply Corollary \ref{extend}  
in the case where $\cs^\prime = (M, M_1, M^+, M^+_1, T, \Delta^\prime)$ (i.e. only the last component changes). 
Clearly, $\cs^\prime$ remains nontrivial. 
\end{proof}

In fact, for our proofs below, we use this only to assume that for $\ma \in \ord(\cs)$ 
and $a \in X_\ma$, the order $(\{ x : x \in X_\ma, x \leq a \}, <_\ma)$ may be regarded as an element of $\ord(\cs)$. 
This would be a perfectly reasonable condition to add to the definition of ESTT, though we prefer to derive it.

\section{Cofinality spectrum problems are locally saturated} \label{s:sat}

In this section we prove that cofinality spectrum problems have a certain amount of local saturation. 
Here ``local'' means partial types which are finitely satisfiable in one of the distinguished linear orders. 

\begin{defn} \label{x15}
Let $\cs$ be a cofinality spectrum problem and $\lambda$ a regular cardinal.
Let $p = p(x_0, \dots x_{n-1})$ be a consistent partial type with parameters in $M^+_1$.
We say that $p$  is an \emph{$\ord$-type over $M^+_1$} if $p$ is a consistent partial type in $M^+_1$ and
for some $\ma_0, \dots \ma_{n-1} \in \ord(\cs)$, we have that
\[ p \vdash \bigwedge_{i < n} x_i \in X_{\ma_i}. \]
We say simply that $M^+_1$ is \emph{$\lambda$-$\ord$-saturated} if every $\ord$-type over $M^+_1$
over a set of size $<\lambda$ is realized in $M^+_1$. Finally, we say that $\cs$ is {$\lambda$-$\ord$-saturated} if $M^+_1$ is.
\end{defn}

As we asked that $\ord(\cs)$ be closed under small Cartesian products,
without loss of generality we prove Theorem \ref{x16} assuming that $p = p(x)$ where $p \vdash$ $x \in X_\ma$ for some
$\ma \in \ord(\cs)$. (Note $p$ may not be a $1$-type, for instance if the formula defining $X_\ma$ has arity greater than $1$.)

\begin{theorem} \label{x16} 
Let $\cs$ be a cofinality spectrum problem. If $\kappa < \min \{ \xp_\cs, \xt_\cs \}$ then $\cs$ is 
$\kappa^+$-$\ord$-saturated.
\end{theorem}

\begin{proof}
We prove the claim by induction on $\kappa < \min \{ \xp_\cs, \xt_\cs \}$. 
Suppose that $\kappa = \aleph_0$ or that the theorem holds for all infinite cardinals $\rho < \kappa$. Suppose we are given $\ma \in \ord(\cs)$ and
$p = \{ \vp_i(x,\overline{a}_i) : i < \kappa \}$ which is finitely satisfiable in $X_\ma$. 
By a second (call it internal) induction on $\alpha \leq \kappa$, we choose $c_\alpha \in \mct_\ma$ and $n_\alpha \in X_\ma$ such that:
\begin{enumerate}
\item $n_\alpha = \lgn(c_\alpha) - 1$
\item $\beta < \alpha$ implies $c_\beta \tlf c_\alpha$
\item if $i < \beta \leq \alpha$ and $n_\beta \leq_\ma n \leq_\ma n_\alpha$, then
$M_1 \models \vp_i(c_\alpha(n), \overline{a}_i)$.
\item $c_\alpha$ is below the ceiling.
\end{enumerate}
For $\alpha = 0$ this is trivial. For $\alpha = \beta + 1$, by the external inductive hypothesis (or by the definition of a consistent type if $\kappa = \aleph_0$), 
let $d$ realize
$\{ \vp_i(x,\overline{a}_i) : i \leq \beta \}$, since $|\alpha| < \kappa$. By the internal inductive hypothesis (4), 
we may concatenate, so let $c_\alpha = {c_\beta}^\smallfrown \langle d \rangle$, $n_\alpha = n_\beta + 1$.
For $\alpha$ limit $\leq \kappa$, $\cf(\alpha) < \min \{ \xp_\cs, \xt_\cs \}$ so by Lemma \ref{c:ceiling} there is $c_* \in \mct$ such that
$\beta < \alpha$ implies $c_\beta \tlf c_*$ and $c_*$ is below the ceiling.
Let $n_* = \lgn(c_*) - 1$. Now we correct the value at the limit by restricting to a suitable initial segment which
preserves item (3).
That is, for each $i < \alpha$, define
\[ n(i) = \max \{ n \leq_\ma n_* : ~ \models \vp_i(c_*(n_*), \overline{a}_i) ~\text{for all $m$ such that }  n_i <_\ma m \leq_\ma n \} \]
As this is a bounded nonempty subset of $X_\ma$, $n(i)$ exists for each $i<\alpha$. 
By the internal inductive hypothesis (3), $n(i) {_\ma>} n_\beta$ for each $i, \beta < \alpha$. 
Thus, $( \{ n_\beta : \beta < \alpha \}, \{ n(i) : i < \alpha \})$ represents a pre-cut in $X_\ma$. 
Let $\gamma$ be the reverse cofinality of the set $\{ n(i) : i < \alpha \}$, i.e. its cofinality under the reverse order.
Necessarily $\gamma \leq |\alpha| \leq \kappa < \xp_\cs$.
If $(\cf(\alpha), \gamma) \in \mc(\cs, \xt_\cs)$, we would contradict the definition of $\xp_\cs$.
Thus, there is an element $n_{**}$ such that for all $i < \alpha$ and $\gamma < \alpha$,  
$n_\gamma <_\ma n_{**} <_\ma n(i)$. Let $c_\alpha = c\rstr_{n_{**}}$ and let $n_\alpha = \lgn(c_\alpha) - 1$.
Then by construction,
\[ i < \alpha \mbox{ implies } M_1 \models \vp_i(c_\alpha(n_\alpha), \overline{a}_i) \]
as desired, completing the limit step. As the limit case was also proved for $\alpha = \kappa$, 
$c_\kappa(n_\kappa)$ realizes the type $p$, which completes the proof.
\end{proof}

Note that Theorem $\ref{x16}$ does not require that the type consist of instances of formulas from some finite set $($a usual definition of local$)$, however 
this is a kind of local saturation in the sense that we use $p(x) \vdash x \in X_\ma$ for some suitable $\ma$.

\section{Cofinality spectrum problems have sufficient Peano arithmetic} \label{s:towards}
In this section two things are accomplished. First, we build a certain amount of arithmetic, 
which shows the naturalness of the corresponding example from Section \ref{s:examples}, and 
makes available addition and multiplication which will be useful in later proofs. 
Second, we show that for certain $\ma \in \ord(\cs)$, we may regard all definable subtrees 
$\mct \subseteq \mct_\ma \in \tr(\cs)$ as definable subsets of some $X_\mb$, $\mb \in \ord(\cs)$. 
Thus, in our future constructions we will have available more powerful trees, e.g. of sequences of finite tuples some of whose coordinates belong
to some $X_\ma$ and some of whose coordinates effectively belong to a definable tree (thus are internal partial functions on, 
or subsets of, some $X_\ma$, $\ma \in \ord(\cs)$). This could have been guaranteed simply by making 
stronger assumptions in Definition \ref{d:estt}, which would remain true in our cases of main interest. Thus, the reader interested 
primarily in models of set theory or in regular ultrapowers may wish to simply read Convention \ref{c:15}, Conclusion \ref{p:trees}, 
Convention \ref{c:d10} and continue. 
However, the analysis here shows that the necessary structure already arises from our more basic assumptions. 

\br
To begin, given $\cs$ and $\ma \in \ord(\cs)$,
we introduce notation for the relative cardinality 
of definable sets $A \subseteq X_\ma$, $B \subseteq X_\mb$ as computed by the model $M^+_1$ under consideration.

\begin{conv} \emph{(An internal notion of cardinality)} \label{c:15}
Let $\cs$ be a cofinality spectrum problem and $\ma, \mb \in \ord(\cs)$. Let $\bc = \ma \times \mb$ and let $\bij(\ma, \mb)$ be the definable
subtree of $\mct_\bc$ given by
${(\vp(\mct_\bc), \tlf_\bc)}$, where $\vp(x)$ says that 
$\{ (x(c,0), x(c,1)) : c <_\bc \lgn(x) \}$ is the graph of a partial one-to-one function from $X_\ma$ to $X_\mb$.   
We will adopt the following convention. Whenever $A$, $B$ are given as definable subsets of $X_\ma$, $X_\mb$ respectively, we write
\[  |A| \leq^\cs |B| \]
to mean ``there exists $x \in \bij(\ma, \mb)$ such that $A \subseteq \{ x(n,0) : n <_\bc \lgn(x) \}$
and $\{ x(m,1) : m <_\bc \lgn(x) \} \subseteq B$''. Likewise, we write
\[ |A| <^\cs |B|  \]
to mean ``$(|A| \leq^\cs |B|) \land \neg ( |B| \leq^\cs |A| )$,'' i.e. $|A| \leq^\cs |B|$ and there does not exist $x \in \bij(\ma)$ which
is an injection from $B$ into $A$.
\end{conv}

\begin{disc}
Does the internal notion of cardinality just described behave more like $($the cardinality of$)$ the natural 
numbers or infinite sets? For our present purposes, it behaves like pseudo-finite numbers. 
However, the same definition applies in very different contexts with different results, 
for instance in weakenings of cofinality spectrum problems studied in \cite{MiSh:F1245}, where the order is dense. 
\end{disc}

We now show that a sufficient amount of G\"odel coding is available in any CSP. More precisely, we would like to find
$\ma \in \ord(\cs)$ such that for some other $\mb \in \ord(\cs)$ extending $\ma$, the set of G\"odel codes of partial functions from 
$X_\ma$ to $X_\ma$ may be viewed as a definable subset of $X_\mb$. 
Informally, the issue is that $X_\ma$ may need to be quite small relative to $X_\mb$; 
we would like not only for statements of basic arithmetic and 
bounded induction to hold in $X_\mb$ of elements in $X_\ma$, but crucially, 
for all functions necessary for G\"odel coding of
$X_\ma$-sequences -- which, a priori, may be fast growing -- to remain available. 

\begin{lemma} \label{coding-lemma}
Let $\cs$ be a CSP and $\mb \in \ord(\cs)$. Then there is $\ma \in \ord(\cs)$ such that $X_\ma$ is an initial segment of $X_\mb$ and 
all G\"odel codes for elements of $\mct_\ma$ belong to $X_\mb$. In particular, we may identify $\mct_\ma$ with a definable subset of $X_\mb$. 
\end{lemma}

\begin{proof}
Let $\mb \in \ord(\cs)$ be given.  
We define a series of formulas in the language of $M^+_1$. 
The variables range over elements of $X_\mb$ unless otherwise indicated.

First, we define (the graph of a partial function corresponding to) addition to be the set of $(x,y,z)$ 
such that there is an element $\eta$ of $\mct_\mb$, i.e. a sequence of elements of $X_\mb$,
of length $y$ whose first element is $x$, whose last element is $z$, and which increments by one. Formally, 
$\vp_+(x,y,z) = (\exists \eta \in \mct_\mb)(\lgn(\eta) = y \land \eta(0) = x \land \eta(y-1) = z \land (\forall i)(i <\lgn(\eta) \rightarrow \eta(S(i)) = S(\eta(i)) )$. 
Note that we may do this uniformly across $\mb \in \ord(\cs)$ if we ask that $\vp_+$ also accept a fourth parameter, the defining parameter $c_\mb$ for $X_\mb$,  
and require that the formula 
$\vp_+(x,y,z)$ imply $x, y, z \in X_\mb$ and that $(\forall i)$ abbreviate $(\forall i \in X_\mb)$. We will assume this to be true, and will omit it here 
and in the remaining formulas of this proof for readability. 

Analogously, we may define multiplication 
$\vp_\times(x,y,z)$ by substituting $(i < \lgn(\eta) \rightarrow \eta(S(i)) = \eta(i)+x )$ 
in the appropriate place, i.e. requiring that the sequence increment by $x$. 
Define exponentiation 
$\vp_{exp}(x,y,z)$ by substituting 
$(i < \lgn(\eta) \rightarrow \eta(S(i)) = \eta(i) \times x )$ in the appropriate place, 
i.e. requiring that the sequence increment by a factor of $x$.  

Note that $\vp_+(x,y,z)$, $\vp_\times(x,y,z)$, and $\vp_{exp}(x,y,z)$ are graphs of partial functions, which we will refer to informally as $x+y$, $x\times y$ and $x^y$. 
Let $\theta(x)$ assert that $x+y$, $x \times y$ and $x^y$ exist for all $y<x$. 

Clearly, with addition and multiplication we can define ``$x$ is a prime.'' 
Let $\vp_4(x,y)$ assert that $x$ is the $y$th prime.
Let the formula $\vp_5(x,n,m)$ assert that $x$ is divisible by the $n$th prime precisely $m$ times, by asserting the existence of $\eta \in \mct_\mb$ of length $m$ whose first element is 
$x$, whose subsequent elements decrease by a factor of the $n$th prime and whose last element has no more such factors. 

Finally, let the formula $\vp_6(x,\eta)$ assert that $x \in X_\mb$ is a G\"odel code for $\eta$, by asserting the existence of $\eta \neq \emptyset$ 
such that $\eta \in \mct_\mb$, $x > 2$, and for all $i < \lgn(\eta)$, $x$ is divisible by the $i$th prime precisely $\eta(i)+1$ times. 

To conclude, let $\psi(w)$ be the formula: 
\[ (\forall x < w)\theta(x) \land (\forall \eta \in \mct_\mb)\left((\lgn(\eta) < w \land (\forall i < \lgn(\eta))(\eta(i) < w)) \rightarrow (\exists x)(\vp_6(x,\eta))\right) \] 
which asserts that G\"odel codes exist for all functions from $X_\mb \rstr_w$ to itself. 
Then $\psi(w)$ holds of $0_\mb$, and of all its finite successors, so it holds also on some nonstandard number $a_*$. 
Let $\ma$ be the order given by $X_\ma = (X_\mb \rstr a_*, \leq_\mb)$, and $d_\ma = \min \{ d_\mb, \max(X_\ma) \}$. Then $\ma \in \ord(\cs)$ by 
Corollary \ref{extend2}, and inherits non-triviality from $\mb$.  Thus $\ma, \mb$ are as desired.
 
The ``in particular'' clause in the statement of the Lemma is clear as this was done in a (uniformly) definable way.  
\end{proof}

\begin{defn} \emph{(Covers)} \label{d:covers}
Let $\cs$ be a cofinality spectrum problem and $\ma \in \ord(\cs)$.
\begin{enumerate}
\item Say that $\mb \in \ord(\cs)$ is a \emph{cover} for $\ma$ if all G\"odel codes for elements of $\mct_\ma$ belong to $X_\mb$.
The usual case is when $X_\ma$ is an initial segment of $X_\mb$, so is itself an element of $\ord$ by $\ref{extend2}$. 

\item We define $k$-coverable by induction on $k<\omega$.
\begin{enumerate}
\item Say that $\ma$ is $0$-coverable if $\ma \in \ord(\cs)$ is nontrivial.
\item Say that $\ma$ is $k+1$-coverable if there exists $\mb \in \ord(\cs)$ such that $\mb$ is a cover for $\ma$ and
is itself $k$-coverable.
\end{enumerate}
\item Say that $\ma$ is \emph{coverable} if it is $1$-coverable; this will be our main case.
\end{enumerate}
\end{defn}

\begin{claim} \label{c:z4}
Let $\cs$ be a cofinality spectrum problem. For each $k < \omega$ the set 
\[ \{ \ma \in \ord(\cs) : ~\ma ~\mbox{is $k$-coverable}\} \] 
is not empty.
\end{claim}

\begin{proof} 
We prove this by induction on $k < \omega$. 
In the case where $k = 0$ this is trivially true as $\Delta$ is nonempty and $\ord(\cs)$ is assumed to be nontrival. 
Suppose then that $\mb \in \ord(\cs)$ is $k$-coverable; we would like to find $\ma \in \ord(\cs)$ which is 
$k+1$-coverable, which we can do by applying Lemma \ref{coding-lemma}. 
\end{proof}

It will be useful to have the analogous result for Cartesian products.  

\begin{defn} \label{c:pair} 
Say that $\ma \in \ord(\cs)$ is \emph{coverable as a pair} $($by $\bd \in \ord(\cs)$$)$ when:
\begin{itemize}
\item[(a)] there is $\bc \in \ord(\cs)$ such that $X_\bc = X_\ma \times X_\ma$ and $\ref{d:estt}(6)$ holds of $\ma, \bc$ 
\item[(b)] $\bc$ is coverable $($by $\bd$$)$.
\end{itemize} 
\end{defn} 

\begin{cor} \label{cov-pair}
There exists $\ma \in \ord(\cs)$ which is coverable as a pair, say by $\bd \in \ord(\cs)$, and moreover this implies:
\begin{enumerate}
\item There is $a \in X_\ma$, not a finite successor of $0_\ma$, 
such that the G\"odel codes for functions from $[0,a]_\ma$ to $[0,a]_\ma \times [0,a]_\ma$ may be identified with a definable subset of $X_\bd$.
\item $\ma$ is coverable by $\bd$. 
\end{enumerate}
When $\ma$ is coverable as a pair, we may informally abbreviate condition $(1)$ by saying 
that some given $M^+_1$-definable function may be thought of as an element of $\mct_{\ma \times \ma}$.
\end{cor}

\begin{proof} 
By Claim \ref{c:z4} (or just Lemma \ref{coding-lemma}) and the fact that 
Definition \ref{c:pair}(a) requires that the pairing function be well behaved.    
\end{proof}

When constructing trees in the continuation of the paper, we will generally build trees of pseudofinite sequences of
$n$-tuples $(n < 10)$. The components of these tuples will either belong to some $X_{\ma}$ for $\ma \in \ord(\cs)$,
or they will be definable partial functions with domain and range contained in some $X_\ma$ which is 
$k$-coverable or coverable as a pair (i.e. they will be elements of certain trees), 
or they will be domains of such functions, as has been justified by 
the work in this section: 

\begin{concl} \emph{(More powerful trees)} 
\label{p:trees} 
Let $\ma \in \ord(\cs)$ be coverable. 
We may in our future constructions consider {as} a definable subtree 
of an element of $\tr(\cs)$ any tree of sequences of finite tuples some of whose coordinates 
range over elements of a definable sub-tree of $\mct_\ma$. Without loss of generality, we also allow some coordinates to 
be the domain or range of such a function.
\end{concl}

\begin{proof} 
By Corollary \ref{n-is-enough}, to study $\mc(\cs, \xt_\cs)$ we may freely choose the order we work in from among the nontrivial elements of 
$\ord(\cs)$. By Claim \ref{c:z4} we may assume $\ma$ is coverable (or $k$-coverable if necessary), so elements of $\mct_\ma$
may be considered as elements of a definable subset of some $X_\mb$, $\mb \in \ord(\cs)$. Then the desired tree will be a definable sub-tree of 
$\mct_\mb$ or of $\mct_\bc$ where $\bc$ is a finite Cartesian product of $\mb$.

For the last line: The use of domains and ranges of such functions is entirely a matter of presentation. 
We could easily in each case use a different defining formula for the tree which did not 
list the domain as a separate component of the tuple. For more complex arguments, one could alternatively 
formalize a suitable internal notion of power set.  
\end{proof}

We now justify the naturalness of models of Peano arithmetic with bounded induction as an example of a cofinality spectrum problem. 
Let exponentiation mean in the sense of the relation from Lemma \ref{coding-lemma} (which satisfies all of the usual properties of the graph of  
exponentiation except for the sentence $(\forall x \forall y \exists z)\vp_{exp}(x,y,z)$). 

\begin{defn} \label{d:mult2} \emph{(Weak powers)} 
Let $\cs$ be a cofinality spectrum problem and $\ma \in \ord(\cs)$.
Let $e_\ma$ denote the last element of $X_\ma$. 
Say that $a \in X_\ma$ \emph{has weak powers} if $a$ is not a finite successor of $0_\ma$, there exists 
$z \in X_\ma$ such that $(a^a = z)$, and for each $\ell < \omega$,
\[ a^a ~ <_\ma ~  S^{-\ell}(e_\ma) \]
\end{defn}

\begin{rmk} \label{k-powers}
Note that Definition \ref{d:mult2} implies that $a^k <_\ma S^{-\ell}(e_\ma)$ 
for each $\ell, k < \omega$, where $a^k$ 
means $a^{S^k(0_\ma)}$.  
\end{rmk}

\begin{lemma} \label{m-pa}
Let $\cs$ be a cofinality spectrum problem, $\ma \in \ord(\cs)$, and suppose $a \in X_\ma$ has weak powers. 
Then there are definable functions $+_\ma$, $\times_\ma$ such that identifying $k$ with $S^k(0_\ma)$, the model
\[ N^*_a := (\bigcup_{k < \omega} [0_\ma, a^k)_\ma; +_\ma, \times_\ma, S^1_\ma, \leq_\ma, 0_\ma ) \]
may be regarded as a model of 
$I\Delta_0$, arithmetic with bounded induction, 
{meaning} that the arithmetic operations
are well defined on the domain of $N^*_a$ and for any bounded formula $\vp(x)$ in the language of arithmetic  
which implies $x\in X_\ma$, the following holds in $M_1$:
\begin{align*} 
(\forall x \leq a_*)(\forall \overline{y}) 
\left( \vp(0_\ma,\overline{y}) \land (\forall z \leq x)(\vp(z, \overline{y}) \rightarrow \vp(S(z), \overline{y})) \rightarrow (\forall z \leq x)\vp(z, \overline{y}) \right) 
\end{align*}
\end{lemma}

\begin{rmk}
Clearly, the domain of $N^*_a$ is not definable in $M^+_1$. 
\end{rmk}

\begin{proof} {(of Lemma \ref{m-pa})}
First, we define $+_\ma$, $\times_\ma$ as partial functions from $X_\ma \times X_\ma \rightarrow X_\ma$ and check they behave as desired. 
We could use the formulas from Lemma \ref{coding-lemma}, but it is interesting to do this another way, based only on $\ref{c:15}$. 
Let $\nt_\ma = \{ [e_1, e_2) : e_1 <_\ma e_2,~ e_1, e_2 \in X_\ma \}$. Formally, elements of $\nt_\ma$ are presented as pairs but we think of them as closed and  
open intervals. Let $E_\ma$ be the definable equivalence relation on $\nt_\ma$ given by having the same cardinality in the sense of $\cs$ (i.e. writing 
$X = \{ x : e_1 \leq_\ma x <_\ma e_2 \}$, $Y = \{ y : e_1^\prime \leq_\ma y <_\ma e_2^\prime \}$, we have that $|X| \leq^\cs |Y|$ and 
$|Y| \leq^\cs |X|$). 

Denote by $0_\ma$ the $\leq_\ma$-minimal element of $X_\ma$.
For addition, define:
\[ +_\ma(x, y) = z ~\mbox{ if } (\exists w \in X_\ma) \left( ~ w \leq_\ma z ~ \land [0_\ma, x) E_\ma [0_\ma, w) ~ \land
~ [0_\ma, y) E_\ma [w, z)  \right) \] 
We will write $1 = 1_\ma$ for the successor of $0_\ma$ in the discrete linear order $\leq_\ma$. Then clearly 
\[ S^1(w) =  +_\ma(w, 1_\ma)  \]
For multiplication, define:
\begin{align*} \times_\ma(x,y) & = z ~\mbox{ if } \\
(\exists f  \in \bij(\ma)) & \left(\dom(f) = [0,x)_\ma ~\land (\forall w <_\ma x = \lgn(f))(~[f(w), f(w+1)) E_\ma [0,y) ~) \right).
\end{align*}
Clearly $+_\ma, \times_\ma$ act like addition and multiplication except for the fact that not all pairs of elements of $X_\ma$ will have a sum or product (if
they are too large). 
This shows $N^*_a$ is a model of Peano arithmetic without induction in which every nonzero element has a predecessor. 

Recalling that $\dom(N^*_\ma) = \bigcup_{k < \omega} [0_\ma, a^k)_\ma$, we verify bounded induction on $[0_\ma, a^k]$ for each relevant formula and each $k$. 
Let $k < \omega$ be given and let $a_* = a^k$.
For any bounded formula $\psi(x)$ in the language of arithmetic, 
possibly with parameters, which implies $x \in X_\ma$, if $M^+_1 \models \psi(0_\ma)$ then the set
$\{ x \in X_\ma : (\forall z \leq_\ma x)\psi(z) \}$ is a definable nonempty subset of $X_\ma$, so has a greatest element 
$n_\psi$. If $n_\psi \leq a_*$, $(\forall z \leq n_\psi)(\psi(z) \rightarrow \psi(S(z)))$ is well defined, 
so it is either true in $M^+_1$ (in which case we contradict $n_\psi \leq a_*$) or false in $M^+_1$ 
(in which case this instance of induction is trivially true). 

This completes the proof. 
\end{proof}

\begin{defn}
Say that $\vp(x)$ is a \emph{weakly initial formula} for $\ma$ if it is a formula in the language of $M^+_1$, possibly with 
parameters, which implies $x \in X_\ma$ and which holds on $S^k(0_\ma)$ for all $k<\omega$. 
\end{defn}

\begin{claim} \label{overspill}
Let $\cs$ be a cofinality spectrum problem, $\ma \in \ord(\cs)$, $\vp(x)$ a weakly initial formula for $\ma$. 

\begin{enumerate}
\item There is a nonstandard $a \in X_\ma$ such that $M^+_1 \models \vp(a)$.  
\item We may also require that $a < S^{-\ell}(e_\ma)$ for all $\ell < \omega$. 
\end{enumerate}
\end{claim}

\begin{proof}
For part (1), 
the set $A = \{ b \in X_\ma ~:~ (\forall c \leq_\ma b) \vp(c) \}$ contains all the standard elements of $X_\ma$, and is definable, so it has a last element $a_*$, 
which is necessarily nonstandard. 
For part (2), 
let $a$ be any element realizing the countable pre-cut described by the standard elements on the left and $\{ S^{-\ell}(a_*) : \ell < \omega \}$ 
on the right.  
\end{proof}

Alternately, Theorem \ref{x16} shows that in Claim \ref{overspill} we may replace $\vp$ with 
a consistent partial type $\Sigma(x)$ closed under conjunction, each of whose formulas $($thus, finite subtypes$)$
is weakly initial for $\ma$ and such that  $|\Sigma| < \min \{ \xp_\cs, \xt_\cs \}$. 

\begin{concl}  \label{e:powers}
Let $\cs$ be a cofinality spectrum problem, $\ma \in \ord(\cs)$. There are $a \in X_\ma$ which have weak powers in the sense 
of \ref{d:mult2}, i.e. for which the hypotheses of Lemma \ref{m-pa} are satisfied. 
\end{concl}

\begin{proof} 
Let $\vp_{exp}(x,y,z) =$  ``$x^y = z$'' be the formula from Lemma \ref{coding-lemma}, let $\vp(x) = \exists y \psi(x,x,y)$,  
and apply Claim \ref{overspill} to $\vp(x)$. 
\end{proof}

\begin{conv} \label{c:d10}
In what follows, whenever the context of a background $X_\ma$ is clear, 
we freely use the partial functions $+$, $\times$, and exponentiation from Lemma \ref{m-pa} $($or Lemma \ref{coding-lemma}$)$, 
as well as $0$, $1$, $+k$, $-k$, internal cardinality $|\cdot|$ in the sense of $\ref{c:15}$ $($really a means of internally comparing two sets$)$, 
successor $S^\ell, S^{-\ell}$ from $\ref{bc}$. 
We also adopt the more usual notation $x+y$, $x \times y$. 
\end{conv}

\section{There are no symmetric cuts strictly below $\min \{ \xp^+_\cs, \xt_\cs \}$ } \label{s:symmetric}
With the results of Sections \ref{s:sat}--\ref{s:towards} $($local saturation, basic arithmetic$)$ in hand, 
we return to the main line of our investigation: the cofinality spectrum $\mc(\cs, \xt_\cs)$. 
To Uniqueness, Theorem \ref{m2}, we now add:

\begin{lemma} \label{m12}
Let $\cs$ be a cofinality spectrum problem.
Then for all regular $\kappa$ such that $\kappa \leq \xp_\cs$, $\kappa < \xt_\cs$, we have that $(\kappa, \kappa) \notin \cts$.
\end{lemma}

\begin{proof}
Without loss of generality, we fix $\ma \in \ord(\cs)$ which is coverable, say by $\mb$.  
We prove the lemma by induction on $\kappa$. Arriving to $\kappa$,
assume for a contradiction that $(\langle a_\alpha : \alpha < \kappa \rangle, \langle b_\alpha : \alpha < \kappa \rangle)$ 
represents a $(\kappa, \kappa)$-cut in $X_\ma$. Fix a definable partial injection $g$ from $X_\mb \times X_\mb$ to $X_\mb$ 
whose domain includes $X_\ma \times X_\ma$. 
(Such a function always exists as $\mct_\ma$ may be identified with a definable subset of $X_\mb$, and ordered pairs of elements of $X_\ma$
may be represented as elements of $\mct_\ma$ of length 2. 
For definiteness, we may assume $X_\ma \subseteq X_\mb$ and use $(a,b) \mapsto (a + b + 1)^2 + a$, which will 
be defined on all pairs of elements from $X_\ma$ by the choice of $X_\mb$.) 
Define $\mct$ to be the set of 
$x \in \mct_\mb$ such that: 
\begin{enumerate}
\item[(a)] $x(n)$, when defined, is $g(a,b)$ for some $a,b \in X_\ma$.  In slight abuse of notation, 
we will denote $a$ by $x(n,0)$ and $b$ by $x(n,1)$. 
\item[(b)] $n_1 <_\mb n_2 <_\mb \lgn(x)$ implies $x(n_1,0) <_\ma x(n_2, 0) < x(n_2, 1) <_\ma x(n_1,1)$. 
\end{enumerate}
Just as in the proof of Theorem \ref{m2}, we choose a sequence $c_\alpha$, $n_\alpha$ by induction on $\alpha < \kappa$ so that:
\begin{enumerate}
\item $c_\alpha \in \mct$ and $n_\alpha = \lgn(c_\alpha)$
\item $\beta < \alpha$ implies $c_\beta \tlf c_\alpha$
\item $c_\alpha$ is below the ceiling
\item $c_\alpha(n_\alpha) = g(a_\alpha, b_\alpha)$ 
\end{enumerate}
For limit $\alpha < \kappa$, Lemma \ref{c:ceiling} applies, so we may carry out the induction.  
Having defined $\langle c_\alpha : \alpha < \kappa \rangle$, as $\kappa < \xt_\cs$ we apply Claim \ref{t-def}
to choose $c \in \mct$ such that $\alpha < \kappa$ implies $c_\alpha \tlf c$. 
Let $\bn = \lgn(c)-1$.
Then by clause (b) of the definition of $\mct$, we have that for all $\alpha < \kappa$,
\[ a_\alpha = c(n_\alpha, 0) <_\ma  c(\bn, 0) <_\ma c(\bn, 1)  <_\ma c(n_\alpha, 1) = b_\alpha. \]
This contradicts the assumption that $(\langle a_\alpha : \alpha < \kappa \rangle, \langle b_\alpha : \alpha < \kappa \rangle)$
represents a $(\kappa, \kappa)$-cut in $X_\ma$.
\end{proof}

\noindent Note that Lemma \ref{m12} generalizes Lemma \ref{c:motiv}. 

\begin{lemma} \label{treetops-sym}
Let $\cs$ be a cofinality spectrum problem and $\kappa$ a regular cardinal. If $\kappa = \xt_\cs$,
then there is a definable linear order which has a $(\kappa, \kappa)$-cut.
\end{lemma}

\begin{rmk}
It is not asserted that this definable linear order is an element of $\ord(\cs)$. 
\end{rmk}

\begin{proof}[Proof of Lemma \ref{treetops-sym}.] 
By Definition \ref{treetops}, if $\kappa = \xt_\cs$ then there is some $\mct_\ma \in \tr(\cs)$ 
which contains a strictly $\tlf_\ma$-increasing sequence $\langle c_i : i < \kappa \rangle$ 
with no upper bound. 

We construct a linear order by collapsing the tree so that the presence or absence of
$\kappa^+$-treetops (upper bounds of linearly ordered sequences of cofinality
$\kappa$) corresponds to realization or omission of symmetric $(\kappa, \kappa)$-cuts. Note that
as $X_\ma$ is linearly ordered by $\leq_\ma$, we have available a definable linear ordering on the immediate successors of any given $c \in \mct_\ma$.
To simplify notation, write $\cis(c_i, c_j)$ for the common initial segment of $c_i, c_j$.

Fix two distinct elements of $X_\ma$; without loss of generality we use $0_\ma, 1_\ma$, called $0,1$, so $0 <_\ma 1$.
Let $X$ be the set $\mct_\ma \times \{ 0, 1 \}$.
Let $<_X$ be the linear order on $X$ defined as follows:
\begin{itemize}
\item If $c = d$, then $(c, 0) <_X (c, 1)$ 
\item If $c \tlf_\ma d$ and $c \neq d$, then $(c, 0) <_X (d, 0) <_X (d, 1) <_X (c, 1)$
\item If $c, d$ are $\tlf$-incomparable, then let $e \in \mct_\ma$, $n_c, n_d \in X_\ma$ be such that
$e = \cis(c,d)$ and $e^\smallfrown n_c \tlf_\ma c$ and $e^\smallfrown n_d \tlf_\ma d$. Necessarily $n_c \neq n_d$ by definition
of $e$, so for $s, t \in \{ 0, 1 \}$ we define
\[ (c, s) <_X (d,t) \iff n_c <_\ma n_d \]
\end{itemize}
(Informally speaking, each node separates into a set of matched parentheses enclosing the cone above it.)
Then $<_X$ is definable linear order on $X$ with a first and last element. 

Recalling the definition of the sequence $\langle c_i : i < \kappa \rangle$ from the beginning of the proof, it follows that
\[ (\langle (c_i, 0) : i < \kappa \rangle, \langle (c_i, 1) : i < \kappa \rangle) \]
represents a $(\kappa, \kappa)$-cut in $(X, <_X)$. This is because it is a definable linear order on a definable set in $M^+_1$ and every element
misses the cut at some initial stage.  This completes the proof.
\end{proof}

\br
Lemma \ref{treetops-sym} shows that if we were to e.g. extend our definition of cuts in $\cts$ to include the one in Lemma \ref{treetops-sym} $($or to show that
$\Delta$ may be suitably extended, so $\ref{extend}$ applies$)$, this would give a characterization of $\xt_\cs$: 
if $\kappa = \xt_\cs$ then $\xt_\cs = \min \{ \kappa ~:~ (\kappa, \kappa) \in \cts \}$, thus 
also $\xp_\cs \leq \xt_\cs$ by definition of $\xp_\cs$. This is not necessary for our present arguments;
it is established in certain cases, e.g. $\ref{sym-cuts}$ below. The full characterization is  
proved under similar additional hypotheses (which hold in 
all the main examples of the present paper) in the authors' paper \cite{MiSh:F1361}.

\section{If $\lcf(\aleph_0, \cs) > \aleph_1$ then $\lcf(\aleph_0, \cs) \geq \min \{ \xp^+_\cs, \xt_\cs \}$} \label{s:lcf-aleph-0}

As a warm-up to Section \ref{s:main-theorems}, in the present section we will prove Theorem \ref{m5}: for a cofinality spectrum problem $\cs$, if 
$\lcf(\aleph_0, \cs) \neq \aleph_1$, then either $\lcf(\aleph_0, \cs) > \xp_\cs$ or $\lcf(\aleph_0, \cs) \geq \xt_\cs$.

\begin{conv}
For $\cs$ a cofinality spectrum problem, $\ma \in \ord(\cs)$, $a_1 <_\ma a_2$ from $X_\ma$ we write
$(a_1, a_2)_\ma$ for
\[ \{ x \in X_\ma : a_1 <_\ma x <_\ma a_2 \} \subseteq X_\ma \]
and analogously for closed and half-open intervals. 
\end{conv}

\begin{conv}
Various subsequent proofs will involve the same three types of sequences: we are given
a cut represented by $(\overline{d}, \overline{e})$ and an additional sequence of constants
$\overline{a}$. We standardize indexing as follows:
\begin{itemize}
\item $\overline{d} = \langle d_\epsilon : \epsilon < ... \rangle$
\item $\overline{e} = \langle e_\alpha : \alpha < ... \rangle$
\item $\overline{a} = \langle a_i : i < ... \rangle$
\end{itemize}
The proofs also involve building sequences $\langle c_\alpha : \alpha < ... \rangle$ of elements in a given tree, where the definition of
$c_\alpha$ generally depends on $e_\alpha$, as the index suggests.
\end{conv}

\br
We begin with a preliminary claim.  Our main case is $\kappa = \aleph_0$, but it will be useful to have the general statement.

\begin{claim} \label{m5b} 
Let $\cs$ be a cofinality spectrum problem.
Suppose that we are given:

\begin{enumerate}
\item $\ma \in \ord(\cs)$, which is coverable as a pair by $\mb$ 
\item $e <_\ma e^\prime \in X_\ma$ infinitely far apart 
\item $\kappa^+ <  \min \{ \xp_\cs, \xt_\cs \}$ 
\item $\{ a_i : i < \kappa^+ \}$ a set of pairwise distinct elements of $X_\ma$
\item $\gamma = \cf(\gamma) < \min \{ \xp_\cs, \xt_\cs \}$  
\item $\langle f_\beta : \beta < \gamma \rangle$ a sequence of definable partial injections from $X_\ma$ into $X_\ma$, 
such that for all $\beta < \gamma$,
\[ \dom(f_\beta) 
\supsetneq \{ a_i : i < \kappa^+ \} \]
\end{enumerate}

Then there exists a definable partial injection $f$ such that:
\begin{itemize}
\item[(a)] $a_i \in \dom(f)$ for all $i < \kappa^+$
\item[(b)] $\dom(f) \subseteq \dom(f_\beta)$ for all $\beta < \gamma$
\item[(c)] $\rn(f) \subseteq (e, e^\prime)_\ma$
\end{itemize}
Moreover, we may identify $f$ with an element of $X_\mb$. 
\end{claim}

\begin{proof} 
We have assumed $\ma$ is coverable as a pair by $\mb$, Definition \ref{c:pair}, which assumes we fix an 
appropriate $\bc \in \ord(\cs)$ be 
such that $X_\bc = X_\ma \times X_\ma$ and $\mb$ covers $\bc$.   
Let $\mct \subseteq \mct_\bc$ 
be the definable subtree of elements $x$ such that 
$\rn(x)$ is the graph of a partial injection from $X_\ma$ into $(e, e^\prime)_\ma$, which we denote $g_x$. 
Let $p$ be the partial type of an element $x \in \mct$ whose associated $g_x$ satisfies the conditions (a)-(c) on $f$ given in the statement of the Claim. 
By hypothesis, elements of $\mct$ may be identified with elements of $X_\mb$.  
Thus, to show that $p$ is a partial $\ord$-type in $M^+_1$ of size 
$\kappa^+ + |\gamma| <  \min \{ \xp_\cs, \xt_\cs \}$, it will suffice to show that $p$ is finitely satisfiable in $\mct$.

Consider an arbitrary but fixed finite subset $p_* \subseteq p$ involving only $\{ f_\beta : \beta \in \sigma \}$ and
$\{ a_i : i \in \tau \}$ for some fixed $\sigma \in [\gamma]^{<\aleph_0}$ and $\tau \in [\kappa^+]^{<\aleph_0}$.  
By assumption (6) of the Claim,
$\{ a_i : i \in \tau \} \subseteq \bigcap_{\beta \in \sigma} \dom(f_{\beta})$. List this set as $\{ a_{i_k} : k < |\tau| \}$.
By assumption (2), there exist $|\tau|$ distinct elements $\{ b_{k} : k < |\tau| \}$ of the interval $(e, e^\prime)_\ma$.
Let $c_* = \langle a_0, b_0 \rangle^\smallfrown \cdots ^\smallfrown\langle a_{i_{|\tau|-1}}, b_{|\tau|-1}\rangle$.  
Note that this function will be an element of $\mct$ by closure under concatenation and the hypothesis that $\ma$ is nontrivial 
(thus also $\bc$ is nontrivial, by choice of $\bc$ in \ref{c:pair}). Thus, $c_*$ realizes $p_*$. 
This completes the proof that $p$ is finitely satisfiable in $\mct$.

By $\ord$-saturation, Theorem \ref{x16} above, $p$ is realized by some $c$. Let $g_c$ be the associated definable partial function 
described in the first paragraph of the proof. Then $f = g_c$ is as required. 
\end{proof}

Note that in the proof just given, we did not require the functions $f_\beta$ to be uniformly definable nor to  
belong to $\mct_\bc$. 

\begin{lemma}  \label{m5-subclaim}
Let $\cs$ be a cofinality spectrum problem and $\aleph_1 < \lambda = \cf(\lambda) \leq \xp_\cs$, $\lambda < \xt_\cs$.
If
\[ ~\lcf(\aleph_1, \cs) \geq \aleph_2 ~~\mbox{ and } ~\lcf(\aleph_0, \cs) \geq \lambda \]
then $\lcf(\aleph_0, \cs) > \lambda$.
\end{lemma}

\begin{proof} Assume for a contradiction that $\lcf(\aleph_0, \cs) = \lambda$. 
Note that the hypotheses of the claim guarantee that $\aleph_1 < \min \{ \xp_\cs, \xt_\cs \}$. 
Since we may choose any order in which to work, let $\ma \in \ord(\cs)$ be coverable as a pair and let $\mb \in \ord(\cs)$ be a cover for it.
Without loss of generality, we may choose
\[ (\langle d_\epsilon : \epsilon < \omega \rangle, \langle e_\alpha : \alpha < \lambda \rangle) \]
representing an $(\aleph_0, \lambda)$-cut in $X_\ma$ such that $d_0$ is the $\leq_\ma$-least element of $X_\ma$ and  
$d_{\epsilon+1} = S^1(d_\epsilon)$ for $\epsilon < \omega$, whereas the elements of $\overline{e}$ are 
infinitely far apart: $S^k(e_{\alpha + 1}) < e_\alpha$ for each $\alpha < \lambda$ and $k < \omega$. 
This can be done as follows. Let  $\langle d_\epsilon : \epsilon <\omega \rangle$ be as specified; 
by choice of $X_\ma$, $X := X_\ma \setminus \{ d_\epsilon : \epsilon < \omega \} \neq \emptyset$
so by uniqueness there is a sequence $\overline{e} = \langle e_\alpha : \alpha < \lambda \rangle$ coinitial in $X$. If necessary,
replace $\overline{e}$ with $\langle e_{\omega \cdot \alpha} : \alpha < \lambda \rangle$ (ordinal product) to achieve the spacing. 

To complete the preliminaries, choose a sequence
$\langle a_i : i < \aleph_1 \rangle$ of distinct elements of $X_\ma$, which exists by Claim \ref{o:incr} 
since a fortiori $\aleph_1 \leq \xp_\cs$.

Let $\bd \in \ord(\cs)$ be such that
$X_\bd = X_\ma \times X_\ma \times X_\mb \times X_\mb$.
Let $\mct_2$ be the subtree of $\mct_\bd$ consisting of elements $x$ such that: 
\begin{enumerate}
\item[(a)] $x(n,0) <_\ma x(n,1)$ are from $X_\ma$
\item[(b)] $x(n,2)$ is a subset of $X_\ma$, equal to $\dom(x(n,3))$
\item[(c)] $x(n,3)$ is a 1-to-1 function from $x(n,2)$ into the interval $(x(n,0), x(n,1))_\ma$
\item[(d)] $n_1 <_\bd n_2$ in $\dom(x)$ implies $x(n_2,1) <_\ma x(n_1, 0)$ and $x(n_2, 2) \subsetneq x(n_1, 2)$
\end{enumerate}

\br
Recalling the choice of $\langle e_\alpha : \alpha < \lambda \rangle$, 
we will now choose $c_\alpha \in \mct_2$, $n_\alpha \in X_\bc$
by induction on $\alpha < \lambda$ such that:
\begin{enumerate}
\item[(i)] $\beta < \alpha$ implies $c_\beta \tlf c_\alpha$
\item[(ii)] $n_\alpha = \maxdom(c_\alpha))$
\item[(iii)] each $c_\alpha$ is below the ceiling
\item[(iv)] $e_{\alpha+1} \leq_\ma c_\alpha(n_\alpha, 0) <_\ma c_\alpha(n_\alpha, 1) \leq_\ma e_\alpha$
\item[(v)] $a_i \in c_\alpha(n_\alpha, 2)$ for each $i < \aleph_1$. 
\end{enumerate}
Let us carry out the induction: 

\emph{For $\alpha = 0$:} Apply Claim \ref{m5b} with ${\kappa^+ = \aleph_1}$, ${\gamma = 0}$, $e = e_1$, $e^\prime = e_0$ to obtain 
a suitable partial injection $g$, which can be coded as element of $X_\mb$, and whose domain contains $\{ a_i : i < \aleph_1 \}$. 
Let $c_0 = \langle e_1, e_0, \dom(g), g \rangle$ and let  $n_0 = 0$.  

\br
\emph{For $\alpha = \beta + 1$:} 
Likewise, apply Claim \ref{m5b} with ${\kappa^+ = \aleph_1}$, ${\gamma = 1}$, $f_0 = c_\beta(n_\beta, 3)$, $e = e_{\alpha+1}$, $e^\prime = e_{\alpha}$ to 
obtain a definable partial injection $g \in X_\mb$ whose domain includes $\{ a_i : i < \aleph_1 \}$
and whose range is included in the interval $(e_{\alpha+1}, e_\alpha)$. 
Let $c_\alpha = {c_\beta}^\smallfrown\langle e_\alpha, e_\beta, \dom(g), g \rangle$. Let $n_\alpha = n_\beta+1$.

\br
\emph{For $\alpha < \lambda$ limit:}
As $\cf(\alpha) \leq \alpha < \lambda$ thus $\cf(\alpha) < \min \{ \xp_\cs, \xt_\cs \}$, by Lemma \ref{c:ceiling} there is $c \in \mct_2$ such that
$\beta < \alpha$ implies $c_\beta \tlf c$ and $c$ is below the ceiling. 
Then 
\[ \{ n : n <_\bc \lgn(c), ~M^+_1 \models e_\alpha \lneq_\ma c(n,0) ~ \} \]
is a bounded nonempty subset of $X_\bc$ which contains $n_\beta = \maxdom(c_\beta))$ for all $\beta < \alpha$.
Let $n_*$ be its maximal element (necessarily $n_* > n_\gamma$ for all $\gamma \leq \beta$). 
Once more, let $g$ be the definable partial injection given by Claim \ref{m5b} when we use: 
\begin{enumerate}
\item $e_{\alpha +1}, e_\alpha$ for $e, e^\prime$ in $\ref{m5b}(2)$
\item $\kappa^+ = \aleph_1$ in $\ref{m5b}(3)$
\item $\{ a_i : i < \aleph_1\}$ in $\ref{m5b}(4)$
\item $\cf(\alpha)$ for $\gamma$ in $\ref{m5b}(5)$, noting that $\cf(\alpha) < \lambda$ thus $\cf(\alpha) < \min \{ \xp_\cs, \xt_\cs \}$ 
\item $\langle c(n_\zeta,3): \zeta \in X \rangle$ for the sequence of functions in $\ref{m5b}(6)$, where $X$ is any sequence cofinal in $\alpha$. 
\end{enumerate} 
Let $c_\alpha = \langle e_{\alpha +1}, e_\alpha, \dom(g), g \rangle$ and let $n_{\alpha+1} = \maxdom(c_\alpha))$. 

This completes the inductive construction of the sequence $\langle c_\alpha : \alpha < \lambda \rangle$.  
As $\lambda < \xt_\cs$, by Claim \ref{t-def} there is $c \in \mct_2$ such that $c_\alpha \tlf c$ for all $\alpha < \lambda$. 
Let $m = \maxdom(c))$. 
By the choice of $(\bar{d}, \bar{e})$, necessarily both $c(m, 0) <_\ma c(m, 1)$ are elements of the sequence $\overline{d}$.
For each $i < \aleph_1$, the set of $m^\prime \leq_\bc m$ such that $M^+_1 \models a_i \in c(m^\prime, 2)$ is definable, 
bounded and nonempty, so let $m_i \leq_\bc m$ be its maximal element. By our construction, $n_\alpha \leq_\bc m_i$ for each 
$i < \aleph_1$ and each $\alpha < \lambda$.  
Thus, for each $i < \aleph_1$,  also $c(m_i, 0) <_\ma c(m_i, 1)$ are elements of the countable sequence $\overline{d}$. By the pigeonhole
principle, there is $d_{\star} \in \overline{d}$ such that {$\uu := \{ i < \aleph_1 : c(m_i, 1) = d_{\star} \}$} is uncountable. 
Note that by our choice of $\bar{d}$, the element $d_*$ is finite, so in particular the set $[0, d_*]_\ma$ 
of predecessors of $d_*$ is finite. There are two cases. 
If there exists an uncountable (or just: infinite) $\uu^\prime \subseteq \uu$ on which the function $i \mapsto m_i$ is constant, let 
$m_*$ denote this constant value. In this case $c(m_*, 3)$ is a definable injection whose domain includes $\{ a_i : i \in \uu^\prime \}$ 
and whose range is contained in the finite set of predecessors of $d_*$, a clear contradiction. 
Otherwise, we may find an infinite $\uu^\prime \subseteq \uu$ on which the function $i \mapsto m_i$ is one-to-one. 
Let $g$ be the function given by $a_i \mapsto c(m_i, 3)(a_i)$. Then $g$ is one-to-one on $\{ a_i : i \in \uu^\prime \}$   
because $k <_\bc l$ implies $c(l, 0) <_\ma c(l,1) <_\ma c(k,0) <_\ma c(k,1)$.  Moreover, $g(a_i) \leq_\ma d_*$ for all $i \in \uu^\prime$. 
Let $\uu^{\prime\prime}$ be any subset of $\uu^\prime$ of cardinality $d_* + 2$. Then $g \rstr \{ a_i : i \in \uu^{\prime\prime} \}$ 
is a definable injection of a finite set of size $d_*+2$ into a finite set of size $d_*+1$, a clear contradiction. This completes the proof. 
\end{proof}

We arrive to the main result of the section, Theorem \ref{m5}, whose statement suggests that to show $(\aleph_0, \lambda) \notin \mc(\cs, \xt_\cs)$
the hardest cut to rule out is the cut $(\aleph_0, \aleph_1)$. This intuition will be borne out in later cases. 

\begin{theorem} \label{m5}
Let $\cs$ be a cofinality spectrum problem. 
If $\lcf(\aleph_0, \cs) \neq \aleph_1$, then either $\lcf(\aleph_0, \cs) > \xp_\cs$ or $\lcf(\aleph_0, \cs) \geq \xt_\cs$.
\end{theorem}

\begin{proof} 
We will prove that if $\lcf(\aleph_1, \cs) \geq \aleph_2$ then $\lcf(\aleph_0, \cs) > \lambda$ for all regular $\lambda$ such that
$\aleph_1 < \lambda \leq \xp_\cs$ and $\lambda < \xt_\cs$.  
This is equivalent to the statement of the theorem because $\lcf(\aleph_1, \cs) \geq \aleph_2$ precisely when $(\aleph_0, \aleph_1) \notin \mc(\cs, \xt_\cs)$, 
recalling Theorem \ref{m2} (Uniqueness) and Lemma \ref{m12} (Anti-Symmetry). 

We prove the claim by induction on $\lambda \geq \aleph_2$. 
Recall that by Corollary \ref{lcf-cor}, for regular cardinals $\gamma_1, \gamma_2 < \xt_\cs$,  
$\lcf(\gamma_1, \cs) = \gamma_2$ iff  $(\gamma_1, \gamma_2) \in \cts$ iff $\lcf(\gamma_2, \cs) = \gamma_1$.

\br
\emph{Base case: $\lambda = \aleph_2$}. 
First, $(\aleph_0, \aleph_0) \notin \cts$ by Lemma \ref{m12}. 
Second, by hypothesis, $\lcf(\aleph_1, \cs) \geq \aleph_2$. So by uniqueness, Theorem \ref{m2}, there is
$\kappa \geq \aleph_2$ such that $(\aleph_1, \kappa) \in \cts$ \emph{thus} $(\aleph_0, \aleph_1) \notin \cts$.
Thus by Corollary \ref{lcf-cor}, $\lcf(\aleph_0, \cs) \geq \aleph_2$. We then apply Lemma \ref{m5-subclaim} 
to conclude $\lcf(\aleph_0, \cs) > \aleph_2$.

\br
\emph{Inductive case.} In both the successor and limit stages, by inductive hypothesis, $\lcf(\aleph_0, \cs) > \kappa$ 
for all $\kappa < \lambda$. Thus $\lcf(\aleph_0, \cs) \geq \lambda$.
By Lemma \ref{m5-subclaim}, $\lcf(\aleph_0, \cs) > \lambda$ as desired.
\end{proof}

\section{There are no asymmetric cuts strictly below $\min\{ \xp^+_\cs, \xt_\cs \}$} \label{s:main-theorems}

In this section we substantially generalize the results of the previous section 
to prove Theorem \ref{last-cut} on asymmetric cuts.   
We assume $\xp_\cs < \xt_\cs$ (though we will 
state where this is used), as our background goal is to show that $\mc(\cs, \xt_\cs) = \emptyset$, which is clearly true if
$\xp_\cs \geq \xt_\cs$ by definition of $\xp_\cs$. 
We also use upgraded trees, following \ref{p:trees}, without further comment. 

\begin{disc} \label{big-picture} \emph{(Strategy for \ref{last-cut})}
\emph{Theorem \ref{m5} above, an early prototype of Theorem \ref{last-cut}, proved that under certain conditions $\lcf(\aleph_0) > \lambda$. The key inductive 
step in the proof involved assuming, for a contradiction, that 
equality held. We fixed a representative $(\bar{d}, \bar{e})$ of an $(\aleph_0, \lambda)$-cut in some $X_\ma$, where the elements of $\overline{e}$ were infinitely far apart 
and the elements of $\overline{d}$ were all finite, and fixed a distinguished uncountable set $\{ a_i : i < \aleph_1 \}$.  
We considered a tree $\mct$ where $x \in \mct$ and $n \leq \maxdom(x))$ meant that $x(n)$ gave essentially the data of a (uniformly) definable partial injection 
along with its domain, which included $\{ a_i : i < \aleph_1 \}$, and an interval of $X_\ma$ containing its range. As $n \leq \maxdom(x))$ grew the 
domains of these partial functions formed a decreasing sequence and the intervals bounding the range were pairwise disjoint and moved towards the left. 
We built a path $\langle c_\alpha : \alpha < \lambda \rangle$ through this tree such that evaluated 
on the maximal element of its domain, $c_\alpha$ gave an injection into the interval $(e_{\alpha+1}, e_\alpha)$. Informally, we were carrying the uncountable set 
$\{ a_i : i < \aleph_1 \}$ left along the right-hand side of the cut. By treetops, this sequence $\langle c_\alpha : \alpha < \lambda \rangle$ had an upper bound $c$. 
Since $(\overline{d}, \overline{e})$ was a cut, the elements $\{ a_i : i < \aleph_1 \}$ each overspilled into the domains of partial injections 
into intervals bounded on the right by elements of $\overline{d}$. 
From this one could find a partial injection of some large set into a finite set for the contradiction.}  

\emph{The main result of this section, Theorem \ref{last-cut}, significantly generalizes that construction to show that $\lcf(\kappa) > \lambda$ whenever $\kappa < \lambda \leq \xp_\cs <\xt_\cs$.  A first issue in this more general case is that when $\kappa > \aleph_0$,  it is not a priori sufficient to carry a set of size $\kappa^+$ into a $\kappa$-indexed sequence to obtain a contradiction. After all, elements of the $\kappa$-indexed sequence may be far apart (consider e.g. the diagonal embedding of $\kappa$ in a regular ultrapower of $(\kappa, <)$).  
Our control on size will need to come from internal cardinality in the sense of \ref{c:15}. A second issue arises in the case where $\kappa^+ = \lambda$, the most subtle case of all. (Note that in \ref{m5}, we had to assume the case $(\aleph_0, \aleph_1)$ did not occur.) 
Our mechanism there for carrying the set of size $\kappa^+$ through limit stages, \ref{m5b}, required $\kappa^+ < \xp_\cs$ in order 
to apply $\ord$-saturation (also $\kappa^+ <\xt_\cs$, but recall we are assuming $\xp_\cs < \xt_\cs$). 
Clearly, this is not satisfied when $\kappa^+ = \lambda = \xp_\cs$. In Theorem \ref{last-cut}, this is solved by gradually growing the size of the set we carry (called there 
$\{ y_{\beta + 1} : \beta < \alpha \cap \kappa^+ \}$) so that it has size $\leq \kappa$ at each inductive stage $\alpha$
when $\kappa^+ = \lambda$. It will grow all the way to $\kappa^+$ if $\kappa^+ < \lambda$, but then it is not a problem.  
The construction of the present section is built to coordinate these various requirements.} 
\end{disc}

We begin by fine-tuning the construction of cuts. The idea of Claim \ref{o:seq} is that what was essential about 
$( \overline{d}, \overline{e} )$ in Lemma \ref{m5-subclaim} was that the elements of $\overline{d}$ 
increased in internal cardinality in the sense of \ref{c:15} and those of $\overline{e}$ were sufficiently spaced. Here we show how to specify a minimum spacing for elements
in the sequence representing the cut without losing the fact that (from the point of view of $\cs$) the cardinality genuinely grows. 
Recall that addition and multiplication are available following Section \ref{s:towards}. 

\begin{claim} \label{o:seq}
Let $\cs$ be a cofinality spectrum problem, $\ma \in \ord(\cs)$, $\kappa \leq \xp_\cs$, 
$(\kappa, \lambda) \in \mc(\cs, \xt_\cs)$.
Let $f: X_\ma \rightarrow X_\ma$ be multiplication by 2. 
Then we may choose sequences
$\langle d_\epsilon : \epsilon < \kappa \rangle$, $\langle e_\alpha : \alpha < \lambda \rangle$ of elements of 
$X_\ma$ such that
\[ ( \langle d_\epsilon : \epsilon < \kappa \rangle, \langle e_\alpha : \alpha < \lambda \rangle ) \]
represents a $(\kappa, \lambda)$-cut, and moreover:
\begin{enumerate}
\item  $\alpha < \lambda$ implies $f(e_{\alpha+1}) <_\ma e_\alpha$
\item  for each $\epsilon < \kappa$ there is $\delta = \delta(\epsilon)$, $\epsilon < \delta < \kappa$ such that
\[ M^+_1 \models |d_\epsilon| <^\cs |d_\delta| \]
where $|d| := | \{ x \in X_\ma : x <_\ma d \}|$ and cardinality is meant the sense of \ref{c:15}.
\end{enumerate}
\end{claim}

\begin{rmk}
As is clear, the proof holds for more general functions $f$.  Recall that $\ref{c:15}$ 
describes an internal relation comparing certain definable sets by means of partial functions available in 
specific trees. 
\end{rmk}

\begin{proof}[Proof of Claim \ref{o:seq}.]
Note that $(\kappa, \lambda) \in \mc(\cs, \xt_\cs)$ implies $\kappa + \lambda < \xt_\cs$ by definition. 
Let $\mct_4$ be the definable subtree of $\mct_\ma$ consisting of elements $x$ such that: 
\begin{itemize}
\item[(i)] $n_1 <_\ma n_2 <_\ma \lgn(x)$ implies $f(x(n_1)) <_\ma x(n_2)$
\item[(ii)] $n_1 <_\ma n_2 <_\ma \lgn(x)$ implies $|x(n_1)| <^\cs |x(n_2)|$ in the sense of \ref{c:15}. 
\end{itemize}
Note that this tree is nonempty and contains arbitrarily long finite branches.
For $a \in X_\ma$, say that \emph{cardinality $f$-grows above $a$} to mean that for each $n<\omega$ there are 
$x_0, \dots, x_n \in X_\ma$ such that $x_0 = a$, $f(x_\ell) <_\ma x_{\ell+1}$ for $\ell < n$, and 
\[ |a| <^\cs |x_1| <^\cs \cdots <^\cs |x_n|.\] 
Clearly this holds for all finite successors of $0_\ma$. 
By induction on $\epsilon < \kappa$ we choose $c_\epsilon \in \mct_4$, $n_\epsilon = \lgn(c_\epsilon) - 1 \in X_\ma$, 
and $d_\epsilon \in X_\ma$ so that:
\begin{enumerate}
\item[(a)] $\beta < \epsilon$ implies $c_\beta \tlf c_\epsilon$
\item[(b)] $c_\epsilon(n_\epsilon) = d_\epsilon$
\item[(c)] $c_\epsilon$ is below the ceiling 
\item[(d)] for each $n \leq_\ma n_\epsilon$, cardinality $f$-grows above $c_\epsilon(n)$. 
\end{enumerate}

The proof is similar to Theorem \ref{m2} (Uniqueness), with an adjustment for (d), however we will give the argument for completeness. 

For $\epsilon = 0$: let $c_0 = \langle 0_\ma \rangle$ and let $n_0 = 0_\ma$. As $\ma$ is nontrivial (so contains all the 
finite successors of $0_\ma$) and $f$
is multiplication by $2$, (d) is satisfied. 

For $\epsilon = \beta + 1$: since $c_\beta$ is below the ceiling, concatenation is possible. By inductive hypothesis (d), 
we may choose $d_\epsilon \in X_\ma$ such that $f(c_\beta(n_\beta)) <_\ma d_\epsilon$, $|f(c_\beta(n_\beta))| <^\cs |d_\epsilon|$ and
cardinality $f$-grows above $d_\epsilon$. Let $c_\epsilon = {c_\beta}^\smallfrown \langle d_\epsilon \rangle$, 
and let $n_\epsilon = n_\beta + 1$, and the inductive hypotheses will be preserved. 

For $\epsilon < \kappa$ limit: As $\cf(\epsilon) < \min \{ \xp_\cs, \xt_\cs \}$, apply Lemma \ref{c:ceiling} 
to obtain $c \in \mct_4$ such that $\beta < \epsilon$ implies $M^+_1 \models c_\beta \tlf c$, and $c$ is below the ceiling.
Let $d_\infty = c(\lg(c)-1)$.  
The key point is to ensure (d). As $c \in \mct_4$, (ii) holds. For each $k<\omega$, let 
$m_k = \max \{ n <_\ma \lg(c) : c(n) \leq_\ma f^{-k} (d_\infty) \}$. Recalling (i) in the definition of $\mct_4$ 
and the fact that $\langle c_\beta : \beta < \epsilon \rangle$ is a strictly increasing
sequence below $c$, each $m_k$ is well defined and 
\[ ( \langle n_\beta : \beta < \epsilon \rangle, \langle m_k : k < \omega \rangle ) \]
is a pre-cut in $X_\ma$, and cannot be a cut without contradicting the definition of $\xp_\cs$.  
Let $n_{**} \in X_\ma$ realize this pre-cut. Let  $c_\epsilon = {c\rstr_{n_{**}+1}}$, $n_\epsilon = n_{**}$ and
$d_\epsilon = c_\epsilon(n_\epsilon)$. Clearly $c_\epsilon$ satisfies (d). 
This completes the inductive choice of the sequence.

\br

Once the sequence has been constructed, $\kappa < \xt_\cs$
implies that $\langle c_\epsilon : \epsilon < \kappa \rangle$ has an upper bound $c$ in $\mct_4$ by Claim \ref{t-def}.  
We first choose a cut in the domain of $c$, as follows.
Consider the sequence $\langle n_\epsilon := \lgn(c_\epsilon)-1 ~:~ \epsilon < \kappa \rangle$ in $X_\ma$. 
By Theorem \ref{m2} (Uniqueness) and the assumption that $(\kappa, \lambda) \in \cts$, 
there is a sequence $\langle m_\alpha : \alpha < \lambda \rangle$ of elements of $\{ n \in X_\ma : n <_\ma \lgn(c) \}$ 
such that $(\langle n_\epsilon : \epsilon < \kappa \rangle, \langle m_\alpha : \alpha < \lambda \rangle)$
represents a $(\kappa, \lambda)$-cut. 

Finally, let the sequence $\langle e_\alpha : \alpha < \lambda \rangle$ in $X_\ma$ given by
$e_\alpha = c(m_\alpha)$. Then $\langle d_\epsilon: \epsilon < \kappa \rangle$, 
$\langle e_\alpha : \alpha < \lambda \rangle$ satisfy our requirements, by the definition of $\mct_4$.
\end{proof}

\begin{fact} \label{about-g}
There exists a symmetric function $g: \kappa^+ \times \kappa^+ \rightarrow \kappa$ such that
\[ \uu \in [\kappa^+]^{\kappa^+} \mbox{ implies } \sup (\rn(g\rstr \uu \times \uu) ) = \kappa. \]
\end{fact}

For instance, a symmetric function $g$ has the desired property if whenever $\alpha < \alpha^\prime < \beta$, 
$g(\alpha, \beta) \neq g(\alpha^\prime, \beta)$.

We now prove the main result of this section:

\begin{theorem} \label{lc} \label{last-cut}
Let $\cs$ be a cofinality spectrum problem. Suppose that $\kappa$, $\lambda$ are regular and
$\kappa < \lambda = \xp_\cs < \xt_\cs$. Then $(\kappa, \lambda) \notin \mc(\cs, \xt_\cs)$.
\end{theorem}

\begin{proof}
We suppose for a contradiction that $\lcf(\kappa) = \lambda$. 
Note that necessarily $\kappa + \lambda$ is minimal among $(\kappa, \lambda) \in \mc(\cs, \xt_\cs)$ by
definition of $\xp_\cs$.  
Let $\ma \in \ord(\cs)$ be coverable as a pair by $\ma^\prime$ in the sense of \ref{cov-pair} (in this proof we will write 
$\ma \times \ma$ to refer to the $\bc$ of that definition).
Let
\[ (\langle d_\epsilon : \epsilon < \kappa \rangle, \langle e_\alpha : \alpha < \lambda \rangle) \]
be a representation of a $(\kappa, \lambda)$-cut in $X_\ma$ given by Claim \ref{o:seq}. 
To complete the preliminaries, fix $g: \kappa^+ \times \kappa^+ \rightarrow \kappa$ given by Fact \ref{about-g}.
This is an outside function, which will help in the proof.

\bbr
We now define an appropriate tree.  
Let $\mb \in \ord(\cs)$ be such that $X_\mb = X_\ma \times X_\ma \times X_\ma \times X_{\ma^\prime} \times X_{\ma^\prime} \times X_{\ma^\prime}$. 
We will work in the definable subtree $\mct_6$ of $\mct_\mb$ consisting of elements $x \in \mct_\mb$ such that:

\begin{enumerate}
\item[(a)] $n_1 <_\mb n_2 <_\mb \lgn(x)$ implies
\[  x(n_2, 0) \leq_\ma x(n_2, 1) <_\ma x(n_2,2) <_\ma x(n_1, 0) \] 
\item[(b)] $x(n,3)$ is a nonempty subset of $X_\ma$ and $|x(n,3)| \leq^\cs |X_\ma|/2$ in the sense of \ref{c:15}, 
i.e. its cardinality is $\leq^\cs$ that of its complement
\item[(c)] $x(n,4)$ is a symmetric $2$-place function with domain $x(n,3) \times x(n,3)$ and range $\subseteq X_\ma$, 
which we call a distance estimate function 
\item[(d)] $x(n,5)$ is a 1-to-1 function from $x(n,3)$ to the interval $(x(n,1), x(n,2))_{\ma}$ such that:
\[ a \neq b \in x(n,3) \mbox{ implies }  | x(n,4)(a,b)| \leq^\cs | x(n,5)(a) - x(n,5)(b) |.  \]
(Recall that for any $z \in X_\ma$, the internal cardinality $|z|$ is identified with $|[0, z]_\ma|$ following \ref{o:seq}(2) above.)
\item[(e)] if $n_1 <_\mb n_2 <_\mb \lgn(x)$ and $a,b$ in $X_\ma$ are such that
\[ (\forall m) \left( (n_1 \leq_\mb m \leq_\mb n_2) \rightarrow \{ a,b \} \subseteq x(m,3) \vspace{3mm}\right) \]
then $x(n_1, 4)(a,b) = x(n_2, 4)(a,b)$.
\end{enumerate}
This completes the description of $\mct_6$. 
(Informally, given $n \leq \maxdom(x)) \in \mct_6$, $x(n,0)$ is a marker moving left towards the cut and $x(n,5)$ is 
an injection from $x(n,3)$ into the interval $(x(n,1), x(n,2))_\ma$. The new point is a distance estimate function 
$x(n,4)$:  we ask $x(n,5)$ to respect 
the internal spacing given by this function, and we ask that as $n$ grows, the distance estimate remains unchanged 
on pairs of elements which remain continuously in the sequence of domains $x(n,3)$.) 

\bbr

In the next, core part of the proof, we will choose $c_\alpha \in \mct_6$, $n_\alpha = \maxdom(c_\alpha))$ by induction on $\alpha < \lambda$. 
When $\alpha$ is a successor we will also choose $y_\alpha$.  
Our inductive hypothesis is as follows. 

\begin{itemize}
\item For all $\alpha < \lambda$, we ensure that:

\begin{enumerate}
\item $\beta < \alpha$ implies $c_\beta \tlf c_\alpha$
\item $\beta < \alpha$ implies 
\[ e_{\alpha+1} \leq_\ma c_\alpha(n_\alpha, 0) <_\ma c_\alpha(n_\alpha, 1) <_\ma c_\alpha(n_\alpha, 2) <_\ma e_{\beta+1} \]
and if $\alpha = \beta +1$, then in addition $c(n_\alpha, 0) = e_{\alpha + 1}$.
\item For all $\gamma < \min\{ \alpha, \kappa^+\}$, 
\begin{enumerate}
\item $y_{\gamma + 1} \in c_\alpha(n_\alpha, 3)$
\item $(\forall m)[n_{\gamma + 1} \leq_\mb m \leq_\mb n_\alpha \rightarrow y_{\gamma + 1} \in c_\alpha(m,3)]$
\item for all $\zeta + 1 < \gamma + 1$ and for all $m$ such that $n_{\gamma + 1} \leq_\mb m \leq_\mb n_\alpha$,
\[ x(m,4)(y_{\zeta + 1}, y_{\gamma + 1}) = d_{g(\zeta + 1, \gamma + 1)}. \]
\end{enumerate}
\end{enumerate}

\item When $\alpha = \beta + 1 < \kappa^+$, then in addition we choose $y_\alpha = y_{\beta + 1}$ so that the following are true:

\begin{enumerate}[resume]
\item $y_{\beta + 1} \in X_\ma \setminus \{ y_{\gamma + 1} : \gamma < \beta \}$
\item $y_{\beta + 1} \in c_\alpha(n_\alpha, 3)$
\item $y_{\beta + 1} \notin c_\beta(n_\beta, 3)$
\item \label{rep} $|c_\alpha(n_\alpha, 3)| \leq^\cs |X_\ma|/2$ 
\item for all $\gamma + 1 < \beta +1$ and all $n$ such that $n_{\gamma + 1} \leq_\mb n \leq_\mb n_\alpha$,
\[  x(n,4)(y_{\gamma + 1}, y_{\beta + 1}) = d_{g(\gamma + 1, \beta + 1)}. \]
\end{enumerate}
\end{itemize}
Note that $d_{g(\beta, \gamma)}$ functions here as a constant; $g$, set at the beginning of the proof, is an outside function not mentioned by the definition of $\mct_6$.
Condition (\ref{rep}) is part of the definition of $\mct_6$, repeated for clarity. 

This completes the statement of the inductive hypothesis; let us now carry out the induction. 

\bbr
\noindent \emph{For the case $\alpha = 0$}: Trivial.

\br

\noindent \emph{For the case $\alpha = \beta + 1$, when in addition $\alpha < \kappa^+$}:
If $\alpha = \beta + 1 < \kappa^+$, then we first define $y_\alpha = y_{\beta + 1}$.
By inductive hypothesis, $M^+_1 \models c_\beta \in \mct_6$.
Thus by (b) in the definition of $\mct_6$, $X_\ma \setminus c_\beta(n_\beta, 3) \neq \emptyset$.
Choose $y_{\beta + 1} \in X_\ma \setminus c_\beta(n_\beta, 3)$.  
Note that we choose $c_\alpha(n_\alpha, 3)$ below and will ensure there that it remains small enough; in particular,
it is irrelevant whether $c_\beta(n_\beta, 3) \cup \{ y_{\beta + 1} \}$ has size no larger than its complement in $X_\ma$. 

Then conditions (4) and (6) of the inductive hypothesis hold, so by condition (e) of the definition of $\mct_6$ we will be allowed to freely choose the value of 
$c_\alpha(n_\alpha, 4)$ on any pair which includes $y_{\beta + 1}$. Label this (**) for later reference.

Having defined $y_{\beta + 1}$, continue as in the general successor step: 

\br

\noindent \emph{For the case $\alpha = \beta + 1$ for arbitrary $\alpha < \lambda$}:
We now assume that $y_{\beta + 1}$ has been chosen for all $\beta < \min \{ \alpha, \kappa^+ \}$
and continue the proof assuming only
$\alpha = \beta + 1 < \lambda$.

The key point at this step is to define $c_\alpha(n_\alpha, \ell)$ for $\ell < 6$. The nontrivial cases are
$\ell = 3,4,5$. We will do this by showing it may be expressed as a consistent partial $\ord$-type, and then applying Theorem \ref{x16}, local saturation. 
Recall that we are assuming $\xp_\cs < \xt_\cs$.

\br
Let $p(x_0, x_1, x_2, x_3, x_4, x_5)$ be the partial type stating the following. (Note that 
$x_0, x_1, x_2$ are determined by $x_3, x_4, x_5$. The $e_\alpha$ are from the background cut.) 

\begin{quote}
\begin{enumerate}
\item[(p.1)\hx] $x_0, x_1, x_2$ are elements of $X_\ma$ and $x_3, x_4, x_5$ are elements of $X_{\ma^\prime}$
\item[(p.2)\hx] $x_0 = e_{\alpha + 1} \leq_\ma x_1 <_\ma x_2 \leq_\ma e_\alpha$
\item[(p.3)\hx] $x_3 = \dom(x_4) \subseteq X_\ma$ and $|x_3| \leq^\cs |X_\ma|/2$
\item[(p.4)\hx] $x_4$ is a definable symmetric 2-place function from $x_3$ to $X_a$
\item[(p.5)\hx] $x_5$ is a definable injection from $x_3$ into the interval $(e_{\alpha + 1}, e_\alpha)_\ma$ such that
\[ a \neq b \in x_3 \rightarrow  | x_4(a,b) | \leq^\cs | x_5(a) - x_5(b) | \]
\item[(p.6)\hx] $x_1 = \min (\rn(x_5))$, $x_2 = \max (\rn(x_5))$
\item[(p.7)\hx] if $a,b \in c_\beta(n_\beta, 3) \cap x_3$ then $x_4(a,b) = c_\beta(n_\beta, 4)(a,b)$.
\end{enumerate}
\end{quote}

\br

{For $\gamma < \min \{ \alpha, \kappa^+ \}$ 
we add:}
\begin{quote}
\begin{enumerate}[resume]
\item[$($p.8$)_\gamma$\hz] ~\hspace{2mm} \begin{center} $y_{\gamma + 1} \in x_3$. \end{center}
\end{enumerate}
\end{quote}
\br

For $\zeta < \gamma < \min \{ \alpha, \kappa^+ \}$
we add:
\begin{quote}
\begin{enumerate}[resume]
\item[$($p.9$)_{\zeta, \gamma}$] ~\hspace{2mm}  \begin{center} $x_4(y_{\zeta+1}, y_{\gamma + 1}) = d_{g(\zeta + 1, \gamma + 1)}$. \end{center}
\end{enumerate}
\br
(Note: (p.9) is legitimate by the observation (**) from the case ``{$\alpha = \beta + 1, \alpha < \kappa^+$}'' when the pair includes 
$y_{\beta + 1}$, and by (1) of the inductive hypothesis for all other pairs of $y$s.) 
\end{quote}
\br

\noindent To show that $p$ is an $\ord$-type in the sense of \ref{x15}-\ref{x16}, note that first,
$p$ depends on the parameters $\{ e_\alpha, e_{\alpha + 1}, c_\beta \} \cup \{ y_{\gamma + 1} : \gamma < \min \{ \alpha, \kappa^+ \} \}$. 
Recall from the statement of the Lemma that $\lambda = \xp_\cs$, thus $|\alpha| < \xp_\cs$. It remains to show that
$p$ is finitely satisfiable in $X_\mb$ (i.e. consistent). 
Since we choose a domain $x_3$, a distance estimate function $x_4$ and a partial injection $x_5$ simultaneously, the true constraints come 
from the schemata $(p.8)$ and $(p.9)$ which require that certain elements are in the domain, thus certain previously set distance estimates must be respected and,
if applicable, certain new distances set.

\bbr
Let us check finite satisfiability (by compactness this will suffice). 
Let a nonempty finite subset $\Gamma \subseteq \alpha \cap \kappa^+$ be given. Let $p_0 \subseteq p$ be finite and such that
$p_0$ implies (p.1)--(p.7), $p_0$ implies $($p.8$)_\gamma$ for each $\gamma \in \Gamma$ and
$p_0$ implies $($p.9$)_{\zeta, \gamma}$ for each $\zeta, \gamma \in \Gamma$, $\zeta < \gamma$. 
We will prove that $p_0$ is satisfiable. 

We define $b_3, b_4, b_5, b_1, b_2$ as follows.

\br
\begin{enumerate}
\item[(i)] Let $b_3 = \{ y_{\gamma + 1} : \gamma \in \Gamma \}$. 
\item[(ii)] Let $b_4$ be the symmetric 2-place function on $b_3$ defined by:
\[  (y_{\zeta+1}, y_{\gamma + 1}) \mapsto d_{g(\zeta+1, \gamma+1)} \]
for $\zeta, \gamma \in \Gamma$. (Note that this function exists and may be identified with an element of $\mct_{\ma \times \ma}$, 
thus of $X_{\ma^\prime}$, by the choice of $\ma, \ma^\prime$ satisfying \ref{cov-pair}, and the fact that finite sequences in this tree are closed
under concatenating an additional element, since $\ma^\prime$ is nontrivial.)
\item[(iii)] Let $d = \max \{ d_{g(\zeta + 1, \gamma + 1)} : \zeta \neq \gamma \in \Gamma \}$. Note that for all $\alpha < \lambda$,
$d < e_\alpha$, by choice of the cuts at the beginning of the proof.
\item[(iv)] Let $\{ \gamma_0, \dots \gamma_n \}$ enumerate $\Gamma$ in increasing order, without repetition.
Let $b_5$ be the definable partial injection given by
\[ y_{\gamma_i +1} \mapsto e_{\alpha+1} +_\ma 1 +_\ma i\cdot d \]
for $i = 0, \dots n$.  By the choice of $\overline{e}$, $\max(\rn(b_5)) <_\ma e_\alpha$. (The parenthetical remark from (ii) applies here also.) 
\item[(v)] Let $b_0, b_1, b_2$ be defined from $b_3, b_4, b_5$ by conditions (p.1), (p.2), (p.6) of the definition of $p$.
\end{enumerate}

\bbr
\noindent Let us verify that $(b_0, \dots b_5) \models p_0$. Condition (p.1) is obvious. For Conditions (p.2) and (p.6), note that by (iv)
$\max (\rn(b_5)) <_\ma e_\alpha$.  Condition (p.3) is obvious as $b_3$ is finite and $X_\ma$ is not.
Conditions (p.4), (p.5), and (p.6) are immediate from the definitions of the elements $b_i$. For $\zeta < \gamma \leq \beta$,
Conditions (p.8)$_\gamma$ and (p.9)$_{\zeta, \gamma}$ are also immediate.
What about Condition (p.7)? 
(Note that on our proposed finite fragment, $c_\beta(n_\beta, 3) \cap b_3 \subseteq b_3$.)
By inductive hypothesis [(1), (5)],
we have that
\[ c_\beta(n_\beta, 3) \cap b_3 = \{ y_{\gamma + 1} : \gamma \in \Gamma, \gamma \neq \beta \}. \]
Thus, when $\zeta < \gamma < \beta$, Condition (p.7) for $a=y_{\zeta+1}, b=y_{\gamma+1}$ is ensured by Condition (p.9)$_{\zeta, \gamma}$, and when $\zeta < \beta$,
Condition (p.7) $a=y_{\zeta+1}, b=y_{\beta+1}$ is trivially true since $y_{\beta + 1} \notin c_\beta(n_\beta, x_3)$.

This completes the proof that $p_0$ is realized, \emph{thus} that $p$ is an $\ord$-type in the sense of Definition \ref{x15} over a set of size $<\xp_\cs$.

By Theorem \ref{x16}, $p$ has a realization $\langle b^*_i : i < 6 \rangle$. By inductive hypothesis, we may concatenate.
Let $n_\alpha = n_\beta +1$ and let
$c_\alpha = {c_\beta}^\smallfrown\langle b^*_0, \dots b^*_5\rangle$. This completes the successor step.

\bbr
\noindent\emph{For the case of $\alpha$ limit}: Since $\cf(\alpha) < \lambda \leq \xp_\cs \leq \xt_\cs$, by Lemma \ref{c:ceiling} 
there is $c \in \mct_6$ such that $\beta < \alpha$ implies $c_\beta \tlf c$ and $c$ is below the ceiling. 
The key point is to adjust this by choosing a suitable initial segment $c_\alpha$ of $c$ so that 
Conditions (1), (2) and (3)(a) of the inductive hypothesis will be satisfied. (Conditions (3)(b)-(c) will then follow by definition of $\mct_6$.) 
First, let
\[ n_* = \max \{ n : n <_\mb \lgn(c), M^+_1 \models ( e_\alpha <_\ma c(n, 0) ) \}. \]
Necessarily for all $\beta < \alpha$, $\lgn(c_\beta) <_\mb n_*$.
Second, for each $\beta < \max \{ \alpha, \kappa^+ \}$, let 
\[ n(\beta) = \max \{ n \leq_\mb n_* :  y_{\beta + 1} \in c(n, 3) \}. \]
For each $\gamma < \beta < \alpha$, inductive hypothesis (3) for $\beta$ implies
$y_{\gamma+1} \in c_\beta(n_\beta, 3)$. In other words,
\[ (\{ n_\beta : \beta < \alpha \cap \kappa^+ \}, \{ n(\beta) : \beta < \alpha \cap \kappa^+ \} ) \]
is a pre-cut in $X_\mb$, thus a $(\kappa_1, \kappa_2)$-pre-cut for some regular $\kappa_1, \kappa_2 \leq |\alpha \cap \kappa^+|$.
It cannot be a cut, as then $\kappa_1 + \kappa_2 \leq |\alpha| < \lambda = \xp_\cs$ contradicting the definition of $\xp_\cs$.
So we may choose $n_{**} \leq_\mb n_*$ such that for all $\gamma, \beta < \alpha \cap \kappa^+$,
\[ n_\gamma <_\mb  n_{**} <_\mb n_\beta. \]
Let $n_\alpha = n_{**}$, and let $c_\alpha = {c\rstr_{n_\alpha}}$. 
This completes the limit case, and thus the inductive construction of the sequence $\langle c_\alpha, n_\alpha : \alpha < \lambda \rangle$.

\bbr

We arrive at the final stage of the proof. Having built our sequence $\langle c_\alpha : \alpha < \kappa \rangle$, as $\lambda = \xp_\cs < \xt_\cs$
we may choose $c_\lambda \in \mct_6$ such that for all $\alpha < \lambda$,
$c_\alpha \tlf c_\lambda$. Let $n_\lambda = \maxdom(c_\lambda))$.
By construction, $\langle n_\alpha  : \alpha < \lambda \rangle$ is an increasing sequence of elements of $X_\mb$ 
below $n_\lambda$. By Theorem \ref{m2} (Uniqueness), there is a sequence $\langle m_\epsilon : \epsilon < \kappa \rangle$
such that
\[ (\langle n_\alpha : \alpha < \lambda \rangle, \langle m_\epsilon : \epsilon < \kappa \rangle) \]
represents a $(\lambda, \kappa)$-cut in $X_\mb$.
Without loss of generality, for some increasing $\zeta : \kappa \rightarrow \kappa$,
\[ d_{\zeta(\epsilon)} <_\ma c_\lambda(m_\epsilon, 0) <_\ma d_{\zeta(\epsilon+1)}. \]
Now for each $\beta < \kappa^+$ and each $\alpha_1, \alpha_2$ with $\beta < \alpha_1 < \alpha_2 < \lambda$, the set
\[ X_\beta := \{ n : n \leq_\mb n_\lambda, (\forall n^\prime)(n_{\beta + 1} \leq_\mb n^\prime \leq_\mb n \rightarrow y_{\beta + 1} \in c_\lambda(n^\prime,3) ) \} \]
is a subset of $X_\mb$ which includes the interval $[n_{\alpha_1}, n_{\alpha_2}]_\mb$, 
recalling the notation $n_\gamma = \maxdom(c_\gamma))$. Thus for some $\epsilon(\beta) < \kappa$,
\[ [n_{\beta+1}, m_{\epsilon(\beta)}]_\mb \subseteq X_\beta. \]
As $\beta < \kappa^+$ was arbitrary, there is $\epsilon_* < \kappa$ such that
\[ \uu = \{ \beta < \kappa^+ : \epsilon(\beta) = \epsilon_* \} \]
has cardinality $\kappa^+$. For simplicity of notation, let $F:= c_\lambda(m_{\epsilon_*}, 4)$.
By Condition (d) in the definition of $\mct_6$ and the inductive hypothesis, for every $\beta \neq \gamma \in \uu$ we have that
$F(y_\gamma, y_\beta) = d_{g(\gamma, \beta)}$.
However, by the choice of $\overline{d}$, the choice of the outside function $g$ and the fact that $|\uu| = \kappa^+$, there exist $\gamma, \beta \in \uu$ such that
$M^+_1 \models$ $|d_{\zeta(\epsilon_*) + 1}| <^\cs |d_{g(\gamma, \beta)}|$.
This contradiction completes the proof.
\end{proof}

\bbr

\section{$\mc(\cs, \xt_\cs)=\emptyset$}  
\label{s:main-thm}

With the results of the previous sections in hand, 
we now prove the paper's fundamental result:

\begin{theorem-c1} \label{no-cuts}
Let $\cs$ be a cofinality spectrum problem. Then $\mc(\cs, \xt_\cs) = \emptyset$.
\end{theorem-c1}

\begin{proof} 
There are two cases.

\emph{Case 1}. $\xp_\cs < \xt_\cs$. Let $\kappa, \lambda$ be such that $\kappa + \lambda = \xp_\cs$ and
$(\kappa, \lambda) \in \mc(\cs, \xt_\cs)$. By Claim \ref{m2}, Conclusion \ref{cor:sym} and Lemma \ref{m12}, $\kappa \neq \lambda$
and we may assume $\kappa < \lambda = \xp_\cs$. Then the hypotheses of Theorem \ref{last-cut} are satisfied,
so $(\kappa, \lambda) \notin \mc(\cs, \xt_\cs)$, contradiction. This shows Case 1 cannot occur. 

\emph{Case 2}. $\xt_\cs \leq \xp_\cs$, so by definition of $\xp_\cs$, $\mc(\cs, \xt_\cs) = \emptyset$.
This proves the theorem. \end{proof}

\br

\section{A new characterization of good regular ultrafilters} \label{s:defns}

We now apply cofinality spectrum problems to study regular ultrapowers and Keisler's order. 
We focus here on the definitions and history most relevant to our present proofs; further details can be found in \cite{MiSh:996}. 
In the present section, we give the necessary background and then prove two theorems: Theorem \ref{maximal-uf}, the analogue of Theorem \ref{no-cuts} for regular 
ultrapowers, and Main Theorem \ref{maximal-x}, a new characterization of Keisler's notion of goodness.

\begin{conv}[Conventions on ultrapowers] \emph{ }
\begin{enumerate}
\item For transparency we assume all languages are countable. 
\item For transparency we assume all theories $T$ are complete. 
\item We use $\de$ to denote a regular ultrafilter on the index set $I$, see $\ref{regular-filter}$, and $\lambda$ to denote $|I|$. 
\item Small means of cardinality $\leq |I|$. 
\item We use notation of the form $a[t]$ following Convention \ref{c:first}(6). 
\end{enumerate}
\end{conv}

\begin{defn} \label{regular-filter} 
Say that the filter $\de$ on $|I|$ is $\kappa$-regular when there is a collection $\overline{X} = \{ X_i : i < \kappa \} \subseteq \de$ such that for each $t \in I$,
\[ | \{ i < \kappa : t \in X_i \} | < \aleph_0 \]
Such a collection is called  a \emph{$\kappa$-regularizing family}. 
Call $\de$ regular when it is $|I|$-regular. 
\end{defn}

Regularity is a kind of strong incompleteness: a $\kappa$-regularizing family is such that 
for any $\sigma \subseteq \kappa$, if $|\sigma| \geq \aleph_0$ then  
$\bigcap \{ X_i : i \in \sigma \} = \emptyset$.

\begin{defn} When $\de$ is an ultrafilter on $I$, $M$ a model,  
say that \emph{$\de$ saturates $M$} to mean $M^I/\de$ is $|I|^+$-saturated.
\end{defn}

\begin{fact}[Keisler \cite{keisler}] \label{r:reg}
If $\de$ is a regular ultrafilter on $I$ then whenever $M \equiv N$, $\de$ saturates $M$ iff $\de$ saturates $N$.
\end{fact}

In particular, when $\de$ is regular, the following is well defined (recalling that all languages are countable). 

\begin{defn}
When $\de$ is regular, say that \emph{$\de$ saturates $T$} to mean $\de$ saturates $M$ for some, equivalently every, $M \models T$. 
\end{defn}

Keisler in 1967 proposed that regular ultrafilters could be used to investigate the relative complexity of first-order theories:

\begin{defn} \label{keisler-order} \emph{(Keisler \cite{keisler})}
Let $T_1, T_2$ be complete, countable first order theories. Say that $T_1 \kleq T_2$ if for any infinite cardinal $\lambda$ and 
any regular ultrafilter $\de$ on $\lambda$, if $\de$ saturates $T_2$ then $\de$ saturates $T_1$. 
\end{defn}

In other words, $T_1 \kleq T_2$ if for all infinite $\lambda$, 
all $M_1 \models T_1, M_2 \models
T_2$, and any $\de$ a regular ultrafilter on $\lambda$,
if $M^{\lambda}_2/\de$ is $\lambda^+$-saturated then $M^{\lambda}_1/\de$ must be
$\lambda^+$-saturated.

Determining the structure of Keisler's order is a far-reaching problem.  It was studied by Keisler, then by Shelah in Chapter VI of \cite{Sh:a}, 
then in a series of papers by Malliaris \cite{mm-thesis}, \cite{mm1}, \cite{mm4}, \cite{mm5}, and very recently by both authors jointly 
in \cite{MiSh:996}, \cite{MiSh:997}, \cite{MiSh:999}, \cite{MiSh:1009}.  
We refer the interested reader to the introduction of \cite{MiSh:996}
for motivation and \cite{MiSh:1009} for a survey of what is known. 
Here, we focus on the part of the classification problem relevant to our current work (on the maximal class, and on non-simple theories).
Towards this let us give several theorems and definitions. 
The two fundamental theorems of ultraproducts we use are \lost theorem and its corollary:

\begin{thm-lit} \emph{(Ultrapowers commute with reducts)} \label{commute-with-reducts}
Let $\tau, \tau^\prime$ denote vocabularies. 
Let $M$ be an $\tau^\prime$-structure, $\tau \subseteq \tau^\prime$, $\de$ an ultrafilter on $\lambda \geq \aleph_0$, $N = M^\lambda/\de$.
Then
\[ \left( M^\lambda/\de      \right)|_{\tau}  = \left(  M|_{\tau} \right)^\lambda/\de      \]
\end{thm-lit}

In light of Theorem \ref{commute-with-reducts}, it is useful to consider structure on the ultrapower coming from the index models 
in some possibly expanded language.  

\begin{defn} \emph{(Internal, i.e. induced)}\label{d:internal}
Let $N = M^\lambda/\de$ be an ultrapower. Say that a relation or function $X$ on $N$ is \emph{internal}, also called \emph{induced},
if we may expand the language
by adding a new symbol $Y$ of the same arity as $X$, and choose for each $t \in \lambda$ an interpretation $Y^{M_t}$ of $Y$,
so that $Y^{N} \equiv X \mod \de$, where $Y^N = \{ a \in {^\lambda M } : \{ t < \lambda : a[t] \in Y^{M_t} \} \in \de \}$.
\end{defn}

Good ultrafilters are a useful family of ultrafilters introduced by Keisler. 
For any infinite $\lambda$, $\lambda^+$-good ultrafilters on $\lambda$ exist
by a theorem of Kunen \cite{kunen}, extending a theorem of Keisler which assumed GCH. 

\begin{defn} \label{good-filters} \emph{(Good filters, Keisler \cite{keisler-1})}  
Say that the filter $\de$ on $I$, $|I| = \lambda$ is \emph{$\kappa^+$-good} if every $f: \fss(\kappa) \rightarrow \de$ has
a multiplicative refinement, i.e. there is $f^\prime: \fss(\kappa) \rightarrow \de$ such that:
\begin{enumerate}
\item $u \in [\kappa]^{<\aleph_0}$ implies $f^\prime(u) \subseteq f(u)$
\item $u, v \in [\kappa]^{<\aleph_0}$ implies $f^\prime(u) \cap f^\prime(v) = f^\prime(u \cup v)$
\end{enumerate}
In this paper we assume all good ultrafilters are $\aleph_1$-incomplete, thus regular. We say that a filter is \emph{good} if it
is $|I|^+$-good.
\end{defn}

Keisler proved that Keisler's order has a maximum class, and that this class had a set-theoretic characterization: 

\begin{thm-lit} \label{good-max} \emph{(Keisler \cite{keisler})}
There is a maximum class in Keisler's order, which consists precisely of those theories $T$ such that
for $M \models T$ and $\de$ a regular ultrafilter on $\lambda$, $M^\lambda/\de$ is $\lambda^+$-saturated iff
$\de$ is $\lambda^+$-good. 
\end{thm-lit}

\begin{concl} \label{must-be-good} 
Let $X$ be any property of regular ultrafilters. Suppose it can be shown that for some countable first-order theory $T$ and model $M \models T$,
the condition ``$\de$ has property $X$'' is necessary for $M^I/\de$ to be ${|I|}^+$-saturated. Then a consequence of Theorem \ref{good-max} is that
any (regular) good ultrafilter $\de$ on $I$ must have property $X$.
\end{concl}

The question of the model-theoretic identity of this maximum class has remained elusive. The importance of this question comes,
in part, from its being an outside definition of a class of unstable theories.

The connection in Theorem \ref{good-max} between realizing types and multiplicative refinements is based on the following definition and 
Fact \ref{mrft}. 

\begin{defn} \label{dist}
Let $N = M^I/\de$ be a regular ultrapower and $p = \{ \psi(x;a_i) : i < \mu \}$ a $\psi$-type in $N$ of cardinality $\mu \leq |I|$.
\begin{enumerate}
\item The \emph{\Los map}
$d_0 : \fss(\mu) \rightarrow \de$ is given by
\[ u \in \fss(\mu) \mapsto \{ s \in I : M \models \exists x \bigwedge \{ \psi(x;a_j) : j \in u \} \} \]
\item A \emph{distribution} for $p$ is a map $d: \fss(p) \rightarrow \de$ such that:
\begin{enumerate}
\item $d$ refines the \los map
\item the image of $d$ is a $\mu$-regularizing set for $\de$
\item without loss of generality $d$ is monotonic, i.e. $u \subseteq v$ implies $d(u) \supseteq d(v)$
\end{enumerate}
\end{enumerate}
\end{defn}

\begin{fact}[see e.g. \cite{MiSh:996} Section 1.2] \label{mrft}
For a small type $p$ in an ultrapower $N = M^I/\de$, the following are equivalent: $(a)$ $p$ is realized, $(b)$ some distribution of $p$ has a multiplicative refinement. 
\end{fact}

If Keisler's order gives a measure of complexity of theories, a first surprise was the nature of the complexity which maximality suggests.
\begin{thm-lit} \emph{(Sufficient conditions for maximality, Shelah)} \label{sop-maxl}
\begin{enumerate}
\item Any theory with the strict order property, e.g. $(\mathbb{Q}, <)$, is maximum in Keisler's order \cite{Sh:a}.VI.2.
\end{enumerate}
\emph{(and considerably later)}
\begin{enumerate}[resume]
\item $SOP_3$, a weakening of the strict order property, is sufficient for maximality \cite{Sh:500}. 
\end{enumerate}
\end{thm-lit}

We will not work directly with $SOP_3$, a weaker order, in this paper; a definition can be found in Shelah and Usvyatsov \cite{ShUs} Fact 1.3.

Since 1996, it has been open whether the boundary of the Keisler-maximal class lies at $SOP_3$ or whether it could be pushed up 
to $SOP_2$, Definition \ref{sop2-tree} below (or beyond). 
A major technical obstacle has been the lack of a framework within which to compare
orders and trees. 
Note that there are also nontrivial model-theoretic questions about whether $SOP_2$ may imply $SOP_3$ on the level of theories, but 
our analysis is able to avoid this. 

\begin{disc} \label{sop2:disc} \emph{($SOP_2$ and $SOP_3$)}
It is known that $SOP_2$ is weaker than $SOP_3$ on the level of formulas, but what about on the level of theories? 
This remains open and interesting. Our analysis ultimately circumvents this problem by showing that already $SOP_2$ suffices for
maximality from the point of view of Keisler's order.
\end{disc}

Before addressing the question of the maximality of $SOP_2$ in Keisler's order, we formally connect regular ultrapowers to CSPs.
Namely, we prove that the results on cofinality spectrum problems developed above may be applied to the study of cuts in regular ultrapowers of linear order.
First, we specialize the definitions of treetops \ref{treetops}, lower cofinality \ref{lcf-cs}, and $\mc(\cs, \xt_\cs)$
\ref{cst:card}) to the context of regular ultrapowers. (Generally, what bounds results in this section is not the size of the index set but that of the regularizing family.)
As $\de$ is regular, the saturation of $M$ will not matter.

\begin{defn} \emph{(Treetops)} \label{d:treetops}
Let $\de$ be an ultrafilter on $I$ and $\kappa$ a regular cardinal. 
\begin{enumerate}
\item We say that $\de$ has \emph{$\kappa$-treetops} when: 
for any $\kappa$-saturated model $M$ which interprets a tree $(\mct_M, \tlf)$, $N = M^I/\de$, $\gamma = \cf(\gamma) < \kappa$ and any
$\tlf$-increasing sequence $\langle a_i : i < \gamma \rangle$ in $(\mct_N, \tlf_N)$ 
there is $a \in \mct_N$ such that $i < \kappa$ implies $a_i \tlf a$. 
\item We say that $\de$ has $<\kappa$-treetops if $\de$ has $\theta$-treetops whenever $\theta =\cf(\theta) < \kappa$.
\item We say that $\de$ has $\leq \kappa$-treetops if $\de$ has $(< \kappa^+ )$-treetops.
\end{enumerate}
Our main case is $\lambda^+$-treetops where $\de$ is a regular ultrafilter on $I$, $|I| = \lambda$. 
\end{defn}

\begin{defn} \label{c-of-d}
For $\de$ an ultrafilter on $I$ we define:
\[ \mc(\de) = \left\{ (\kappa_1, \kappa_2) \in (\operatorname{Reg} \cap {|I|}^+) \times (\operatorname{Reg} \cap {|I|}^+) :
~\mbox{$(\omega, <)^I/\de$ has a $(\kappa_1, \kappa_2)$-cut} \right\} \]
\end{defn}

The paper's central question from the introduction thus becomes:

\begin{qst}
Let $\de$ be a regular ultrafilter on $I$ with ${|I|}^+$-treetops. What are the possible values of $\mc(\de)$?
\end{qst}

\begin{cor} \emph{(of the maximality of strict order, Theorem \ref{sop-maxl} above)}
\label{empty-good}
Let $\de$ be a regular ultrafilter on $I$, $|I| = \lambda$. Then $\mc(\de) = \emptyset$ iff $\de$ is $\lambda^+$-good.
\end{cor}

\begin{claim} \label{approx-omega}
Let $\de$ be a regular ultrafilter on $I$, $|I| = \lambda$. For any $n<\omega$, write $<_n$ for the order on $\omega$
restricted to $n$, i.e. to $\{ 0, \dots, n-1 \}$. Then there exists a sequence   
$\overline{n} = \overline{n}(\de) = \langle n_t : t \in I \rangle \in {^I \omega}$ 
such that for all regular cardinals $\kappa_1, \kappa_2$ with $\kappa_1 + \kappa_2 \leq \lambda$, 
the following are equivalent:
\begin{enumerate}
\item $(\kappa_1, \kappa_2) \in \mc(\de)$, i.e. $(\omega, <)^I/\de$ has a $(\kappa_1, \kappa_2)$-cut.
\item $\prod_t (n_t, <_{n_t})/\de$ has a $(\kappa_1, \kappa_2)$-cut. 
\end{enumerate}
\end{claim}

\begin{proof}
Note that it suffices to show (1) $\rightarrow$ (2). 

Without loss of generality, consider $M = (\omega, <)^I/\de$ and $M_1 = M^I/\de$. 
Let $\{ X_i : i < \lambda \}$ be a regularizing family, \ref{regular-filter} above. For $t \in I$,
let $n_t = | \{ i < \lambda : t \in X_i \} | + 1$.  
We verify that $\langle n_t : t \in I \rangle$ works.  
Let $(\langle a_\alpha : \alpha < \kappa_1 \rangle, \langle b_\beta : \beta < \kappa_2 \rangle)$ 
represent a $(\kappa_1, \kappa_2)$-cut in $(\omega, <)^\lambda/\de$. As $\kappa_1 + \kappa_2 \leq \lambda$, there is a map  
$d: \kappa_1 \times \{ 0 \} \cup \kappa_2 \times \{ 1 \} \rightarrow \de$ such that for each $t \in I$, 
$| \{ x \in \dom(d) : t  \in d(x) \} | < n_t$. For each $t \in I$, let 
$X_t = \{ a_\alpha[t] : t \in d((\alpha, 0)) \} \cup \{ b_\beta[t] : t \in d((\beta, 1)) \}$, which is a 
(linearly ordered) subset of $(\omega, <)^M$ with 
fewer than $n_t$ elements. Let $<_{X_t}$ denote the restriction of the linear order on $\omega$ to $X_t$. 
Then we may choose at each index $t$ an order preserving injection 
$h_t: (X_t, <_{X_t}) \rightarrow (n_t, <_{n_t})$ whose image is an interval. 
Let $h$ be the internal function whose projection to $t$ is $h_t$. Then by \lost theorem and the requirement that the range be an interval, 
we have that $(\langle h(a_\alpha) : \alpha < \kappa_1 \rangle, \langle h(b_\beta) : \beta < \kappa_2 \rangle)$ represents a
$(\kappa_1, \kappa_2)$-cut in $\prod_t (n_t, <_{n_t})/\de$. This completes the proof. 
\end{proof}

\begin{defn} \label{s-c}
Let $\de$ be a regular ultrafilter on $I$, $M$ a model extending $(\omega, <)$. 
If $\langle n_t : t \in I \rangle \in {^I \omega}$ is a sequence satisfying the conclusion of $\ref{approx-omega}$ for $\de$ 
and $(X, <_X) \subseteq M^I/\de$ is given by
\[ (X, <_X) = \prod_t (n_t, <_{n_t}) /\de \]
then we say $(X, <_X)$ \emph{captures pseudofinite cuts.} $[$Clearly, this depends on the background data of $I, \de, M$.$]$
\end{defn}

The next Claims \ref{up-csp}, \ref{up-csp2}, and \ref{up-csp3} justify regarding regular ultrapowers 
extending the theory of linear order [see Claim \ref{up-csp2}] as cofinality spectrum problems, and 
show that the specialized definitions for $\mc(\de)$ and treetops accurately reflect the properties of this background CSP. 

\begin{claim} \emph{(Regular ultrapowers as CSPs)}  \label{up-csp} 
Let $\de$ be a regular ultrafilter on $I$, $|I| = \lambda$. 
Let $M$ expand $(\omega, <)$ and let $M_1 = M^I/\de$. 

Then there exist expansions $M^+, M^+_1$ of $M, M_1$ respectively such that $M^+_1 = (M^+)^I/\de$
and a set of formulas
$\Delta \supseteq \{ x < y < z\}$ of the language of $M$ such that
\begin{enumerate} 
\item $\cs = ( M, M_1, M^+, M^+_1, Th(M^+), \Delta )$ 
is a cofinality spectrum problem, and 
\item some nontrivial $\ma \in \ord(\cs)$ captures pseudofinite cuts in the sense of $\ref{s-c}$.
\end{enumerate}
\end{claim}

\begin{proof} 
As ultrapowers commute with reducts, 
for (1) choose any expansion $M^+$ of $M$ which will code sufficient set theory 
for trees in the sense of \ref{d:estt}, e.g. the complete expansion, or an expansion to a model of $(\mch(\chi), \in)$ 
for some sufficiently large $\chi$. Let $M^+_1 = (M^+)^I/\de$.   
For (2), let $\langle n_t : t \in I \rangle$ be given by \ref{approx-omega}. By construction, the linear order 
$\prod_t (n_t, <_{n_t}) /\de$ is $\Delta$-definable in $M_1$ and captures pseudofinite cuts. 
It will correspond to some nontrivial $\ma \in \ord(\cs)$ provided we choose $d_\ma$ to not be a natural number. 
\end{proof}

\begin{defn}
Call any $\cs$ satisfying the conclusion of \ref{up-csp} a cofinality spectrum problem \emph{associated to $\de$}.
\end{defn}

In the next series of claims we verify that cuts and trees behave as expected.

\begin{claim} \label{up-csp2}
Let $\de, I, \lambda, M, M_1$ be as in \ref{up-csp} and let $\cs$ be a cofinality spectrum problem given by that Claim. 
For $\kappa_1 + \kappa_2 \leq \lambda$, $\kappa_1, \kappa_2$ regular, the following are equivalent:

\begin{enumerate}
\item There is a $(\kappa_1, \kappa_2)$-cut in some $M^+_1$-definable linearly ordered set. 
\item $(\kappa_1, \kappa_2) \in \mc(\cs, |I|^+)$.  
\item $(\kappa_1, \kappa_2) \in \mc(\de)$.
\end{enumerate}
\end{claim}

\begin{proof}

(3) $\rightarrow$ (1): By Claim \ref{up-csp}(2).

\br
(1) $\rightarrow$ (2): It suffices to show that:

\step{Sub-claim}.
For any definable linear order $(Y, <_Y)$ in $M^+_1$ $[$i.e. both $Y$ and $<_Y$ are $M^+_1$-definable but this order is not necessarily in $\ord(\cs)$$]$ and any discrete $A \subseteq Y$, $|A| \leq \lambda$ (e.g. a representation of a cut) 
there exist a nontrivial $\ma \in \ord(\cs)$ and an internal partial injection $f$ such that:
\begin{enumerate}
\item[(a)] $A \subseteq \dom(f)$, $\rn(f) \subseteq X_\ma$
\item[(b)] $f$ is order preserving, i.e. for all $a, b \in \dom(f)$, $a \leq_Y b$ iff $f(a) \leq_\ma f(b)$
\item[(c)] $\rn(f)$ is an interval in $(X_\ma, \leq_\ma)$
\end{enumerate}

\br
\noindent The proof is almost exactly the same as that of \ref{approx-omega}, using $(Y, <_Y)$ here instead of the representation 
of the cut there, and letting $\ma \in \ord(\cs)$ be given by \ref{up-csp}(2). 
(The point: by regularity, any discrete linearly ordered set $|A| \leq \lambda$ in the ultrapower may be considered as a 
subset of some internal, pseudofinite linear order.)

\br
(2) $\rightarrow$ (3):
Assume (2), so there are regular cardinals $\kappa_1, \kappa_2$ with $\kappa_1 + \kappa_2 \leq \lambda$ and
some nontrival $\mb \in \ord(\cs)$ such that $X_\mb$ contains a $(\kappa_1, \kappa_2)$-cut. 
To conclude that $(\kappa_1, \kappa_2) \in \mc(\de)$, let $\ma$ be given by \ref{up-csp}(2). 
By Observation \ref{n-is-enough} (``any $\ma \in \ord(\cs)$ will work''), 
also $(X_\ma, <_\ma)$ has a $(\kappa_1, \kappa_2)$-cut, thus also $(\omega, <)^I/\de$. 
\end{proof}

\begin{claim} \label{up-csp3}
Let $\de, I, \lambda, M, M_1$ be as in \ref{up-csp} and let $\cs$ be a cofinality spectrum problem given by that Claim. 
For $\kappa = \cf(\kappa) \leq \lambda$, then the following are equivalent:

\begin{enumerate}
\item $\de$ has $\kappa^+$-treetops in the sense of $\ref{d:treetops}$.
\item $\kappa^+ \leq \xt_\cs$. 
\end{enumerate}
\end{claim}

\begin{proof}
Clearly (1) $\rightarrow$ (2). 

To show (2) $\rightarrow$ (1), we prove the contrapositive. That is, we show that:

\step{Subclaim.} If $(\mct, \tlf_\mct)$ is any tree definable in $M^+$,
not necessarily an element of $\tr(\cs)$, and there is in 
$(\mct, \tlf_\mct)^I/\de = (\mct, \tlf_\mct)^{M^+_1}$ an increasing sequence 
of length $\kappa = \cf(\kappa) \leq \lambda$ with no upper bound, \emph{then} there is
$\ma \in \ord(\cs)$ such that in $\mct_\ma$ there is an increasing sequence of length $\kappa$ with no upper bound. 

\br
\noindent
(While $\ma$ nontrivial will be guaranteed by choosing $\ma$ from \ref{up-csp}(2), it follows from 
$\mct_\ma$ having arbitrarily long paths.) 

So let such a $(\mct, \tlf_\mct)$ be given. Let $\overline{c} = \langle c_\alpha : \alpha < \kappa \rangle$ 
be an increasing sequence in $\mct$ with no upper bound. By regularity, as $\kappa \leq \lambda$  there is a map 
$d : \kappa \rightarrow \de$ whose image is a regularizing family. In other words, by \lost theorem, 
we may assume that there is a sequence of finite trees $(\mct_t, \tlf^t_\mct)$ for $t \in I$ such that
\[ \prod_t (\mct_t, \tlf^t_\mct) /\de  \]
is a subtree of $(\mct, \tlf_\mct)^{M^+_1}$ which includes the sequence $\overline{c}$. 
Let $\ma \in \ord(\cs)$ be given by \ref{up-csp}(2).

Analogously to \ref{approx-omega}, we may choose at every (or almost every) index $t \in I$ a function $f_t: (\mct_t, \tlf^t_\mct) \rightarrow (\mct^{M^+}_\ma, \tlf^{M^+}_\ma)$ 
such that $f_t$ is injective and respects the partial ordering, i.e. for $x, y \in \dom(f_t)$ we have that $x \tlf^t_\mct y$ iff 
$f_t(x) \tlf^{M^+}_\ma f_t(y)$.  
Now let $f = \prod_t f_t /\de$ and suppose for a contradiction that
$\langle b_\alpha := f(c_\alpha) : \alpha < \kappa \rangle$ has an upper bound in $\mct_\ma$, call it $b_*$.  Consider the map
$d^\prime : \kappa \rightarrow \de$ given by
$\alpha \mapsto d(\alpha) \cap \{ t\in I : b_\alpha[t] \tlf b_*[t] \} \cap \{ t \in I : f_t$ is injective and respects the partial ordering $\}$. 
Now for each $t \in I$, the set $B_t := \{ b_\alpha[t] : \alpha < \kappa \land t \in d^\prime(\alpha) \}$ is linearly ordered by $\tlf$, 
by the choice of $b_*$. 
For each $t \in I$, let $b_t$ be the maximal element of $B_t$ under this linear ordering. Then by \lost theorem and the choice of the $f_t$s, we have that
the element $c_* := \prod_t f^{-1}_t(b_t)/\de$ is well defined. By \lost theorem (recalling that $\mct$ is definable) $c_* \in \mct$, 
and again by \lost theorem $c_*$ is an upper bound for the sequence $\overline{c}$ in $\mct$, contradiction. 
We have shown that $\langle b_\alpha : \alpha < \kappa \rangle$ is an increasing sequence in $\mct_\ma$ with no upper bound, 
which completes the proof. 
\end{proof}

\begin{concl} \label{up-concl}
Regular ultrapowers extending the theory of linear order 
may be regarded as cofinality spectrum problems, 
for which the specialized definitions $\mc(\de)$,  
$|I|^+$-treetops retain their intended meaning as stated in $\ref{up-csp2}$ and $\ref{up-csp3}$.  
\end{concl}

\begin{proof}
By \ref{up-csp}, \ref{up-csp2}, \ref{up-csp3}.
\end{proof}

We may now give the analogue of Theorem \ref{no-cuts} for regular ultrapowers:

\begin{theorem} \label{maximal-uf}
Let $\de$ be a regular ultrafilter on $I$, $|I| = \lambda \geq \aleph_0$. 
If $\de$ has $\lambda^+$-treetops, then $\de$ is $\lambda^+$-good.
\end{theorem}

\begin{proof}
Let $\cs$ be a cofinality spectrum problem associated to $\de$, given by \ref{up-csp}.  
By Theorem \ref{no-cuts}, $\mc(\cs, \xt_\cs) = \emptyset$. 
By Claim \ref{up-csp3} and the assumption of $\lambda^+$-treetops, 
$\lambda^+ \leq \xt_\cs$, so necessarily $\mc(\cs, |I|^+) = \emptyset$. 
Apply Claim \ref{up-csp2} to conclude $\mc(\de) = \emptyset$. 
Then by Corollary \ref{empty-good} (or Fact \ref{max-good})  
$\de$ is $\lambda^+$-good, which completes the proof.
\end{proof}

Moreover, the result from Section \ref{s:symmetric} may here be strengthened to a characterization: 

\begin{lemma} \label{sym-cuts}
Let $\de$ be a regular ultrafilter on $I$, $|I| = \lambda$. 
Then the following are equivalent:
\begin{enumerate}
\item $\kappa = \cf(\kappa) \leq \lambda = |I|$ implies $(\kappa, \kappa) \notin \mc(\de)$.
\item $\de$ has $|I|^+$-treetops. 
\end{enumerate}
\end{lemma}

\begin{proof}
(2) $\rightarrow$ (1):  Lemma \ref{c:motiv}. 

(1) $\rightarrow$ (2):  We prove the contrapositive. Suppose that for some regular $\kappa \leq |I|$ and some model $M$ (in a countable 
signature), $M$ interprets, or without loss of generality, 
defines a tree $(\mct, \tlf_\mct)$ whose $\de$-ultrapower contains a path of length $\kappa$ with no upper bound. 
Also without loss of generality, $M$ expands $(\omega, <)$; since ultrapowers commute with reducts, there is no harm in adding this order   
under a disjoint unary predicate, or as a separate sort. 

Let $\cs$ be a cofinality spectrum problem given by Claim \ref{up-csp}.  By Claim \ref{up-csp3}, there is $\mct_\ma \in \tr(\cs)$ 
which contains a path of length $\kappa$ with no upper bound. By Lemma \ref{treetops-sym}, 
there is a definable (i.e. definable in $M^+_1$)  linear order which has a $(\kappa, \kappa)$-cut. Now apply 
Claim \ref{up-csp2} to conclude that $(\kappa, \kappa) \in \mc(\de)$, which completes the proof. 
\end{proof}

Thus we obtain the following new characterization of Keisler's notion of goodness. 

\begin{theorem-m1} \label{maximal-x}
Let $\de$ be a regular ultrafilter on $I$, $|I| = \lambda$. Then the following are equivalent:
\begin{enumerate}
\item $\de$ is $\lambda^+$-good.
\item $\de$ has $\lambda^+$-treetops.
\item $\mc(\de)$ contains no symmetric cuts.
\item $\mc(\de) = \emptyset$.
\end{enumerate}
\end{theorem-m1}

\begin{proof}

(2) $\leftrightarrow$ (3): Lemma \ref{sym-cuts}. 

(4) $\rightarrow$ (3): Immediate. 

(2) $\rightarrow$ (1): Theorem \ref{maximal-uf}.

(1) $\leftrightarrow$ (4): Fact \ref{max-good}.
\end{proof}

\begin{rmk}
By a theorem of Malliaris and Shelah, 
to the equivalent conditions of Theorem \ref{maximal-x} we may add: 
$\de$ is $\lambda^+$-excellent, see \cite{MiSh:999} Theorem 5.2. 
\end{rmk}

\section{$SOP_2$ implies maximality in Keisler's order} \label{s:sop2-reg} \label{s:sop2}

In this section, we first show that for a regular ultrafilter $\de$ on $I$, $\de$ has $\lambda^+$-treetops precisely when $\de$-ultrapowers 
realize $SOP_2$-types, as defined in \ref{conventions}(3).  
We then prove Main Theorem \ref{concl:sop2-max}, showing that any theory with $SOP_2$ is maximal in Keisler's order, 
which solves one of the two open problems mentioned in the introduction. 

\begin{defn} \label{sop2-tree} \emph{(Shelah \cite{Sh:93})} 
The theory $T$ has $SOP_2$ if there is a formula $\psi(x;\overline{y})$ such that in some model $M \models T$
there exist $\langle \overline{a}_\eta : \eta \in {^{\kappa > } \mu} \rangle$, called an \emph{$SOP_2$ tree for $\psi$}, such that:

\begin{enumerate}
\item for $\eta, \rho \in {^{\kappa > } \mu}$ incomparable, i.e. $\neg (\eta \tlf \rho) \land \neg (\rho \tlf \eta)$, we have that
$\{ \psi(x;\overline{a}_\eta), \psi(x;\overline{a}_\rho) \}$ is inconsistent.
\item for $\eta \in {^{\kappa } \mu}$, $\{ \psi(x;\overline{a}_{\eta|_i}) : i < \kappa \}$ is a consistent partial $\psi$-type.
\end{enumerate}
\end{defn}

By compactness, clearly we can replace ${^{\kappa > } \mu}$ by ${^{\omega > } 2}$ in the definition. 

\begin{defn}[Definitions and conventions on $SOP_2$] \label{conventions} In this section, 
\begin{enumerate}
\item The formula $\psi=\psi(x;y)$ will denote a formula with $SOP_2$, and $\ell(x), \ell(y)$ need not be 1. 
\item If $M^I/\de$ is a regular ultrapower, by ``$SOP_2$-type'' or ``$SOP_2$-$\kappa$-type''
we will mean a partial type $p(x) = \{ \psi(x;a_\ell) : \ell < \kappa \}$ 
almost all of whose projections to the index model come from an $SOP_2$-tree for $\psi$.  
\item Say that \emph{$\de$ realizes all $SOP_2$-types} to mean that all $SOP_2$-$|I|$-types are realized in all ultrapowers $M^I/\de$. 
\item All $SOP_2$-types considered will be $SOP_2$-$\mu$-types for $\mu \leq |I|$. \lp A regular ultrapower of a non-simple theory
will fail to be $|I|^{++}$-saturated by prior work of the authors \cite{MiSh:996}.\rp
\item As ultrapowers commute with reducts, we will freely assume that the theory in question has enough set theory for trees, Definition $\ref{d:estt}$.
\end{enumerate}
\end{defn}

Note that saying $\de$ realizes all $SOP_2$-types in the sense of Definition $\ref{conventions}$ certainly need not imply $($a priori$)$ that
for any $\vp$ with $SOP_2$, all $\vp$-types are realized in $\de$-ultrapowers. By the usual coding tricks
one may take the disjoint union of a formula with $SOP_2$ and one with e.g. $SOP$; such a formula will
necessarily be maximal. Rather, Definition $\ref{conventions}$ captures the essential structure in the sense that any
$\de$ which is able to realize all $\psi$-types for \emph{some} formula $\psi$ 
with $SOP_2$ will necessarily realize all $SOP_2$-types in the sense of Definition $\ref{conventions}(5)$.  
That said, it will follow a posteriori from 
Theorems \ref{no-cuts} and \ref{sop2-concl}
that realizing this minimal set of $SOP_2$ types is, indeed, strong enough to guarantee $\lambda^+$-saturation in general.

\begin{cor} \label{type}
The type $p$ is an $SOP_2$ type in $M^\lambda/\de$ if and only if we may add a predicate $P$ of arity $\ell(y)$ to the vocabulary $\tau$, and
for each $i \in I$ define $M_i$ as the index model $M$ expanded to a model of $\tau \cup \{ P \}$ in which $P$ names an $SOP_2$-tree, so that
in the ultraproduct $N = \prod_i M_i /\de$ we have that $p$ is a type whose parameters come from $P^N$.
\end{cor}

\begin{proof}
By \lost theorem and Definition \ref{conventions}.  
\end{proof}

Given instances $\psi(x;a_i), \psi(x;a_j)$ belonging to some consistent partial $SOP_2$-type, and some index $s \in I$,
we may write $a_i[s] \tlf a_j[s]$ to indicate comparability in the chosen $SOP_2$-tree at index $s$, and likewise
$a_i \tlf a_j$ to indicate comparability in the $SOP_2$-tree induced on $N$ by $P$. Since ultrapowers commute with reducts,
we may consider $\tlf$ as an additional relation in some expansion of the language.

We may thus consider $SOP_2$-types as arising in the following canonical way. 

\begin{defn} \label{canonical} \emph{(The canonical presentation)}
\begin{enumerate}
\item Let $T^\prime_{SOP_2}$ be the universal first-order theory in the vocabulary $\{ P, Q, \tlf, R \}$ such that
$M \models T^\prime_{SOP_2}$ if:
\begin{enumerate}
\item $M$ is the disjoint union of $P^M$, $Q^M$
\item $R^M \subseteq Q^M \times P^M$
\item $\tlf^M \subseteq P^M \times P^M$
\item $(P^M, \tlf)$ is a tree
\item if $a_1 \neq a_2 \in P^M$, $\neg(a_1 \tlf^M a_2) \land \neg (a_2 \tlf^M a_1)$ then
$M \models \neg (\exists x) (x R a_1 \land x R a_2)$.
\end{enumerate}
\item Let $\ts$ be the model completion of $T^\prime_{SOP_2}$.
\item We say that the regular ultrafilter $\de$ on $I$
\[ (\lambda^+, Q)-\mbox{saturates} ~\ts \]
\noindent if whenever $M \models \ts$
we have that $M^I/\de$ realizes all 1-types $q(x)$ such that $|q(x)| \leq \lambda$ and $Q(x) \in q(x)$.
\end{enumerate}
\end{defn}

\begin{rmk} \label{tq}
Since ultrapowers commute with reducts, Section $\ref{s:defns}$ Theorem $\ref{commute-with-reducts}$,
for any regular ultrafilter $\de$ on $I$ and any model $M$ whose theory has $SOP_2$, we clearly have that
$N = M^I/\de$ realizes all $SOP_2$-$\mu$-types if and only if $\de$ $(\mu^+, Q)$-saturates $\ts$. In what follows, we will use these two
presentations interchangeably.
\end{rmk}

Note that the goal of Definition \ref{canonical} is simply to standardize the presentation of $SOP_2$-types, which are focused on the single
formula $\psi$ (in the case of that definition, $xRy$); in particular, it makes no claim to have constructed a minimally complex $SOP_2$ theory from
any point of view other than that of capturing the necessary $xRy$-types. Recall the definition of distribution, \ref{dist} above.

Having set the stage,  
we state a simple criterion for a regular ultrafilter $\de$ to realize $SOP_2$, in terms of upper bounds for
increasing sequences in trees.

\begin{lemma} \label{equiv-conds} \emph{($SOP_2$-types and treetops)}
Let $|I| = \lambda$.
Let
\[ \mathbf{P} = \{ p :p = \{ \psi(x;a_i) : i < \lambda \} ~\mbox{is an $SOP_2$-type in $N = M^I/\de$, $|p| \leq |I|$}\} \]
Then the following are equivalent:
\begin{enumerate}
\item Each $p \in \mathbf{P}$ is realized in $N$.
\item Each $p \in \mathbf{P}$ has a distribution $d$ such that for $\de$-almost all $s$, for all $i, j < \lambda$,
\[ s \in d(i) \cap d(j) \mbox{ implies } \left( a_i[s] \tlf a_j[s]\right) \lor \left(a_j[s] \tlf a_i[s] \right) \]
\item if $(\mct, \tlf_\mct)$ is any tree 
and $\langle c_i : i < \lambda \rangle$ is a $\tlf_\mct$-increasing sequence
in $N_\mct := (\mct, \tlf_\mct)^I/\de$, then $\langle c_i : i < \lambda \rangle$ has an upper bound in $N_\mct$. That is, there exists
$c \in N_\mct$ such that $i < \lambda$ implies $c_i \tlf_\mct c$.
\end{enumerate}
\end{lemma}

\begin{proof}
(1) $\rightarrow$ (2) Let $p$ be given, let $\alpha \in N$ be a realization of $p$,
and let $\langle X_i : i < \lambda \rangle$ be a $\lambda$-regularizing family for $\de$.
Then the distribution $d: \fss(\lambda) \rightarrow \de$ given by:
\begin{itemize}
\item $\{ i \} \mapsto \{ s \in I : M \models \psi(\alpha[s], a_i(s)) \} \cap X_i $
\item for $|u| > 1$, $u \mapsto \bigcap \{ d(\{i\}) : i \in u \}$
\end{itemize}
satisfies the criterion (2) by definition of $SOP_2$.

(2) $\rightarrow$ (1) For any given $p \in \mathbf{P}$,
if (2) holds then it is easy to define a realization $\alpha[s]$ in each index model by definition of $SOP_2$, and
any $\alpha \in N$ with $\alpha = \prod_s \alpha[s] \mod \de$ will realize the type by \lost theorem.

(3) $\rightarrow$ (2) Let $(P, \tlf)$ be the (infinite) $SOP_2$-tree in $M \models \ts$.  
Let $p \in \mathbf{P}$ be given, where $p = \{ \psi(x;a_i) : i < \mu \}$. Then the sequence
$\langle a_i : i < \lambda \rangle$ is $\tlf$-increasing and thus has an upper bound $c$ in $M^I/\de$.
Then the distribution $d$ given by
\[\{i\} \mapsto \{ s \in I : a_i[s] \tlf c[s] \} \]
and for $|u| > 1$, $u \mapsto \bigcap \{ d(\{i\}) : i \in u \}$, satisfies (2) by definition of $SOP_2$: 
any consistent set of instances must lie along a branch. 

(2) $\rightarrow$ (3) By compactness, we may suppose that $(P, \tlf)$ contains an $\omega$-branching tree of height $\omega$. 
We would like to realize the type $\{ x > c_i : i < \lambda \}$ in $(\mct, \tlf_\mct)^I/\de$. 
Fix some distribution $d_\mct$ of this type.  Then $d_\mct$ assigns finitely many
formulas to each index model, and we may build a corresponding $SOP_2$-type $p=\{ xRa_i : i < \lambda \}$ by copying the patterns
at each index: for each $s \in I$ let $a_i[s] \tlf a_j[s]$ if and only if $c_i[s] \tlf_\mct c_j[s]$, and then choose each $a_i \in M^I$
so that $a_i = \prod_{s \in I} a_i[s] \mod \de$. Since $\{ c_i : i < \lambda \}$ is $\tlf_\mct$-increasing, 
by \lost theorem $p$ will be a consistent $SOP_2$-type,
so will, by assumption, have a distribution $d$ satisfying (2). Then $d$ naturally refines $d_\mct$ and gives a distribution in which
for each $s \in I$, the set $C[s] := \{ c_i[s] : i < \lambda, s \in d(\{i\}) \}$ is finite and $\tlf_\mct$-linearly ordered in 
$(\mct, \tlf_\mct)$. Choose $c \in (\mct, \tlf_\mct)^I$ so that 
$c[s]$ is the $<$-maximum element of $C[s]$ in each index model, and $c/\de$ will be an upper bound by \lost theorem.
\end{proof}

\begin{rmk}
In Lemma \ref{equiv-conds}(2) $\rightarrow$ (3), it is $SOP_2$ rather than simply the tree property which is used.
\end{rmk}

On the level of theories, treetops therefore gives a necessary condition for saturation:

\begin{cor} \label{sop2-treetops}
Let $\de$ be a regular ultrafilter on $\lambda$ and suppose that $\de$ saturates some theory with $SOP_2$. Then $\de$ has
$\lambda^+$-treetops.
\end{cor}

\begin{proof}
This simply translates Lemma \ref{equiv-conds} via Remark \ref{tq}.
\end{proof}

\begin{concl} \label{sop2-concl}
Let $\de$ be a regular ultrafilter on $I$, $|I| = \lambda$. 
Then recalling Definition \ref{conventions}, the following are equivalent:
\begin{enumerate}
\item $\de$ has $\lambda^+$-treetops.
\item $\mc(\de)$ contains no symmetric cuts.
\item $\de$ realizes all $SOP_2$-types over sets of size $\lambda$.
\end{enumerate}
\end{concl}

\begin{proof}
$(1) \iff (2)$: Lemma \ref{sym-cuts}.

$(1) \iff (3)$: Lemma \ref{equiv-conds}.
\end{proof}

\begin{rmk}
Clearly in \ref{sop2-concl}, i.e. in its constituent claims, 
one could separate the role of $\lambda$ from the size of the index set, using 
$\cts$ instead of $\mc(\de)$. 
\end{rmk}

This yields: 

\begin{theorem-m1} \label{concl:sop2-max}
Every theory with $SOP_2$ is maximal in Keisler's order.  
\end{theorem-m1}

\begin{proof}
By Keisler's characterization, Section \ref{s:defns} Theorem \ref{good-max}, it suffices to show that 
if $\de$ is a regular ultrafilter on $I$, $|I| = \lambda$ and $M \models T$ then
$M^I/\de$ is $\lambda^+$-saturated only if $\de$ is $\lambda^+$-good.  

By Conclusion \ref{sop2-concl}, a necessary condition for any regular ultrafilter $\de$ on $\lambda$ 
to saturate $T$ is that $\de$ have $\lambda^+$-treetops. By Theorem \ref{maximal-uf}, any regular ultrafilter
$\de$ on $\lambda$ with $\lambda^+$-treetops must be $\lambda^+$-good. 
So a necessary condition for $\de$ to saturate $T$ is that $\de$ be good, which completes the proof. 
\end{proof}

\begin{disc} \label{c:evidence}
\emph{To conclude this section, we review some evidence for Conjecture \ref{conj:a} from the introduction, which says that $SOP_2$ characterizes maximality in Keisler's order.  
Any non-simple theory either contains a minimally inconsistent tree, called $TP_2$, or a maximally inconsistent tree, called $TP_1$, or both (\cite{Sh:a} Theorem III.7.11). $TP_1$ may be identified with $SOP_2$, the lowest level of the so-called $SOP_n$ hierarchy of properties whose complexity, in some sense, approaches that of linear order as $n$ grows.  (Considered as a property of formulas, $SOP_2$ is much weaker than $SOP_3$; it is open whether they coincide for first order complete $T$.) Briefly, then, the move from $SOP_3$ to $SOP_2$ moves Keisler's order out of the territory of the $SOP_n$ hierarchy onto what appears to be a major dividing line for which there are strong general indications of a theory. $NSOP_2$ (=not $SOP_2$) is in some senses, close to simplicity; we hope to develop this theory in light of our work here, leveraging the tool of Keisler's order. $NSOP_2$ and Conjecture \ref{conj:a} also connect to 
work of D\v{z}amonja-Shelah \cite{DzSh} and Shelah-Usvyatsov \cite{ShUs} on a weaker, related ordering; there it was shown, for instance, that under GCH $NSOP_2$ is necessarily non-maximal in that ordering (\cite{ShUs} 3.15(2)) thereby strengthening the case for Conjecture \ref{conj:a}.}

\emph{Finally, there is a key analogy in this case between the independence/strict order dichotomy for non-stable theories and the $TP_2/SOP_2$ dichotomy for non-simple theories.  
There is a Keisler-minimum unstable theory, the random graph, and as already noted strict order implies maximality. By a theorem of Malliaris \cite{mm4}, quoted below in Section \ref{s:tfeq} as Theorem \ref{tfeq-m}, there is a Keisler-minimum theory among the theories with $TP_2$.  
In Section \ref{s:tfeq} below, we apply Theorem \ref{concl:sop2-max} to prove that this theory is indeed a minimum non-simple theory in Keisler's order.}
\end{disc}

\section{If $\de$ is good for some non-simple theory then $\de$ is flexible} \label{s:homogeneous}

In this section, we connect treetops to  several key model-theoretically meaningful properties of ultrafilters. 
The main result is that any ultrafilter which is good for some non-simple theory must be flexible (defined below), Conclusion \ref{flex-not-simple}. 
Also, with an eye to Conjecture \ref{conj:a}, we further develop the picture of internal maps between sequences from Corollary \ref{bijection}. 

It is known that Keisler classes other than the maximal class may be 
characterized by properties of filters, for example: 

\begin{defn}
Let $\de$ be a regular ultrafilter on $I$, $|I| = \lambda$. Say that $\de$ has \emph{$2$-separation} if whenever
$N = M^I/\de$ is infinite and $A, B \subseteq N$ are disjoint sets of size $\leq \lambda$, there is an internal predicate $P$
which contains $A$ and is disjoint from $B$.
\end{defn}

\begin{fact} \label{rg-good} \emph{(The Keisler-class of the random graph, see e.g. \cite{MiSh:996})}
There is a minimum class in Keisler's order among the unstable theories, which includes the theory of the random graph.
It can be characterized set-theoretically as the class of countable complete theories which are saturated precisely
by regular ultrafilters with $2$-separation.
\end{fact}

\begin{cor}
If $\de$ is a regular ultrafilter on $\lambda$ which has $\lambda^+$-treetops, then $\de$ has $2$-separation. 
\end{cor}

\begin{proof}
By Conclusion \ref{must-be-good} and Theorem \ref{maximal-uf}. 
\end{proof}

The following definition is standard for ultrapowers and coincides with the definition for CSPs when $\de$ has treetops. 

\begin{defn} \emph{(Lower cofinality)} \label{d:lcf}
Let $\de$ be a regular ultrafilter on $I$ and let $\kappa \leq \lambda = |I|$ be a regular cardinal. Let $N = (\lambda, <)^I/\de$.
The \emph{lower cofinality of $\kappa$ modulo $\de$}, written $\lcf(\kappa,\de)$, is the cofinality of the set
$\{ a \in N : \zeta \in \kappa \mbox{ implies } N \models a > \zeta \}$ considered with the reverse order type.
In other words, it is the smallest regular cardinal
$\rho$ so that there is a $(\kappa, \rho)$-cut in $N$ half of which is given by the diagonal embedding of $\kappa$.
This is also called the \emph{coinitiality} of $\kappa$ with respect to $\de$. 
\end{defn}

Theorem \ref{m2} (Uniqueness) need not hold for regular ultrapowers without the assumption of treetops: the following theorem gives a
family of examples where it will fail.

\begin{thm-lit} \emph{(Shelah \cite{Sh:c} Theorem VI.3.12 p. 357)}
Suppose $\aleph_0 = \lambda_0 < \lambda_1 < \dots < \lambda_n = \lambda^+$, each $\lambda_i$ is regular,
and for $\ell < n$, $\mu_\ell$ are regular such that $\lambda_{\ell+1} \leq \mu_\ell \leq 2^\lambda$. Then
for some regular $\lambda_1$-good ultrafilter $\de$ on $\lambda$, $\lcf(\kappa, \de) = \mu_\ell$ whenever
$\lambda_\ell \leq \kappa < \lambda_{\ell+1}.$
\end{thm-lit}

These results were generalized by Koppelberg \cite{koppelberg}.

\begin{defn} \emph{(Good for equality, Malliaris \cite{mm4})} \label{good-for-equality}
Let $\de$ be a regular ultrafilter. Say that $\de$ is \emph{good for equality}
if for any set $X \subseteq N = M^I/\de$,
$|X| \leq |I|$, there is $d: X \rightarrow \de$ such that for any $a,b \in X$, $t \in \lambda, t \in d(a) \cap d(b)$ implies that
$(M \models a[t] = b[t]) \iff (N \models a = b)$.
\end{defn}

In the language of homogeneity, Malliaris had shown that the minimum $TP_2$-class is precisely
the class of theories saturated by ultrafilters on $\lambda$ whose ultrapowers
admit an internal bijection between any two disjoint sets of size $\leq \lambda$. (That is, 
for any two enumerations $\langle a_i : i < \lambda \rangle$, $\langle b_i : i < \lambda \rangle$ of 
small sets in the ultrapower $N$, there is an internal bijection $f: N \rightarrow N$ such that $f(a_i) = (b_i)$ for $i<\lambda$. 
There is no harm in replacing 'bijection' with partial injection whose domain contains the desired set.)  
The name \emph{good for equality} reflects that these maps are not assumed to preserve any additional structure. 
So among the non-simple theories, we have on the one hand a class of theories characterized by ultrapowers 
admitting internal partial injections which preserve equality (in the sense of Definition \ref{good-for-equality}), and on the other hand 
a class of theories characterized by ultrapowers admitting internal partial injections which preserve order (in the sense 
of Corollary \ref{bijection}). Towards a possible separation between these classes, 
we investigate the relative strength of these hypotheses. 

\begin{cor} \label{c:txe}
If $\de$ is a regular ultrafilter on $\lambda$ which has $\lambda^+$-treetops, then $\de$ is good for equality. 
\end{cor}

\begin{proof} 
Conclusion \ref{must-be-good} and Theorem \ref{maximal-uf}. 
\end{proof}

\begin{rmk}
If Conjecture \ref{conj:a} is true, then the converse to Corollary \ref{c:txe} is false, i.e. if $\de$ is a regular ultrafilter on $\lambda$ 
which is good for equality then it need not have $\lambda^+$-treetops.
\end{rmk}

\begin{proof}
Malliaris had shown that the Keisler class of the theory $\tfeq$ of a parametrized family of independent equivalence relations may be characterized 
as the class of theories saturated precisely by those regular ultrafilters which are good for equality 
(quoted below as Section \ref{s:tfeq} Theorem \ref{tfeq-m}).  
This theory does not have $SOP_2$. If Conjecture \ref{conj:a} is true, then $\tfeq$ is not in the maximal class, so any ultrafilter able to saturate it 
will be good for equality but not good, so (by Theorem \ref{maximal-uf}) will not have $\lambda^+$-treetops. 
\end{proof}

\begin{cor}[of \ref{bijection}]  \label{r10}
Suppose that $\de$ has $\lambda^+$-treetops, and let $\kappa = \cf(\kappa) \leq \lambda$. Let $M = (\lambda, <)$.
Let $A = \langle a_i : i < \kappa \rangle$ be any strictly increasing, $\kappa$-indexed sequence of $N = M^\lambda/\de$.
Then in $N$ there is an internal partial injection $f$ which takes $\kappa$ to the the diagonal embedding of $\kappa$.
\end{cor}

\begin{cor} \label{true} \emph{(of \ref{r10})}
Suppose that $\de$ has $\lambda^+$-treetops, and let $\kappa = \cf(\kappa) \leq \lambda$. Then every strictly
increasing (or strictly decreasing) $\kappa$-indexed sequence has a distribution, Definition \ref{dist}, which is good for equality.
\end{cor} 

\begin{proof}
The diagonal embedding of $\kappa$ has such a distribution.
\end{proof}

\begin{defn} \emph{(near-$\kappa$-indexed)}
Let $M$ be an infinite model, $\de$ a regular ultrafilter on $\lambda$, $N = M^\lambda/\de$ and
$\kappa$ a regular cardinal.
Say that the set $A \subseteq N$, $|A| = \kappa$ is \emph{near-$\kappa$-indexed}
there exists an internal linear order on $N$
under which $A$ is either monotonic increasing or monotonic decreasing of order-type $\kappa$.
\end{defn}

In the context of ultrapowers, what the proof of Corollary \ref{bijection} actually shows is the following slightly stronger statement:

\begin{obs} \label{obs-x}
If $\de$ is a regular ultrafilter on $\lambda \geq \kappa$
with $\lambda^+$-treetops, then any near-$\kappa$-indexed set $X$ in any $\de$-ultrapower has a distribution which is good for equality.
\end{obs}

Let $N = M^I/\de$, $|I| = \lambda$ and suppose $\de$ has $\lambda^+$-treetops, $\kappa \leq \lambda$.
Then Claim \ref{m2} shows that the coinitiality of any two $\kappa$-indexed sequences in the ultrapower is the same. Thus if $\mc(\de)$ contains
\emph{some} $(\kappa, \theta)$-cut, \emph{every} monotonic $\kappa$-indexed sequence will represent half of a $(\kappa, \theta)$-cut. This is a strong
omission of types. It does not contradict the universality of regular ultrapowers since elementary embeddings need not preserve cuts.

How strong is the assumption that all sets are near-$\kappa$-indexed? 

\begin{concl} \label{eq-ki} When $\de$ is a regular ultrafilter on $\lambda$, then 
$(1)$ implies $(2)$, where:
\begin{enumerate}
\item $\de$ is good for equality
\item For any infinite model $M$, $N = M^\lambda/\de$, we have that any $A \subseteq N$, $|A| = \kappa \leq \lambda$ is near-$\kappa$-indexed.
\end{enumerate}
If in addition $\de$ has $\lambda^+$-treetops, then $(2)$ implies $(1)$.
\end{concl}

\begin{proof} 
The last line holds by Observation \ref{obs-x}, so we prove $(1)$ implies $(2)$.

Without loss of generality, $M$ is a two sorted structure one side of which contains an infinite set (from which we choose $A$),
the other side of which contains $(\lambda, <)$.
Fix an enumeration $\pi: \kappa \rightarrow A$ of $A$.
Let $K = \langle k_i : i < \kappa \rangle$ be the image of the diagonal embedding of $\kappa$ in $N$,
so $k_i = {^\lambda \{ i \}}$.
Choose a distribution $d_A: A\rightarrow \de$ which is good for equality.
Let $d_\kappa: K \rightarrow \de$ be the distribution given by
$d_\kappa(k_i) = d_A(\pi(i))$, which will be good for equality by definition.
Now simply expand each index model $M[t]$ by adding a linear order $<_*$ to the first sort in such a way that
the existential $<_*$-type of $\{ a[t] : a \in A, t \in d_A(a) \}$ is the same as the existential
$<$-type of $\{ i : t \in d_\kappa(k_i) \}$.

Then in $N$ the order $<_*$ on $A$ will agree with the order $<$ on the true copy of $\kappa$,
as described by the given enumeration.
\end{proof}

Finally, the following property was introduced by Malliaris in \cite{mm-thesis} and studied by Malliaris and by Malliaris and Shelah in  \cite{mm4}, \cite{MiSh:996}.
It is equivalent to $\lambda$-OK, see \cite{MiSh:996} Section 6.

\begin{defn} \emph{(Flexible filters, \cite{mm-thesis})} \label{d:flexible}
We say that the filter $\de$ is $\lambda$-flexible if for any $f \in {^I \mathbb{N}}$ with
$n \in \mathbb{N}$ implies $n <_{\de} f$, we can find $X_\alpha \in \de$ for $\alpha < \lambda$ such that
for all $t \in I$
\[ f(t) \geq | \{ \alpha : t \in X_\alpha \}|\]
Informally, we can find a $\lambda$-regularizing family below any nonstandard integer.
\end{defn}

\begin{fact} \emph{(Malliaris \cite{mm4})} \label{tp2-flex}
Suppose that $\de$ is regular and $T$ is non-low or has $TP_2$. If $\de$ saturates $T$, then $\de$ must be flexible.
\end{fact}

We now complete ``$T$ is non-low or has $TP_2$'' to all non-low or non-simple theories. 
Note that any regular ultrafilter on $\lambda$ which saturates \emph{some} unstable theory must satisfy 
$\lcf(\aleph_0, \de) \geq \lambda^+$ by \cite{Sh:a} VI.4. Alternately, one can derive this condition from Theorem \ref{maximal-uf}. 
Indeed one can also derive Claim \ref{tree-flex} from that theorem, but it is interesting to prove it directly. 

\begin{claim} \label{tree-flex}
If $\de$ is a regular ultrafilter on $\lambda$, $\lcf(\aleph_0, \de) \geq \lambda^+$ and $\de$ has $\lambda^+$-treetops,
then $\de$ is flexible.
\end{claim}

\begin{proof}
Let $M = (\lambda, <)$ and let $N = (\lambda, <)^I/\de$.
Let some $\de$-nonstandard integer $n_*$ be given. We would like to show that there is a regularizing family below $n_*$.
By hypothesis, $\lcf(\aleph_0,\de) \geq \lambda^+$ so there is $B \subseteq N \setminus \mathbb{N}$, $B = \langle b_i : i < \lambda \rangle$
such that $i<j<\lambda$, $m<\omega$ implies $m < b_j < b_i < n_*$.
By Corollary \ref{bijection} and Corollary \ref{true}, there is a distribution $d$ of $B$ which is good for equality,
that is, for all $b, b^\prime \in B$, and all $t \in I$,
\[ t \in d(b) \cap d(b^\prime) \mbox{ implies } \left( M \models b[t] = b^\prime[t] \iff N \models b = b^\prime \right) \]
Let $\{ X_b : b \in B \} \subseteq \de$ be given by $X_b = d(b)$.
By choice of $B$ and goodness for equality, $\{ X_b : b \in B \}$ is a regularizing family and by choice of $d$, it is below $n_*$, which completes
the proof.
\end{proof}

\begin{concl} \label{flex-not-simple} \label{concl-flex}
Suppose $\de$ is a regular ultrafilter on $\lambda$ and $\de$ saturates some non-low or non-simple theory. Then $\de$ is flexible.
\end{concl}

\begin{proof}
Suppose $\de$ saturates some theory $T$ which is not low or not simple. If $T$ is not low, apply Fact \ref{tp2-flex}. So we may assume $T$ is not simple. 
We know from \cite{Sh:a} Theorem III.7.11 that any non-simple theory will have either $TP_2$ or $SOP_2$. 
In the case where $T$ has $TP_2$, Fact \ref{tp2-flex} applies again. Suppose then that $T$ has $SOP_2$.  
By Corollary \ref{sop2-treetops}, $\de$ has $\lambda^+$-treetops, so by Theorem \ref{maximal-x} we know that $\lcf(\aleph_0, \de) \geq \lambda^+$. 
Then the hypotheses of Claim \ref{tree-flex} are satisfied, and $\de$ is flexible. 
\end{proof}

\section{There is a minimum non-simple class in Keisler's order} \label{s:min-non-simple} \label{s:tfeq}
The main result of this section is:

\begin{theorem} \label{t:tfeq}
There is minimum class among the non-simple theories in Keisler's order, which
contains the theory ~$\tfeq$~ of a parametrized family of independent equivalence relations.
\end{theorem}

We will build on the following theorems.
(The reader may take the property $TP_2$, the tree property of the second kind, to be a black box in the following results.)

\begin{thm-lit} \emph{(Shelah \cite{Sh:a} III.7.11, in our language)} \label{tp2-sop2}
Every non-simple theory has either $TP_2$ or $SOP_2$ $($equivalently, $TP_1$$)$.
\end{thm-lit}

\begin{thm-lit} \emph{(Malliaris \cite{mm4} Theorems 6.9--6.10, and Malliaris \cite{mm5} Theorem 5.21)} \label{tfeq-m}
There is minimum class among the theories with $TP_2$ in Keisler's order, which
contains the theory $\tfeq$ of a parametrized family of independent equivalence relations.

Moreover, for a regular ultrafilter $\de$
on $\lambda$, the following are equivalent:
\begin{enumerate}
\item $\de$ saturates $\tfeq$ 
\item $\de$ is good for equality, Definition $\ref{good-for-equality}$ above
\item for any $N = M^\lambda/\de$ and any two disjoint sets $\{ a_i : i < \lambda \}$, $\{ b_i : i < \lambda \}$
of elements of $N$, listed without repetition, there exists an internal partial injection $f$ such that
for all $i< \lambda$, $f(a_i) = b_i$.
\end{enumerate}
\end{thm-lit}

The tools of the present paper will combine to give two distinct proofs for Theorem \ref{t:tfeq}. 
First, in light of \ref{concl:sop2-max}, i.e. the maximality of $SOP_2$ in Keisler's order,
Theorem \ref{t:tfeq} follows from Theorem \ref{tfeq-m} and Theorem \ref{tp2-sop2} by virtue of collapsing the $SOP_2$ case.

However, it is possible to give a direct and more illuminating proof of this theorem, as we do below. 
The direct proof is not specific to ultrapowers, but holds of any cofinality spectrum problem
allowing an endless $X_\ma$, i.e. one in which there is no bound $d_\ma$ $($which extends the context of regular ultrapowers$)$. 
The proof itself is deferred until after Claim \ref{c:tfeq-c}.  It supposes a
reduction \ref{z10}, an example of the phenomenon of reduction to few asymmetric cuts.
Historically, this reduction was the turning point of our argument. Since
the reduction is now trivially true by the umbrella Theorem \ref{last-cut}, we do not give a separate proof of \ref{z10}.

\begin{red} \label{z10}
If $\cs$ is a cofinality spectrum problem, to show that $\mc(\cs, \xt_\cs) \neq \emptyset$,
it suffices to rule out the case of a $(\kappa, \kappa^+)$-cut where $\kappa^+ = \xp_\cs$. 
$[$See Discussion \ref{big-picture}.$]$ 
\end{red}

\begin{cor} \label{m-cor} \emph{(of Theorem \ref{tfeq-m})}
If we consider the regular ultrapower $M_1 = M^\lambda/\de$ as a cofinality spectrum problem
$(M, M_1, \dots)$ then to show that $\de$ saturates $\tfeq$ it suffices to verify
Theorem \ref{tfeq-m}(3) in the case where the sets are taken from $X_\ma$ for
\emph{some} pseudofinite $\ma \in \ord(\cs)$.
\end{cor}

\begin{proof}
In Theorem \ref{tfeq-m}, by regularity of $\de$ and the fact that ultrapowers commute with reducts there
is no harm in assuming $M = (\mathbb{N}, <)$, so $N = M^\lambda/\de$, in which case
this is precisely what is proved.
\end{proof}

Recall also that by Fact \ref{rg-good}, the theory $\trg$ of the random graph is minimum in Keisler's order
among the unstable theories, and $\de$ saturates $\trg$ if and only if it has so-called 2-separation
(disjoint sets of size $\leq \lambda$ in the ultrapower can be separated by an internal set, see \ref{rg-good}).
For an arbitrary c.s.p. this means:

\begin{defn} \emph{($2$-separation for $\cs$)} \label{2-sep}
Let $\cs$ be a cofinality spectrum problem. We say that
\emph{$\cs$ has $2$-separation} if for any $\ma \in \ord(\cs)$ and any two disjoint sets
$A, B \subseteq X_\ma$ with $|A| + |B| \leq \xp_\cs$, there is a definable $X \in M_1$ such that
$A \subseteq X$ and $B \cap X = \emptyset$.
\end{defn}

\begin{lemma} 
\label{c:tfeq}
Let $\cs$ be a cofinality spectrum problem and suppose that:
\begin{enumerate}
\item $\cs$ has $2$-separation
\item $\kappa^+ = \lambda = \xp_\cs < \xt_\cs$
\item $\ma \in \ord(\cs)$ is endless, i.e. there is no bound $d_\ma$  
\item $\{ d_\epsilon : \epsilon < \lambda \}$, $\{ e_\epsilon : \epsilon < \lambda \}$ are
disjoint subsets of $X_\ma$, listed without repetition
\end{enumerate}
Then there is in $M_1$ a bijection $f: X_\ma \rightarrow X_\ma$ such that
$f(d_\epsilon) = e_\epsilon$ for all $\epsilon < \lambda$.
\end{lemma}

\begin{proof} 
We begin with $\ma \in \ord(\cs)$, $\{ d_\epsilon : \epsilon < \lambda \}$,
$\{ e_\epsilon : \epsilon < \lambda \}$ be as given. We define
\begin{align*}
\mci = \{ S \subseteq \lambda ~&: ~\mbox{there is $f$ such that $M_1 \models$ ``$f$ is a 1-1 function from $X_\ma$ to $X_\ma$''} \\
& \mbox{and} ~ \epsilon \in S_1 \iff M_1 \models f(d_\epsilon) = e_\epsilon \}  
\end{align*}

First, we verify that if $M_1 \models$ ``$f$ is a 1-1 function with $\dom(f) \subseteq X_\ma$ and $\rn(f) \subseteq X_\ma$'' then
\[ \{ \epsilon < \lambda : f(d_\epsilon) = e_\epsilon \} \in \mci \]
Equivalently, this holds if $f \in M_1$ is a 1-1 function on some $M_1$-definable subset of $X_\ma$.

We can define from $f$ a total function $g$ as follows. Necessarily the domain and range of $f$ are definable sets. Let
$g: X_\ma \rightarrow X_\ma$ be defined by:

\[ g(a) =
	 \begin{cases}
	 f(a) & \text{if } ~a \in \dom(f) \\
	f^{-1}(a)	& \text{if } ~a \in \rn(f) \setminus \dom(f)\\
	a& \text{if } ~a \in X_\ma \setminus \dom(f) \setminus \rn(f) \\
		 \end{cases}\]

Second, we prove that $\mci$ is an ideal on $\lambda$, by checking the conditions for an ideal. In light of the above it suffices to specify bijections
on definable subsets of $X_\ma$.

(a) $\mci \subseteq \mcp(\lambda)$ by definition.

(b) $\emptyset \in \mci$ trivially.

(c) Suppose $S_1 \subseteq S_2 \subseteq \lambda$ and $S_2 \in I$, exemplified by $f_2$.
By $2$-separation, there is a set $X \in M_1$ such that $M_1 \models X \subseteq X_\ma$ and
for all $\epsilon < \lambda$,
\[ \left( M_1 \models d_\epsilon \in X \right) \iff \left( \epsilon \in S \right) \]
Let $f_1 = f_2 \rstr_Y$ where $Y = \dom(f_2) \cap X$. Then $f_1$ witnesses that $S_1 \in \mci$.
So $\mci$ is downward closed.

(d) If $S_1, S_2 \in \mci$ then by (c) the set $S^\prime_2 = S_2 \setminus S_1$ belongs to $\mci$.
Let $f_1, f^\prime_2$ witness that $S_1, S^\prime_2$ belong to $\mci$, respectively.
Define $f_2 \in M_1$ by: $\dom(f) = \dom(f_1) \cup \dom(f^\prime_2)$, and
\[ f_2(a) =
\begin{cases}
		f_1(a) & \text{if } ~a \in \dom(f_1) \\
		f^\prime_2(a)	& \text{if } ~a \in \dom(f_2) \setminus \dom(f_1)
\end{cases}\]
Iterating, we have that $\mci$ is closed under finite union.

It follows from (a)-(d) that $\mci$ is an ideal on $\lambda$.

Third, we prove closure for $\cf(\delta) < \kappa$. 
Let $\mb \in \ord(\cs)$ be such that $X_\mb = X_\ma \times X_\ma$.
Let $(\mct, \tlf)$ be the subset of $\mct_\mb$ consisting of elements $c$ whose range
is the graph of a partial injection from $X_\ma$ to $X_\ma$, i.e.

\begin{itemize}
\item for each $n < \lgn(c)$, $\tc(n) \in X_\ma \times X_\ma$; as before, we denote these values by
$\tc(n,0)$ and $\tc(n,1)$ respectively
\item the set $\{ \tc(n) : n \leq \maxdom(c)) \}$ is the graph of a partial injection from $X_\ma$ to $X_\ma$
\item if $\tc_1 \tlf \tc_2$ then $\tc_2$ extends $\tc_1$, considered as a function
\end{itemize}

\br
\noindent \emph{Note that 
for elements $\tc \in \mct$, we will refer extensively to the
function whose graph is $\{ (\tc(n,0), \tc(n,1)) : n \leq \maxdom(\tc)) \}$. For purposes of clarity, we denote this
function by $\fun(\tc)$ and will say e.g. that $d_\epsilon \in \fun(\tc)$ or that $\fun(\tc)(d_\epsilon) = e_\epsilon$.}

\begin{scl} Suppose that $\delta$ is an ordinal with $\cf(\delta) < \kappa$ and we are given $\tc_\alpha, n_\alpha, S_\alpha$
($\alpha < \delta$) and a set $S$, such that:
\begin{itemize}
\item for each $\alpha < \delta$, $M_1 \models \tc_\alpha \in \mct$
\item $M_1 \models \tc_\beta \tlf \tc_\alpha$ for each $\beta < \alpha < \delta$
\item for each $\alpha < \delta$, $n_\alpha = \lgn(\tc_\alpha) - 1$
\item for each $\alpha < \delta$, $S_\alpha = \{ \epsilon < \lambda : d_\epsilon \in \dom(\fun(\tc_\alpha)) \}$, i.e.
$ = \{ \epsilon < \lambda : (\exists n \leq n_\alpha) (\tc_\alpha(n,0) = d_\epsilon) \}$
\item for each $\alpha < \delta$, $\epsilon \in S_\alpha$ implies $\fun(\tc_\alpha)(d_\epsilon) = e_\epsilon$
\item $S = \bigcup \{ S_\alpha : \alpha < \delta \}$
\end{itemize}
\emph{Then} there is $\tc_*$ such that:
\begin{enumerate}
\item $M_1 \models \tc_* \in \mct$
\item $M_1 \models \tc_\alpha \tlf \tc_*$ for $\alpha < \delta$
\item for $\epsilon < \lambda$,
\[d_\epsilon \in \dom( \fun(\tc_*) ) \iff\epsilon \in S \iff \fun(\tc_*)(d_\epsilon) = e_\epsilon\]
\end{enumerate}
In particular, $\fun(\tc_*)$ witnesses that $S \in \mci$.
\end{scl}

\begin{proof}
Since $\cf(\delta) < \kappa < \xt_\cs$, by Treetops there is $\tc \in \mct$ satisfying (1)-(2).
Now the sequence $\langle n_\alpha : \alpha < \delta \rangle$ represents the left half of some cut $(C_1, C_2)$ of $\dom(\tc)$.

Let $\langle m_\beta : \beta < \theta \rangle$ represent the right half of this cut.

By Reduction \ref{z10}, $\cf(\delta) < \kappa < \kappa^+ = \lambda$ implies $\lcf(\delta) > \lambda$,
as otherwise there would be a corresponding cut in $\tc(\de)$. Thus, $\theta > \lambda$.

Now for each $\epsilon \in \lambda \setminus S$ there is $\beta(\epsilon) < \theta$ such that $d_\epsilon \notin \dom(\tc(m_\beta)$.
Let $\beta = \sup \{ \beta(\epsilon) : \epsilon \in \lambda \setminus S \}$. Since $\theta > \lambda$, $\beta > \alpha$ for all
$\alpha < \delta$. Then $\tc\rstr m_\beta$ is the desired $\tc_*$.
\end{proof}
This completes the proof of closure. 

Fourth, we prove that the ideal $\mci$ is $\lambda$-complete, and contains each $\{ \alpha \}$ for $\alpha \in \lambda$. 
The claim about containing the singletons is trivial since we know it is an ideal on $\lambda$. So let us show that $\mci$ is $\lambda$-complete.

Let $\langle S_\alpha : \alpha < \delta \rangle$ be an increasing sequence of elements of $\mci$, with $\delta < \lambda$.
Without loss of generality $\delta = \cf(\delta)$; call it $\theta$. So $\theta \leq \kappa$.
By induction on $\alpha < \kappa$ we choose a sequence $\langle \tc_\alpha : \alpha < \kappa \rangle$ just as before.
The subclaim proved in above says precisely that we can continue the induction for all $\alpha < \kappa$, and we now
address the case of $\kappa$, i.e. $\theta$.

That is, having chosen $\langle \tc_\alpha : \alpha < \theta \rangle$, as $\kappa < \lambda \leq \xt_\cs$ we have $\kappa^+$-treetops
so may choose an upper bound $\tc$. Let $n = \maxdom(\tc))$, and by definition $\fun(\tc)$ is a $1$-$1$ function.

Thus by definition $\fun(\tc)$ witnesses that $S \in \mci$ where $S = \{ \epsilon : \fun(\tc)(d_\epsilon) = e_\epsilon \}$.
By choice of $\tc$ as an upper bound, $\bigcup \{ S_\alpha : \alpha < \theta \} \subseteq S \in \mci$,
and thus necessarily $\bigcup \{ S_\alpha : \alpha < \theta \} \in \mci$ by (c) above. This finishes the proof that the ideal is $\lambda$-complete.

Fifth, we prove that $\kappa^+ = \lambda$ is the union of $\kappa$ sets from $I$. 
By induction on $\alpha < \lambda$ we choose $(\tc_\alpha, n_\alpha)$ such that:

\begin{itemize}
\item $\tc_\alpha \in \mct$, i.e. it represents an increasing pseudofinite sequence of 1-1 functions
from $X_\ma$ to $X_\ma$
\item $n_\alpha = \lgn(\tc) - 1$
\item $\beta < \alpha$ implies $\tc_\beta \tlf \tc_\alpha$
\item if $\epsilon < \beta < \alpha$ and $n \in (n_\epsilon, n_\beta]$ then $d_\epsilon \in \dom(\fun({\tc_\alpha\rstr_n})$ and
$\fun({\tc_\alpha\rstr_n})(d_\epsilon) = e_\epsilon$
\end{itemize}

The induction is straightforward. For $\alpha=0$, let $\tc_\alpha = \emptyset$.

For $\alpha = \beta + 1$, let $n_{\alpha + 1} = n_\beta + 1$, and $\tc_\alpha$ is determined by asking that
$\dom(\fun(\tc_\alpha)) = \dom(\fun(\tc_\beta)) \cup \{ d_\beta \}$, that $d \in \dom(\fun(\tc_\beta))$ implies
$\fun(c_\alpha)(d) = \fun(c_\beta)(d)$, and that $\fun(c_\alpha)(d_\beta) = e_\beta$.

For $\alpha = \delta < \lambda$ limit, let $\tc$ be a $\tlf$-upper bound given by treetops. Now for each
$\epsilon < \alpha$ we define
\[ k^\alpha_\epsilon = \max \{ n: n_{\epsilon + 1} \leq n \leq \dom(\tc) ~ \mbox{and}~ \fun(\tc\rstr_n)(d_\epsilon)=e_\epsilon \} \]
Thus for all $\epsilon < \alpha$ and all $\beta < \alpha$, $k^\alpha_\epsilon > n_\beta$, while for fixed $\alpha$
and increasing $\epsilon < \alpha$ the $k^\alpha_\epsilon$ form a descending sequence. In other words,
$( \langle n_\beta : \beta < \alpha \rangle, \langle k^\alpha_\epsilon : \epsilon < \alpha \rangle)$ represent a pre-cut
in $X_\mb$. However, $\lcf(\cf(\alpha), \cs) \geq \lambda$ so necessarily it is a pre-cut and not a cut; we may fin
$k_\epsilon$ such that for all $\epsilon < \alpha$ and all $\beta < \alpha$, $n_\beta < k_\epsilon < k^\alpha_\epsilon$.
Let $\tc_\alpha = \tc \rstr_{k_\epsilon}$. This completes the inductive construction of the sequence.

\br

Thus $\langle (\tc_\alpha, n_\alpha) : \alpha < \lambda \rangle$ is well defined. As we assumed $\xt_\cs > \lambda$, we have
$\lambda^+$-treetops so we may choose $c$ to be an upper bound in $\mct$ for $\langle \tc_\alpha : \alpha < \lambda$.
For each $\epsilon < \lambda$, let $k_\epsilon$ be as given in the previous paragraph.

By definition $\langle n_\alpha : \alpha < \lambda \rangle$ is an increasing sequence in $X_\mb$.
By assumption in the statement of the Claim, we are in the case
where $\xp_\cs < \xt_\cs$ \emph{thus} by
Reduction \ref{z10} $\lcf(\lambda) = \kappa = \lambda^{-}$. Thus for some
$\langle n^*_i : i < \kappa \rangle$ we have that $(\langle n_\alpha : \alpha < \lambda \rangle, \langle n^*_i : i < \kappa \rangle)$
represents a cut in $X_\mb$. For each $i < \kappa$, let
\[ Y_i = \{ \epsilon < \lambda : n^*_i < k_\epsilon\} \]
Now, each $Y_i \in \mci$, since this is witnessed by $\tc\rstr_{n^*_i}$ by what we have shown.  
Moreover, $\{ Y_i : i < \kappa \}$ is an increasing sequence of subsets of $\lambda$ whose union is $\lambda$.
So we have presented $\lambda$ as the union of $\kappa$ elements of $\mci$. 

Having proved each of these properties, we may therefore assume $\lambda$ is the union of $\kappa$ sets from $I$, and that
the ideal $\mci$ is $\lambda$-complete. Thus $\lambda \in \mci$. As $\lambda \in \mci$ shows the existence of the desired bijection, this completes the proof.
\end{proof}

\begin{cor} \label{c:tfeq-c}
Let $\de$ be a regular ultrafilter on $\lambda$ which $\lambda^+$-saturates the theory of the random graph.
If $\de$ has $\lambda^+$-treetops then $\de$ $\lambda^+$-saturates $\tfeq$.
\end{cor}

\begin{proof}
As $\de$ is a regular ultrafilter, we may choose any $M$ with enough set theory for trees and
consider the ultrapower as a cofinality spectrum problem $(M, M^\lambda/\de, \dots)$.
By assumption, $\de$ has $\lambda^+$-treetops \emph{thus} $\xt_\cs > \lambda$. There
are two cases. If $\xt_\cs = \xp_\cs > \lambda$, then by definition of $\xp_\cs$
(\ref{cst:card} above) and Corollary \ref{empty-good}, $\de$ is $\lambda^+$-good.
Thus necessarily $\de$ saturates $\tfeq$.

Otherwise, $\xt_\cs > \xp_\cs$. It will suffice by Corollary \ref{m-cor} to show that
bijections exist.
By Reduction \ref{z10}, the case $\xt_\cs > \xp_\cs$ necessarily entails that $\lambda$ is the successor
of a regular cardinal $\kappa$ and that $\xt_\cs > \xp_\cs = \lambda = \kappa^+$.
Thus, for any $\ma \in \ord(\cs)$ and any suitably chosen sequences
$\{ d_\epsilon : \epsilon < \lambda \}$, $\{ e_\epsilon : \epsilon < \lambda \}$ of
elements of $X_\ma$, we may apply Lemma \ref{c:tfeq} to obtain a suitable bijection.
We conclude by Theorem \ref{tfeq-m} or just Corollary \ref{m-cor} that $\de$ saturates $\tfeq$, as desired.
\end{proof}

We now have the ingredients to prove Theorem \ref{t:tfeq}.

\begin{proof}[Proof of Theorem \ref{t:tfeq}.]

Let $\lambda \geq \aleph_0$ and let $\de$ be a regular ultrafilter on $\lambda$.
Suppose that $\de$ saturates some non-simple theory $T_*$. It suffices to show that
$\de$ necessarily also saturates $\tfeq$.
There are two cases which, by Fact \ref{tp2-sop2}, cover all possibilities.

\step{Case 1}.
$T_*$ has $TP_2$. Then $\de$ saturates $\tfeq$ by Theorem \ref{tfeq-m}.

\step{Case 2}.
$T_*$ has $SOP_2$. By Fact \ref{rg-good}, $\de$ saturates $\trg$. 
By Conclusion \ref{sop2-concl}, $\de$ has $\lambda^+$-treetops.
So we may apply Claim \ref{c:tfeq-c} to conclude $\de$ saturates $\tfeq$.

\br \noindent This completes the proof of Theorem \ref{t:tfeq}.
\end{proof}

\br

\section{$\xp=\xt$} \label{s:p-t}

In this section we apply the main theorem of Section \ref{s:main-thm} to prove Theorem \ref{theorem:p-t}.

\begin{defn} \emph{(see e.g. van Douwen \cite{douwen}, Vaughan \cite{vaughan}, Blass \cite{blass})}
We define several properties which may hold of a family $D \subseteq [\mathbb{N}]^{\aleph_0}$. 
Let $A \subseteq^* B$ mean that $\{ x : x \in A, ~x \notin B \}$ is finite.

\begin{itemize}
\item $D$ has a \emph{pseudo-intersection} if there is an infinite $A \subseteq \mathbb{N}$ such that
for all $B \in D$, $A \subseteq^* B$.
\item $D$ has the s.f.i.p. (strong finite intersection property)
if every nonempty finite subfamily has infinite intersection.
\item $D$ is
called a tower if it is well ordered by ${\supseteq^*}$ and has no infinite pseudo-intersection.
\item $D$ is called open if it is closed under almost subsets, and dense if every $A \in [\mathbb{N}]^{\aleph_0}$
has a subset in $D$.
\end{itemize}

We then define:
\begin{align*}
\xp = & \min \{ |\eff| ~ : ~ \eff \subseteq [\mathbb{N}]^{\aleph_0}~ \mbox{has the s.f.i.p. but has no infinite pseudo-intersection} \} \\
\xt = & \min \{ |\mct| ~ : ~ \mct \subseteq [\mathbb{N}]^{\aleph_0} ~\mbox{is a tower} \} \\
\xb = & \min \{ |B| ~ : ~ B \subseteq {^\omega \omega} ~\mbox{is unbounded in $({^\omega \omega}, \leq_*)$} \} \\
\xh = & \mbox{ the smallest number of dense open families with empty intersection}  
\end{align*}
\end{defn}

\vspace{1mm}

\begin{fact} \label{x:fact} $\xp$ and $\xt$ are regular. It is known that
$\xp \leq \xt \leq \xh \leq \xb$.
\end{fact}

\begin{proof} 
Regularity of $\xt$ and the first inequality are clear from the definitions. 
For regularity of $\xp$, due to Szyma\'nski, see e.g. van Douwen \cite{douwen} Theorem 3.1(e). 
The result $\xt \leq \xb$ is due to Rothberger \cite{roth-b}, attributed in \cite{douwen} p. 123, 
and the result $\xt \leq \xh \leq \xb$ is due to Balcar, Pelant and Simon \cite{b-p-s}, attributed in \cite{vaughan} p. 200. 
For completeness: for the second inequality, see e.g. Blass \cite{blass} Prop. 6.8; for the third, \cite{blass} Theorem 6.9. 
\end{proof}

To begin, we look for a relevant cofinality spectrum problem. 

\begin{defn} \label{d:forcing} We fix the following for the remainder of this section.
\begin{enumerate}
\item Let $M = (\mch(\aleph_1), \in )$.
\item Let $\bq = ( [\mathbb{N}]^{\aleph_0}, \supseteq^* )$ be our forcing notion, and $\vv$ a transitive model of ZFC. 
\item Let $\name{\mg}$ be the canonical name of a generic subset of $\bq$ 
$($which is forced to be an ultrafilter on the Boolean algebra $\mcp(\mathbb{N})^{\vv}$$)$. 
\item Let $\mg$ be a generic subset of $\bq$ over $\vv$, which we fix for this section. $($Often, however, 
we will simply work in $\vv$ using the name $\name{\mg}$.$)$
\item For $f \in \vv$, let $\name{f}$ denote the $\bq$-name for $f$.
\end{enumerate}

Define the \emph{generic ultrapower} in the forcing extension $\vv[\mg]$ as follows: 

\begin{enumerate}[resume] 
\item Given $M, \bq$ and $\mg$, by the generic ultrapower $M^\omega/\mg$ in $\vv[\mg]$ we will mean the model $\cn \in \vv[\mg]$ 
with universe $\{ f/\mg : f \in (^\omega M)^\vv \}$, such that 
\begin{itemize} 
\item $\cn \models$ ``$f_1/\mg = f_2/\mg$'' iff $\{ n : f_1(n) = f_2(n) \} \in \mg$ $($this set is necessarily from $\vv$$)$
\item $\cn \models$ ``$(f_1/\mg) \in (f_2/\mg)$'' iff $\{ n : f_1(n) \in f_2(n) \} \in \mg$.
\end{itemize}
We denote by $\jj = \jj_\mg : M \rightarrow \cn$ the map given by $\jj(a) = \langle \dots a \dots \rangle/\mg$.
\end{enumerate}

It will also be useful to refer to these objects in $\vv$. 
\begin{enumerate}[resume]
\item In $\vv$, 
let $\name{\cn}$ be the $\bq$-name of the generic ultrapower $M^\omega/\name{\mg}$, i.e. the model with
\begin{enumerate}
\item universe $\{ {f}/\name{\mg} : f \in (^\omega M)^\vv \}$ such that:
\item $\Vdash_\bq$ ``$(\name{\cn} \models {f_1}/\name{\mg} = f_2/\mg)$ iff 
$\{ n : f_1(n) = f_2(n) \} \in \name{\mg}$''  \\ $($as noted, this set is necessarily from $\vv$$)$
\item $\Vdash_\bq$ ``$(\name{\cn} \models$ $(f_1/\name{\mg}) \in (f_2/\name{\mg}))$ 
iff $\{ n : f_1(n) \in f_2(n) \} \in \name{\mg}$''
\end{enumerate}

\item Note that $\cn = \name{\cn}[\mg]$. 
\end{enumerate}
\end{defn}

By definition of $\xt$, $\bq$ is $\xt$-complete. Thus, 
forcing with $\bq$ adds no new bounded subsets of $\xt$ $($where new means ``$\notin \vv$''$)$ 
and no new sequences of length $<\xt$ of members of $\vv$.
Moreover, by the $\xt$-completeness of $\bq$, and since $\xp \leq \xt$, moving from $\vv$ to $\vv[\mg]$ will not affect whether 
$\xp < \xt$.

We now build a cofinality spectrum problem. We will let $M = M^+ = (\mch(\aleph_1), \in )$ as above and 
$M_1 = M^+_1 = \cn$. For generic ultrapowers, the parallel of \lost theorem holds, so $\jj$ is an elementary 
embedding of $M$ into $\cn$.

\begin{defn} \label{d:psf}
Working in $\vv[\mg]$, let $M, \cn$ be as in $\ref{d:forcing}$. Let 
$\Delta_{\mathrm{psf}}$ be the set of all first-order formulas  
$\vp(x,y,\bar{z})$ in the vocabulary of $M$, i.e. $\{ \in, = \}$, such that if 
$\bar{c} \in {^{\ell(\bar{z})}M}$ then $\vp(x,y,\bar{c})$ is a linear order on the finite set 
$A_{\vp,\bar{c}} = \{ a : M \models \vp(a,a,\bar{c}) \}$, denoted by $\leq_{\vp,\bar{c}}$.  
We require $\ell(x) = \ell(y)$ but do not require $\ell(x) = 1$.
\end{defn}

\begin{obs} \label{o:x}
If $M, \cn$ are as above and $\vp \in \Delta_{\mathrm{psf}}$, then: 
\begin{enumerate}
\item[(a)] for each $\overline{c} \in {^{\ell(\overline{y})}\cn}$,
$\vp(x,y;\overline{c})$ is a discrete linear order on the set
\[ \{ a : \cn \models \vp(a,a;\overline{c}) \} \]
\item[(b)] each nonempty $\cn$-definable subset of $A_{\vp, \bar{c}}$ has a first and last element.
\item[(c)] we may in $\cn$ identify $(A_{\vp, \bar{c}}, \leq_{\vp, \bar{c}})$ with a definable subset of some 
\[ \langle (X_n, \leq_n) : n < \omega \rangle /\mg \]
where each $X_n$ is finite and linearly ordered by $\leq_n$.
\end{enumerate}
\end{obs}

\begin{proof}
\lost theorem.
\end{proof}

\begin{claim} \label{q-cor}
Working in $\vv[\mg]$, $(M, \cn, Th(M), \Delta_{\mathrm{psf}})$ is a cofinality spectrum problem which, for the remainder of this section, we call $\cs$.
\end{claim}

\begin{proof}
We check Definition \ref{d:estt}. Conditions (1)-(2) are immediate. For (3), let $\Delta = \Delta_{\mathrm{psf}}$, so the only other 
data we need to specify is in each instance to set $d_\ma$ to be the maximum element. (4) is by construction, see \ref{o:x}.    
Suppose we are given $\ma$, $\mb$ with $X_\ma = \{ a : \cn \models \vp_1(a,a,\bar{c_1})\}$ and $X_\mb = \{ a: \cn \models \vp_2(a,a,\bar{c_2}) \}$, 
where $\vp_1, \vp_2$ are from \ref{d:psf}.
Let $\theta(x_1 y_1 x_2 y_2, \bar{c_1}, \bar{c_2})$ be the formula which implies $(x_1, y_1), (x_2, y_2) \in X_\ma \times X_\mb$ and 
gives the G\"odel pairing function, i.e we order $(x_1, y_1)$, $(x_2, y_2)$
first by maximum, then by first coordinate, then by second coordinate. Clearly $\theta(x_1y_1, x_2y_2, z_1z_2) \in \Delta_{psf}$.
This gives (5) and (6). It remains to check we have trees. 
Let $\ma$ be given, and suppose $\vp_\ma = \vp(x,y,z)$. Let $\psi(w,z)$ be such that for each $\bar{c} \in {^{\ell(\bar{z})}M}$ 
we have that $\mct_{\vp, \bar{c}} = \{ \eta : M \models \psi(\eta, \bar{c}) \}$ is the set of finite sequences of members of $A_{\vp, \bar{c}} = \{ a : M \models \vp(a,a,\bar{c}) \}$ 
of length $< \max A_{\vp, \bar{c}}$. Let $\tlf = \{ (\eta, \nu) : \eta, \nu \in \mct_{\vp, \bar{c}}$ and $\eta$ is an initial segment of $\nu$ $\}$. 
We can likewise define the functions $\lgn$ and $\xr$. Clearly by \lost theorem these objects will behave as desired in $\cn$. 
This completes the proof of (7), and (8) is immediate, so we are done. 
\end{proof}

\br
\noindent It follows from the definitions that $\xp \leq \xt$. If $\xp = \xt$ then Theorem \ref{theorem:p-t} is immediately true. So
we shall assume, towards a contradiction, that $\xp < \xt$ in $\vv$.

\begin{claim} \label{c:x1} Working in $\vv[\mg]$, 
let $\cs$ be the cofinality spectrum problem from $\ref{q-cor}$. Then $\xt \leq \xt_\cs$.

That is, let $\ma \in \ord(\cs)$ be given, so $\cn \models$ ``$(\mct_\ma, \tlf_\ma)$ is a tree of finite sequences of
$(X_{\ma}, \leq_{\ma})$''. Then any increasing sequence of cofinality $\kappa < \xt$ in $(\mct_\ma, \tlf_\ma)^\cn$ has an
upper bound.
\end{claim}

\begin{rmk}
In fact $\xt = \xt_\cs$, but this will not be needed.
\end{rmk}

\begin{proof}[Proof of Claim \ref{c:x1}.]

First, a reduction: Recalling the definition of $\cs$ in $\ref{q-cor}$ and \ref{o:x}(c), we may assume 
$(X_\ma, <_\ma, \mct_\ma, \tlf_\ma) = \langle (X_{\ma_n}, <_{\ma_n}, \mct_n, \tlf_{\mct_n}) : n < \omega \rangle /\mg$.
So without loss of generality for each $n < \omega$ $(X_{\ma_n}, <_{\ma_n}, \mct_n, \tlf_{\mct_n})$ is standard, i.e.
$X_{\ma_n}$ is finite and $\mct_{n}$ is the set of finite sequences of elements of $X_{\ma_n}$, partially ordered by inclusion.

For each $n<\omega$, there is an isomorphism $h_n : (X_{\ma_n}, <_{\ma_n}) \rightarrow (k_n, <_{k_n})$ where $k_n \in \omega$, 
$<_k$ is the usual order on $\omega$ restricted to $k$, and $\lim_\mg \langle k_n : n < \omega \rangle$ is infinite. 
Then in $\cn$, $h = \langle h_n : n < \omega \rangle/\mg$ gives an isomorphism between
$(\mct_\ma, \tlf_\ma)$ and a definable downward closed subset of $({^{\omega >} \omega}, \tlf)^\cn$. So for the remainder of the proof,
without loss of generality, we work in the tree $({^{\omega >} \omega}, \tlf)^\cn$. This completes the reduction. 

We work now in $\vv$. 
Let $\theta = \cf(\theta) < \xt$ be given and let $B \in \bq$ ($B \in \mg$) be such that:
\[ B \Vdash_{\bq} \mbox{``} \langle \name{f}_\alpha/\name{\mg} : \alpha < \theta \rangle ~
\mbox{is $\tlf^{\name{\cn}}$-increasing in } ({^{\omega >} \omega}, \tlf)^{\name{\cn}} \mbox{''} \]
where without loss of generality, $B \Vdash_\bq$ ``$\name{f}_\alpha = f_\alpha$'' for $\alpha < \theta$ since 
forcing with $\bq$ adds no new sequences of length $<\xt$.

By assumption, $\theta < \xt$ thus $\theta < \xb$, the bounding number. So we may choose some increasing function
$g: \mathbb{N} \rightarrow \mathbb{N} \setminus \{0\}$
such that for each $\alpha < \theta$ there is $n_\alpha$ satisfying:
\[ \mbox{if $n \geq n_\alpha$ then } g(n) > \lg(f_\alpha(n)) + \Sigma \{ f_\alpha(n)(j) : j < \lg(f_\alpha)(n) \} \]
Informally, for each $\alpha < \theta$, for all but finitely many $n$, $g(n)$ dominates the sum of all values in the range of $f_\alpha$
when the domain is restricted to $n$.

Now let $\overline{s} = \langle s_n : n < \omega \rangle$ be given by
\[ s_n = {^{g(n) \geq} g(n)} = \{ \eta : \eta ~\mbox{a sequence of length $\leq g(n)$ of numbers $< g(n)$} \} \]

Then
\begin{enumerate}
\item each $s_n$ is a finite nonempty subtree of ${^{\omega > } \omega}$
\item if $\alpha < \theta$ then $(\forall^{\infty} n) (f_\alpha(n) \in s_n)$
\end{enumerate}

We use the following notation:
\[ \left( {^{\omega >} \omega} \right)^{[\nu]} =\{ \eta \in {^{\omega >} \omega} : \nu \trianglelefteq \eta \} \]

Then for each $\alpha < \theta$, we define $Y_\alpha$ as follows: 
\[ Y_\alpha = \bigcup \{ \{ n \} \times \left( s_n \cap \left( {^{\omega >} \omega} \right)^{[f_\alpha(n)]} \right) : n \in B \} \]

Finally, let
\[ Y_* = \bigcup \{ \{n \} \times s_n : n \in B \} \]

Then for each $\alpha < \theta$, we have:
\begin{enumerate}
\item $Y_\alpha \subseteq B \times {^{\omega >} \omega}$
\item $Y_\alpha \cap ( \{n\} \times {^{\omega >} \omega} )$ is finite, and $\subseteq \{ n\} \times s_n$ for every $n$
\item $Y_\alpha$ is infinite
\item $Y_\alpha \subseteq Y_*$
\end{enumerate}

Moreover, if $\alpha < \beta$ then $Y_\beta \subseteq^* Y_\alpha$. 
Why? If $\alpha < \beta$ then $\{ n \in B : f_\alpha(n) \not\trianglelefteq f_\beta(n) \}$ is finite,
as otherwise there is $B^\prime \geq_\bq B$ contradicting 
\[ B \Vdash_\bq \left(\name{\cn} \models ~\mbox{``} f_\alpha/\name{\mg} \tlf f_\beta/\name{\mg} \mbox{''} \right) \]

So as $\xt = \lambda > \theta$, there is an infinite $Z \subseteq Y_*$ such that
$\alpha < \theta$ implies $Z \subseteq^* Y_\alpha$.  Let $B_1 = \{ n \in B : Z \cap ( \{n\} \times s_n ) \neq \emptyset \}$.
For each $n \in B_1$ choose $\nu_n \in s_n$ such that $(n, \nu_n) \in Z \cap ( \{ n \} \times s_n )$.
Choose $\nu_n = \langle 0 \rangle$ for $n \in \mathbb{N} \setminus B$.
Then
 \[ B_1 \Vdash_\bq \text{``}\langle \nu_n : n < \omega \rangle/\name{\mg} 
 \text{ is an upper bound for $\langle f_\alpha/\name{\mg} : \alpha < \theta \rangle$ in }({^{\omega >} \omega}, \tlf)^{\name{\cn}}\text{''} \]
This completes the proof.
\end{proof}

\begin{concl} \label{c:x5} 
Working in $\vv[\mg]$, let $\cs$ be the cofinality spectrum problem from \ref{q-cor}. Then $\mc(\cs, \xt) = \emptyset$.
\end{concl}

\begin{proof}
By Theorem \ref{no-cuts}, $\mc(\cs, \xt_\cs) = \emptyset$, and by Claim \ref{c:x1}, $\xt \leq \xt_\cs$.
\end{proof}

\begin{disc}
These results do not contradict the existence of Hausdorff gaps, see e.g. Definition $2.26$ of \cite{todorcevic}. 
This is because to obtain treetops 
and the transfer of the peculiar cut below, we restrict ourselves to some infinite subset of $\omega$. 
\end{disc}

Aiming for a contradiction, we will leverage Conclusion \ref{c:x5} against a cut existence result from Shelah \cite{Sh:885}, which we prove can
be translated to our context.

\begin{defn} \emph{(Peculiar cut, \cite{Sh:885} Definition 1.10)} \label{p-cut}
Let $\kappa_1, \kappa_2$ be infinite regular cardinals. A $(\kappa_1, \kappa_2)$-peculiar cut in ${^\omega \omega}$ is a pair
$(\langle g_i : i < \kappa_2 \rangle, \langle f_i : i < \kappa_1 \rangle)$ of sequences of functions in ${^\omega \omega}$
such that:
\begin{enumerate}
\item $(\forall i < j < \kappa_2) (g_i \lls g_j)$
\item $(\forall i < j < \kappa_1) (f_j \lls f_i)$
\item $(\forall i < \kappa_1)(\forall j < \kappa_2) (g_j \lls f_i)$
\item if $f: \omega \rightarrow \omega$ is such that
$(\forall i < \kappa_1)(f \lls f_i)$, then $f \lls g_j$ for some
$j < \kappa_2$
\item if $f: \omega \rightarrow \omega$ is such that
$(\forall j < \kappa_2)(g_j \lls f)$, then $f_i \lls f$ for some
$i < \kappa_1$
\end{enumerate}
\end{defn}

\begin{thm-lit} \label{t:885} \emph{(Shelah \cite{Sh:885} Theorem 1.12)}
Assume $\xp < \xt$. Then for some regular cardinal $\kappa$ there exists a $(\kappa, \xp)$-peculiar cut in
${^\omega \omega}$, where $\aleph_1 \leq \kappa < \xp$.
\end{thm-lit}

We include here a definition which we plan to investigate in a future paper (it is not used in the main line of our proof here).

\begin{disc} \label{disc:lcf}
We note here that for a $D$ a filter on $\mathbb{N}$ $($if $D$ is the cofinite filter we may omit it$)$, 
one may also define  
\begin{enumerate} 
\item for $\delta_1, \delta_2$ limit ordinals, normally regular cardinals, 
we say $(\overline{f}^1, \overline{f}^2)$ is a $D$-weakly peculiar $(\delta_1, \delta_2)$-cut when:
\begin{enumerate}
\item $\overline{f}^1 = \langle f^1_\alpha : \alpha < \delta_1 \rangle$ is $<_D$-increasing
\item $\overline{f}^2 = \langle f^2_\beta : \beta < \delta_2 \rangle$ is $<_D$-decreasing
\item $f^1_\alpha <_D f^2_\beta$ if $\alpha < \delta_1, \beta < \delta_2$
\item for no $A \in D$ and $f \in {^\omega \omega}$ do we have 
\[ \alpha < \delta_1 \land \beta < \delta_2 \mbox{ implies } f^1_\alpha <_\de f <_\de f^2_\beta \]
\end{enumerate}
\item we define
\[ \mc(D, {^\omega \omega} ) = \{ (\kappa_1, \kappa_2) : ~\mbox{there is a $D$-weakly peculiar $(\kappa_1, \kappa_2)$-cut, 
$\kappa_1, \kappa_2$ regular} \} \]
\end{enumerate}

In a work in preparation we will show that the lower cofinality exists for $\kappa < \xp$ in e.g. $\mc({^\omega \omega})$.
\end{disc}

We now connect the cut existence from Theorem \ref{t:885} with existence of a cut in the generic ultrapower $\cn$.

\begin{claim} \label{c:x0}
Working in $\vv[\mg]$, suppose $\xp < \xt$ and let $\cn, \cs$ be as above. 
Then for some regular $\kappa_1$ with $\aleph_1 \leq \kappa_1 < \xp$, we have that
$(\kappa_1, \xp) \in \mc(\cs, \xt)$. 
\end{claim}

\begin{rmk}
We use Claim \ref{c:x0} in the case where $\kappa_1 < \xp = \kappa_2$, but it holds generally.
\end{rmk}

\begin{proof}[Proof of Claim \ref{c:x0}.] 
Let $\cn$ be as above. 
We will prove that $\cn$ has a $(\kappa_1, \kappa_2)$-cut for some regular $\kappa_1, \kappa_2$ with $\aleph_1 \leq \kappa_1 < \kappa_2 =\xp$. 
The construction will show that this cut is in a pseudofinite linear order of $\cn$, in the sense of \ref{d:psf} above.  
As we have assumed 
that $\xp < \xt$, this will suffice to prove that $(\kappa_1, \xp) \in \mc(\cs, \xt)$.  

Let $(\langle g_i : i < \kappa_2 \rangle, \langle f_i : i < \kappa_1 \rangle)$ be as in Theorem \ref{t:885}, i.e.
a $(\kappa_1, \kappa_2)$-peculiar cut in ${^\omega \omega}$.  
Working in $\vv[\mg]$, consider in $\cn$ the set
\[ I = \prod_{n < \omega} [0, f_0(n)]/\mg \]
with the usual linear order, i.e. the order $< = <_I$ induced on the generic ultrapower by the factors. 
Note that the product is in $\vv$, though $\mg$ is not. Clearly $I$ is pseudofinite in $\cn$. (In other words, recalling \ref{d:psf} and \ref{q-cor}, 
there is a nontrivial $\ma = (I, <_I, \dots) = (X_\ma, <_\ma, \dots) \in \ord(\cs)$.) Moreover,
\begin{itemize}
\item$i < j < \kappa_1$ implies $g_i/\mg, g_j/\mg \in I = X_\ma$ and $g_i/\mg <_\ma g_j/\mg$
\item$i < j < \kappa_2$ implies $f_i/\mg, f_j/\mg \in I = X_\ma$ and $f_i/\mg <_\ma f_j/\mg$
\item$i < \kappa_2$, $j < \kappa_1$ implies $g_i/\mg <_\ma g_j/\mg$
\end{itemize}

So $(\langle g_i/\mg : i < \kappa_2 \rangle, \langle f_i/\mg : i < \kappa_1 \rangle)$ represents a pre-cut in $X_\ma$
and it will suffice to show that it represents a cut.

We carry out the remainder of the proof in $\vv$. Assume for a contradiction that the conclusion fails, i.e. our pre-cut 
is not a cut. Then this failure is forced by some $B \in \bq$, $B \in \mg$. That is, for some 
$h \in (^\omega \omega)^\vv$, $B \Vdash_\bq$ ``$g_i/\name{\mg} < h/\name{\mg} < f_j/\name{\mg}$'' for $i<\kappa_2, j<\kappa_1$.
Then $B$ is infinite, and $i < \kappa_2$ implies that $\{ n : g_i(n) < h(n) \} \supseteq^* B$, as otherwise (recalling the definition
of $\bq$) there is $B_1 \geq_\bq B$, $B_1 \Vdash g_i/\name{\mg} \geq h/\name{\mg}$.
Likewise, $j < \kappa_1$ implies that $\{ n : h(n) < f_j(n) \} \supseteq^* B$. 
This contradicts Definition \ref{p-cut} and so completes the proof.
\end{proof}

\begin{concl} \label{cor:x1} In Claim \ref{c:x0} we have shown that if 
$\xp < \xt$ $($in $\vv$$)$ then:
\begin{enumerate}
\item $\Vdash_\bq$ ``$\name{\cn}$ has a $(\kappa_1, \kappa_2)$-cut for some $\kappa_1 < \kappa_2 = \xp$''.
\item In $\vv[\mg]$ for some regular $\kappa_1, \kappa_2$ with $\aleph_1 \leq \kappa_1 < \kappa_2 = \xp$, 
$(\kappa_1, \kappa_2) \in \mc(\cs, \xt)$, \emph{thus} $(\kappa_1, \kappa_2) \in \mc(\cs, \xt_\cs)$.  
In particular, $\mc(\cs, \xt_\cs) \neq \emptyset$. 
\end{enumerate}
\end{concl}

\begin{proof}
So there is no confusion about the assumption, recall by $\ref{q-cor}$ that 
$\xp^\vv < \xt^\vv$ iff $\xp^{\vv[\mg]} < \xt^{\vv[\mg]}$.  Then both (1) and (2) are immediate from \ref{c:x0}, 
noting in the case of (2) that $\xt \leq \xt_\cs$ by \ref{c:x1}. 
\end{proof}

\noindent We now prove Theorem \ref{theorem:p-t}.

\begin{theorem} \label{theorem:p-t}
$\xp = \xt$. 
\end{theorem}

\begin{proof} 
It follows from the definitions that $\xp \leq \xt$. 
Suppose, in $\vv$, that $\xp < \xt$.  
Working now in $\vv[\mg]$, let $\cs$ be the cofinality spectrum problem from \ref{q-cor}.
By Conclusion \ref{c:x5}, which does not assume $\xp < \xt$, $\mc(\cs, \xt_\cs) = \emptyset$. 
By Conclusion \ref{cor:x1}(2), which does assume $\xp < \xt$, $\mc(\cs, \xt_\cs) \neq \emptyset$, a contradiction. 
So necessarily $\xp = \xt$, which completes the proof.
\end{proof}



\end{document}